\documentclass[reqno,11pt]{amsart}
\usepackage[paper=letterpaper,margin=1in]{geometry}
\usepackage{amsmath,amssymb,mathrsfs,bm,amsthm,amsfonts,comment,csquotes}
\usepackage{graphicx}
\usepackage{soul}
\usepackage[usenames,dvipsnames]{xcolor}
\usepackage{amssymb}
\usepackage{hyperref}
 
\hypersetup{%
	colorlinks=true, linkcolor=blue,
	citecolor=ForestGreen
}
\usepackage{enumitem}
\usepackage{bm}
\usepackage{marginnote}
\usepackage{stmaryrd}
\usepackage{mathabx} 
\usepackage{bbm}
\usepackage{caption}
\usepackage{tikz}
\usetikzlibrary{matrix,graphs,arrows,positioning, shapes, calc,decorations.markings,decorations.pathmorphing,shapes.symbols,hobby}

\newtheorem{theorem}{Theorem}[section]
\newtheorem{lemma}[theorem]{Lemma}
\newtheorem{assumption}[theorem]{Assumption}
\newtheorem{proposition}[theorem]{Proposition}
\newtheorem{corollary}[theorem]{Corollary}
\theoremstyle{definition}
\newtheorem{definition}[theorem]{Definition}
\newtheorem{example}[theorem]{Example}
\newtheorem{observation}[theorem]{Observation}
\newtheorem{remark}[theorem]{Remark}
\numberwithin{equation}{section}
\allowdisplaybreaks
\usepackage{acronym} 
\newcommand{\rev}{\textcolor{black}}

\renewcommand{\Pr}{\mathbb{P}}	
\newcommand{\Ex}{\mathbb{E}}	
\newcommand{\ind}{\bm{1}}	
\newcommand{\set}[1]{{\{#1\}}}	
\newcommand{\tr}{\mathrm{tr}} 	
\newcommand{\vh}{V_1}
\newcommand{\vm}{V_2}
\newcommand{\vl}{V_3}
\newcommand{\lc}{\llbracket}
\newcommand{\rc}{\rrbracket}
\newcommand{\ds}{\displaystyle}
\newcommand{\eg}{|E(G_n)|}

\newcommand{\norm}[1]{\Vert#1\Vert}

\newcommand{\R}{\mathbb{R}} 
\newcommand{\Z}{\mathbb{Z}} 

\newcommand{\vr}{\operatorname{Var}}
\newcommand{\e}{\varepsilon}

\newcommand{\ceil}[1]{\lceil #1 \rceil} 
\renewcommand{\i}{\bm{i}}

\usepackage{graphicx}

\title[Fluctuations of Quadratic Chaos]{Fluctuations of Quadratic Chaos}

\author[B.\ B.\ Bhattacharya]{Bhaswar B.~Bhattacharya}
\address{B.\ B.\ Bhattacharya,
	Department of Statistics and Data Science, University of Pennsylvania.
}
\email{bhaswar@wharton.upenn.edu}
\author[S.\ Das]{Sayan Das}
\address{S.\ Das,
	Department of Mathematics, Columbia University. 
}
\email{sayan.das@columbia.edu}
\author[S.\ Mukherjee]{Somabha Mukherjee}
\address{S.\ Mukherjee, 
Department of Statistics and Data Science, National Institute of Singapore. 
}
\email{somabha@nus.edu.sg}

\author[S.\ Mukherjee]{Sumit Mukherjee}
\address{S.\ Mukherjee,
	Department of Statistics, Columbia University. 
}
\email{sm3949@columbia.edu}

\begin{document}
	\begin{abstract} 
In this paper we characterize all distributional limits of the quadratic chaos $T_n =\sum_{1\le u< v\le n} a_{u, v} X_u X_v$, where $((a_{u, v}))_{1\le u,v\le n}$ is a $\{0, 1\}$-valued symmetric matrix with zeros on the diagonal and $X_1, X_2, \ldots, X_n$ are i.i.d.~ mean $0$ variance $1$ random variables with common distribution function $F$. In particular, we show that any distributional limit of $S_n:=T_n/\sqrt{\vr[T_n]}$ can be expressed as the sum of three independent components: a Gaussian, a (possibly) infinite weighted sum of independent centered chi-squares, and a Gaussian mixture with a random variance. As a consequence, we prove a fourth moment theorem for the asymptotic normality of $S_n$, which applies even when $F$ does not have finite fourth moment. More formally, we show that $S_n$ converges to $N(0, 1)$ if and only if the fourth moment of $S_n$ (appropriately truncated when $F$ does not have finite fourth moment) converges to 3 (the fourth moment of the standard normal distribution). { The proofs combine a Lindeberg-type replacement argument and combinatorial moment calculations using results of Erd\H os and Alon on extremal subgraph counts. } 
\end{abstract}


%
%
	
	\maketitle

	\section{Introduction}

Given a $\{0, 1\}$-valued symmetric matrix $((a_{u, v}))_{1 \leq u, v \leq n}$ with zeros on the diagonal and i.i.d.~mean $0$ and variance $1$ random variables $X_1, X_2, \ldots, X_n$ with common distribution function $F$, consider the $F$-{\it quadratic chaos}
\begin{align}\label{eq:Tn}
	T_n= \sum_{1 \leq u < v \leq n} a_{u, v}X_u X_v . 
\end{align}
In this paper we will study the asymptotic distribution of 
\begin{align}\label{eq:Sn}
	S_n := \frac{T_n}{\sqrt{\vr[T_n]}} =  \frac1{\sqrt{\sum_{1 \leq u < v  \leq n} a_{u, v}}}\sum_{1 \leq u < v \leq n} a_{u, v}X_u X_v,
\end{align}
in the regime where 
\begin{align}\label{eq:dense}
	\vr[T_n]= \sum_{1 \leq u < v  \leq n} a_{u, v} \to \infty.
\end{align} 

\rev{Quadratic forms with $\{0, 1\}$-valued coefficients appear in various contexts (see Remark \ref{r.01coef}) and there are numerous results on asymptotic normality of $S_n$ going back to the classical results of Beran \cite{beran_dependence}, Rotar \cite{rotar}, and de Jong \cite{dejong87}.} Specifically, for $F= N(0, 1)$ and $F$ is the Rademacher distribution, $S_n$ is the well-known quadratic Wiener/Gaussian chaos and the quadratic Rademacher chaos, respectively (see \cite{gaussianlimit,limitrademacher,jansongaussian,rademacher,NoPe09,book_np} and the references therein). {A quadratic chaos is also the first non-trivial component in a polynomial chaos expansion (a multi-linear polynomial of independent random variables) \cite{ghanem2003stochastic}, which play an important role in the study of $U$-statistics \cite{dependenceindependence,dm,jansongaussian,ustatistics}, Boolean functions \cite{de2018gaussian,filmus2019harmonicity,stability}, directed polymers \cite{akq,carcot,polychaos,marrel,csz2d}, random graph coloring problems \cite{bbb_pd_sm,fourthmomenttriangle,fourthmomentH,fourthmomentinfluence}, among several others. 
}

A detailed discussion of the conditions for the asymptotic normality of $S_n$ is given in Section \ref{sec:theorem4}. Broadly speaking, $S_n$ has a Gaussian limit when the dependence between the collection of random variables $\{a_{u, v}X_u X_v\}_{1 \leq u < v \leq n}$ is either local and/or weak. The following is a simple example of this: 
\begin{enumerate}
	\item[(1)] Assume $n=2 L$ and take 
	\begin{align*}
		a_{v, u}=a_{u, v}=\begin{cases}
			1 & \text{ if } u=2 \ell -1 \mbox{ and }v=2 \ell,  \mbox{ for } 1\le \ell \le L,  \\
			0 & \mbox{otherwise.}
		\end{cases}
	\end{align*} 
	Then it is easy to show that, for example, by a direct application of Stein's method based on dependency graphs (cf.~\cite[Theorem 2.7]{dependencygraph}),  
	$$S_n:=\frac1{\sqrt{L}}\sum_{\ell=1}^L X_{2\ell-1}X_{2 \ell} \stackrel{D}{\to} N(0,1).$$ 
\end{enumerate} 
\noindent Examples where $S_n$ converges to a weighted sum of independent centered chi-squared random variables are also well-known (cf.~\cite[Chapter 3]{ustatistics}). This usually  happens when the matrix $((a_{u, v}))_{1 \leq u, v, \leq n}$ is `dense', that is, it has a positive fraction of non-zero elements, as in the example below: 
\begin{enumerate}
	\item[(2)] Take $a_{u, v}=1$, for all $1 \leq u \ne v \leq n$. Then 		\begin{align}
		S_n & := \frac{1}{\sqrt{{n \choose 2}}} \sum_{1\le u<v\le n} X_uX_v \nonumber \\
		\label{eq:Kn}			 & =\frac12 \cdot \frac{1}{\sqrt{{n \choose 2}}} \left(\sum_{u=1}^n X_u \right)^2-\frac12 \cdot \frac{1}{\sqrt{{n \choose 2}}} \sum_{u=1}^n X_u^2 	\nonumber \\ 
		& \stackrel{D} \rightarrow \frac{1}{\sqrt 2} (\chi^2_1-1) , 	
	\end{align}
	since $\frac{1}{n} \left(\sum_{u=1}^n X_u \right)^2 \stackrel{D} \rightarrow \chi_1^2$ and	$\frac{1}{n} \sum_{u=1}^n X_u^2 \stackrel{P} \rightarrow 1$. 
\end{enumerate}
Perhaps more interestingly, there are also examples where the limit of $S_n$ is neither a Gaussian nor a weighted sum of chi-squares (or a combination of both). 
\begin{enumerate}
	
	\item[(3)] Take $a_{1, v}=a_{v, 1}=1$ for all $2 \leq v \leq n$ and $a_{u, v}=a_{v, u}=0$ otherwise. Then 
	\begin{align*}
		S_n = \frac1{\sqrt{n-1}}X_1 \sum_{v=2}^n X_v \stackrel{D}{\to} X_1\cdot Z, 
	\end{align*} 
	where $Z \sim {N}(0,1)$ is independent of $X_1$. Note that $X_1\cdot Z \stackrel{D} = N(0, X_1^2)$, which is a normal distribution with random variance $X_1^2$.\footnote{Given a non-negative random variable $A$, we denote by $Z \sim N(0, A)$ the normal distribution with variance $A$ ({\it normal variance mixture}). More precisely, $Z$ is a random variable with characteristic function $\phi_Z(t):=\Ex[{e^{-\frac{1}{2}A^2t^2}}]$, where the expectation is taken over the randomness of $A$. } This limit is non-Gaussian whenever $X_1$ is not the Rademacher distribution. Moreover, this limit, unlike those in the previous two examples, is `non-universal', that is, it depends on the distribution $F$.  
\end{enumerate}

Despite being a quantity of fundamental interest, it appears that the regime where $S_n$ neither has a Gaussian nor a chi-squared-type limit has not been systematically explored. In fact, the different limits obtained in the examples above raise the following natural question: {\it What are the class of all possible limiting distributions of the $F$-quadratic chaos $S_n$ in the regime} \eqref{eq:dense}? 
In this paper we answer this question by proving a general decomposition theorem which allows us to express the limiting distribution of $S_n$ as the sum of three independent components: a Gaussian, a (possibly) infinite weighted sum of independent $\chi^2_1-1$ random variables, and a normal variance mixture, where the random variance is a (possibly) infinite quadratic form in the variables $\{X_u\}_{u \geq 1}$ (Theorem \ref{thm:main}). Moreover, we show that any distributional limit of $S_n$ must be of the aforementioned form (Theorem \ref{thm:converse}), thus identifying all possible limiting distributions of $S_n$. As a consequence, we obtain a necessary and sufficient condition for the asymptotic normality of $S_n$. In particular, we show in Theorem \ref{thm:moment4} that $S_n$ converges to $N(0, 1)$ if and only if the fourth-moment of $S_n$ (appropriately truncated when $\Ex[X_1^4] = \infty$) converges to 3 (the fourth-moment of $N(0, 1)$). The key idea in the proofs is to decompose the matrix $A=((a_{u, v}))_{1 \leq u, v \leq n}$ into parts, such that contributions to $S_n$ from the corresponding parts are either asymptotically negligible or mutually  independent. For this we use estimates from extremal combinatorics \cite{alon81,hypergraphcopies} to bound various moments of $S_n$ and a Lindeberg-type argument for replacing $F$ with the standard Gaussian distribution (in the relevant parts of $A$), for which the limiting distribution can be explicitly computed. 
The formal statements of the results are given below. 


\subsection{Limiting Distribution of $F$-Quadratic Chaos}
\label{sec:theoremquadratic}

Hereafter, we will adopt the language of graph theory and think of the matrix $((a_{u, v}))_{1\le u, v \le n}$ as an adjacency matrix of a graph on $n$ vertices. We begin with the following definition: 

\begin{definition}\label{defg} We denote by $\mathscr{G}_n$ the space of all simple undirected graphs on $n$ vertices labeled by $\lc n \rc:=\{1,2,\ldots,n\}$, where the vertices are labeled in non-increasing order of the degrees $d_1\ge d_2 \ge \cdots \ge d_n$, where $d_v$ denotes the degree of the vertex labeled $v$.  
\end{definition}

For a graph $G_n\in \mathscr{G}_n$, we denote the adjacency matrix by $A(G_n) = (a_{u,v})_{u,v \in V(G_n)}$, the vertex set by $V(G_n)= \lc n \rc $, and the edge set by $E(G_n)$. Then the quadratic form in \eqref{eq:Tn} can be re-written (in terms of the adjacency matrix of $G_n$) as follows: 
\begin{align}\label{def:TG_n} 
	T_{G_n}(\bm{X}_n):=\sum_{1\le u < v \le n}a_{u, v} X_u X_v =\frac12\bm{X}_n^{\top}A(G_n)\bm{X}_n,
\end{align}
where $\bm{X}_n:=(X_1,X_2,\ldots,X_n)^{\top}$. Note that $\Ex[T_{G_n}(\bm{X}_n)] =0$ and $\vr[T_{G_n}(\bm{X}_n)]=|E(G_n)|$. Throughout, we will assume that $\{G_n\}_{n\ge 1}$ is a sequence of graphs with $G_n \in \mathscr{G}_n$ and $\eg\to\infty$ (recall \eqref{eq:dense}). Then the rescaled quadratic form \eqref{eq:Sn} can be re-written as: 
\begin{align}\label{def:SG_n}
	S_{G_n}(\bm{X}_n) := \frac{T_{G_n}(\bm{X}_n)}{\sqrt{|E(G_n)|}}\stackrel{\eqref{def:TG_n}}=\frac1{2\sqrt{|E(G_n)|}}\bm{X}_n^{\top}A(G_n)\bm{X}_n. 
\end{align}

\begin{remark}\label{r.01coef} \rev{The statistic $T_{G_n}(\bm{X}_n)/S_{G_n}(\bm X_n)$ is a prototypical example of a degenerate  $U$-statistic of order 2 \cite{dependenceindependence,jansongaussian,ustatistics} which arises in various contexts, for example, the Hamiltonian of the Ising model on $G_n$ \cite{bmpotts,chatterjee}, non-parametric two-sample tests based on geometric graphs \cite{fr}, testing independence in auto-regressive models \cite{beran_dependence}, random graph-coloring problems \cite{bbb_pd_sm,xiao}, and Ramsey theory \cite{anitconcentrationramsey}. The related statistic where $F$ is the Bernoulli distribution is also of interest, due its connections to the birthday problem \cite{diaconismosteller,dasguptasurvey} and motif counting  \cite{kw1}, and has been studied recently in \cite{bbb_sm_sm}.} 
\end{remark}

To describe the limiting distribution of $S_{G_n}(\bm{X}_n)$ we assume the following two conditions on the graph sequence $\{G_n\}_{n \geq 1}$:

\begin{assumption}[Co-degree condition] \label{asum1} {\em The graph sequence $\{G_n\}_{n\ge 1}$ will be said to satisfy the {\it $\Sigma$-co-degree condition} if there exists an infinite dimensional matrix $\Sigma=((\sigma_{st}))_{s, t\ge 1}$ such that for each fixed $s, t\ge 1$,
		\begin{align}\label{as:point}
			\lim_{n \rightarrow \infty}\frac1{\eg}\sum_{v =1}^n a_{s, v}a_{v, t} = \sigma_{st}.
	\end{align}}
\end{assumption}

\begin{assumption}[Spectral Condition] \label{asum2}  {\em Fix $K \geq 0$ and let $G_{n,K}$ be the subgraph of $G_n$ induced by the vertex set $[K+1,n]$. Denote by $A_{n, K}$ the adjacency matrix of $G_{n, K}$.  Then the graph sequence $\{G_n\}_{n\ge 1}$ will be said to satisfy the {\it $\bm{\rho}$-spectral condition}, if there exists a non-negative sequence $\bm \rho = ( \rho_1, \rho_2, \ldots )$ such that for every fixed $s \geq 1$, 
		\begin{align}\label{as:eigen}
			\lim_{K\to\infty}\lim_{n\to\infty} \frac{1}{\sqrt{\eg}} \lambda_{n, K}^{(s)} = \rho_s,  
		\end{align} 
		where  ${|\lambda_{n, K}^{(1)}| \ge |\lambda_{n, K}^{(2)}| \ge \cdots \geq |\lambda_{n, K}^{(n-K)}|}$ are the eigenvalues of $A_{n, K}$. } 
\end{assumption}

Note that $\sum_{v =1}^n a_{s, v}a_{v, t}$ is the {\it co-degree} (the number of common neighbors) of the vertices $u$ and $v$. Since the vertices of $G_n$ are arranged in non-increasing order of the degrees, Assumption \ref{asum1} means  that the scaled co-degrees between pairs of `high'-degree vertices in $G_n$ have a limit. On the other hand, Assumption \ref{asum2} ensures that the edge of the spectrum (properly scaled) of the graph obtained from $G_n$ by removing the `high'-degree vertices have a limit. 

With the above assumptions we are now ready to state the main result of the paper.


\begin{theorem}\label{thm:main} 
	Let $\{X_u\}_{u \geq 1}$ be a collection of i.i.d.~mean $0$ and variance $1$ random variables with common distribution function $F$. Suppose there exists an infinite dimensional matrix $\Sigma=((\sigma_{st}))_{s, t \ge 1}$ and vector $\bm \rho = ( \rho_1, \rho_2, \ldots )$ such that the graph sequence $\{G_n\}_{n\ge 1}$, with $G_n \in \mathscr G_n$, 
	satisfies the $\Sigma$-codegree condition and the $\bm \rho$-spectral condition as in Assumption \ref{asum1} and \ref{asum2}, respectively. Then for 
	$\bm{X}_n=(X_1,X_2,\ldots,X_n)^\top$ and $S_{G_n}(\bm{X}_n)$ as defined in \eqref{def:SG_n} the following hold: 
	\begin{align}\label{eq:main}
		S_{G_n}(\bm{X}_n)\stackrel{D}{\to} Q:=Q_1+Q_2+Q_3,
	\end{align}
	where $Q_1, Q_2,$ and $Q_3$ are independent random variables with \begin{itemize}
		\item[--] $Q_1\sim N(0,\bm{X}_{\infty}^{\top}\Sigma \bm{X}_{\infty})$ where $\bm{X}_{\infty}=(X_1,X_2,\ldots)^{\top}$ and $\Sigma$ is defined in Assumption \ref{asum1}.
		\item[--] $Q_2 \sim N(0,\rho^2)$ where $\rho^2 :=1-\sum_{s=1}^{\infty} (\sigma_{ss}+\frac12\rho_s^2)$, 
		\item[--] $Q_3 \sim \frac12 \sum_{s=1}^\infty \rho_s Y_s$ where $\{Y_s\}_{s \geq 1}$ are i.i.d.~$\chi_1^2-1$ random variables and $\{ \rho_s\}_{s \geq 1}$ are defined in \eqref{as:eigen}. 
	\end{itemize} 
\end{theorem}

\begin{remark} Note that the distribution of $Q_1$ in \eqref{eq:main} is a {\it normal variance mixture}, where the random variance $\bm{X}_{\infty}^{\top}\Sigma \bm{X}_{\infty}=\sum_{1\le s, t \le \infty} \sigma_{s, t} X_s X_t$ is {well-defined under Assumption \ref{asum1}} (see Lemma \ref{mim}). Moreover, the distribution of $Q_3$, which is an infinite weighted sum of centered $\chi^2_1$ random variables, is also {well-defined under Assumption \ref{asum2}} (see Proposition \ref{p:lim2}). 
\end{remark}

The proof of Theorem \ref{thm:main} is given in Section \ref{sec:pfquadratic}. The proof proceeds by partitioning the vertices of $G_n$ into three components based on their degrees, which we refer to as the `high'-degree, `medium'-degree, and `low'-degree vertices, respectively (see \eqref{eq:V123} for the precise definition). This partitions the edge set of $G_n$ into 6 parts, namely, the high-to-high (\texttt{hh}), high-to-medium (\texttt{hm}), high-to-low (\texttt{hl}), medium-to-medium (\texttt{mm}), medium-to-low (\texttt{ml}), and low-to-low (\texttt{ll}) edges (see Figure \ref{fig:graphpf}). The proof then involves analyzing the contributions from these 6 components. (A detailed overview of the proof outline is given in Section \ref{sec:pfquadratic}.) In particular, we show the following: 

\begin{itemize}
	
	\item The contributions from the \texttt{hh} and \texttt{hm} edges are asymptotically negligible. 
	
	\item 
	The joint contribution from the \texttt{mm}, \texttt{ml}, and \texttt{ll} edges converges to $Q_2+Q_3$, where $Q_2$ and $Q_3$ are as defined in Theorem \ref{thm:main}. Observe that this limit is {\it universal}, that is, it does not depend on the distribution $F$. 
	
	\item The contribution from the \texttt{hl} edges is asymptotically independent from the rest and converges to the  normal variance mixture $Q_1$. Note that here the limit is {\it non-universal} because the (random) variance of $Q_1$ depends on the distribution of $F$.  For the proof of the asymptotic independence we use estimates from extremal combinatorics \cite{alon81,hypergraphcopies} to bound the number of copies of various subgraphs of $G_n$ which arise in the moments of $S_{G_n}$. 
\end{itemize}

\begin{remark}\label{remark:normalqf}  A particular case that has classically studied is the quadratic Wiener chaos, that is, when $F=N(0, 1)$ is the standard normal distribution (see \cite{gaussianlimit,multidimensionaldistribution,quadraticnormallimit}). In this case, using the spectral theorem, it follows that any distributional limit of $S_{G_n}$ is of the form $Q_2'+Q_3'$ with $(Q_2', Q_3')$ mutually independent, where $Q_2'$ is Gaussian and $Q_3'$ is an infinite weighted sum of centered $\chi_1^2$ random variables. On the other hand, Theorem \ref{thm:main} implies that $S_{G_n} \stackrel{D} \rightarrow Q_1+Q_2+Q_3$ 
	with $(Q_1, Q_2, Q_3)$ mutually independent, where $Q_1\sim N(0,{\bm X}_\infty^\top\Sigma{\bm X}_\infty)$, $Q_2$ is Gaussian, and $Q_3$ is an infinite weighted sum of centered $\chi_1^2$ random variables. 
	This apparent dichotomy can be explained by observing that in the Gaussian case $$ N(0,{\bm X}_\infty^\top\Sigma{\bm X}_\infty)\stackrel{D}{=}\sum_{s=1}^\infty \eta_s Y_s,$$
	where $\{Y_s\}_{s\ge 1}$ are i.i.d. $\chi_1^2-1$ and for some sequence $\{\eta_s\}_{s \geq 1}$ (see Lemma \ref{lem:normalqf}). Consequently, in the Gaussian case $Q_1+Q_3$ is an infinite weighted sum of centered $\chi_1^2$ random variables. 
\end{remark}

%
%

Given the above discussion, it is natural to wonder whether Assumptions \ref{asum1} and \ref{asum2} are
necessary for the distributional convergence of $S_{G_n}(\bm{X}_n)$. More generally, one can ask what are the possible limiting distributions of $S_{G_n}(\bm{X}_n)$? We answer this question is the following theorem: 

\begin{theorem}\label{thm:converse} Let  $\bm{X}_n=(X_1,X_2,\ldots,X_n)^\top$ be i.i.d.~mean $0$ and variance $1$ random variables with common distribution function $F$ and consider a sequence of graphs $G_n\in \mathscr{G}_n$, where $\mathscr{G}_n$ is as in Definition \ref{defg}. If the rescaled quadratic form $S_{G_n}(\bm{X}_n)$, defined in \eqref{def:SG_n}, converges weakly to some random variable $Q'$, then $Q'$ must be of the form $Q$ defined in \eqref{eq:main}.
\end{theorem}

The result above shows that Theorem \ref{thm:main} indeed characterizes all possible distributional limits of $S_{G_n}(\bm{X}_n)$. 
The proof of this result entails showing that Assumptions \ref{asum1} and \ref{asum2} always hold along a subsequence, and then we invoke Theorem \ref{thm:main} along the subsequence to characterize the limit.

\subsection{Characterizing Normality: The Fourth Moment Phenomenon} 
\label{sec:theorem4} 

Theorem \ref{thm:main} can be used to characterize when the limiting distribution of $S_{G_n}(\bm{X}_n)$ is asymptotically Gaussian. 
To this end, for $M > 0$ let $$a_{M}:=\Ex [X_1\ind\{|X_1|\le M\}] \quad \text{ and } \quad b_M:=\vr[X_1\ind\{|X_1|\le M\}].$$ Note that $a_M \to 0$ and $b_M \to 1$. Define \begin{align}\label{def:trunc}
	X_{u, M}:=b_M^{-\frac{1}{2}}(X_u \ind \{|X_u|\le M\}-a_M ) , 
\end{align} 
for $1 \leq u \leq M$. Observe that for all large enough $M$ we have $b_M>0$ and hence, 
$\bm{X}_{n, M} := (X_{1, M},\ldots,X_{n, M} )^\top$ is well defined and consists of i.i.d.~ mean 0, variance 1 \textit{bounded} random variables. Therefore, without loss of generality we will hereafter assume that $M$ is large enough. Now, we have the following result: 

\begin{theorem}\label{thm:moment4} Let  $\bm{X}_n=(X_1,X_2,\ldots,X_n)^\top$ be i.i.d.~mean $0$ and variance $1$ random variables with common distribution function $F$ and consider a sequence of graphs $\{G_n\}_{n \geq 1}$ with $G_n \in \mathscr{G}_n$. 
	
	\begin{itemize} 
		
		\item[$(1)$] Then 
		\begin{align*}
			S_{G_n}(\bm{X}_n)\stackrel{D}{\to} N(0,1) \ \mbox{ if and only if } \ \lim_{M \rightarrow \infty }\lim_{n\to\infty} \Ex[(S_{G_n}(\bm{X}_{n, M}))^4]=3, 
		\end{align*} 
		where $S_{G_n}$ is as in \eqref{def:SG_n} and $\bm{X}_{n, M}$ is defined via \eqref{def:trunc}. 	
		
		\item[$(2)$] If further $\mathbb E[X_1^4] < \infty$, then
		\begin{align*}
			S_{G_n}(\bm{X}_n)\stackrel{D}{\to} N(0,1) \ \mbox{ if and only if } \ \lim_{n\to\infty} \Ex[(S_{G_n}(\bm{X}_n))^4]=3.
		\end{align*}

	\end{itemize} 
	
\end{theorem}

The above result shows that the asymptotical normality of the random quadratic form $S_{G_n}(\bm{X}_n)$ is characterized by a {\it fourth-moment phenomenon}. In particular, if  $F$ has finite fourth moment, then  $S_{G_n}(\bm{X}_n)$ is asymptotically standard Gaussian if and only if its fourth-moment converges to 3 (the fourth moment of the standard Gaussian). Furthermore, if $F$ does not have finite fourth moment, then the limiting distribution of $S_{G_n}(\bm{X}_n)$ is asymptotically Gaussian if and only if the fourth-moment of the quadratic form evaluated on the truncated variables $\{X_{u, M}\}_{1 \leq u \leq N}$ converges to 3, that is, the asymptotical normality is characterized by a {\it truncated fourth-moment phenomenon}.

\begin{remark}\label{remark:4moments}
	The fourth moment phenomenon was first discovered by Nualart and Peccati \cite{NuPe05}, who showed that the convergence of the first, second, and fourth moments to $0, 1$, and $3$, respectively, guarantees asymptotic normality for a sequence of multiple stochastic Wiener-It\^o integrals of fixed order. 
	Later, Nourdin and Peccati \cite{NoPe09,moment4_np} provided error bounds for the fourth moment theorem of \cite{NuPe05}. Thereafter, this emerged as a 
	ubiquitous principle governing the central limit theorems for various non-linear functionals of random fields.  
	We refer the reader to the book \cite{book_np} for an introduction to the topic and  website \url{https://sites.google.com/site/malliavinstein/home} for a list of the recent results. 
	Related results for degenerate $U$-statistics of a fixed order were first obtained by de Jong \cite{dejong87,de90}. Here, in addition to the fourth moment condition, in general, an extra condition is  needed to control the maximum influence of the underlying independent random variables (cf.~\cite[Theorem 2.1]{dejong87}, \cite[Theorem 1]{de90}, and also \cite[Theorem 1.6]{DoKr19}). 
	For other classical sufficient conditions for asymptotic normality and rates of convergence of quadratic forms see \cite{chatterjee_normal_approximation,gotze,gotze_applications,hall,rotar} and the references therein. 
\end{remark}

Adapting the aforementioned results to the specific case of the quadratic chaos (for general symmetric matrices $((a_{u, v}))_{1 \leq u, v \leq n}$), implies that the fourth-moment phenomenon is sufficient for asymptotic normality when $F$ is such that $\Ex_{X \sim F}[X^4]\geq 3$ \cite{greater_three} and $F$ is the Rademacher distribution  \cite{rademacher}.   In Theorem \ref{thm:moment4} we provide a complete characterization of the asymptotic normality of $F$-quadratic chaos when $((a_{u, v}))_{1 \leq u, v \leq n}$ is the adjacency matrix of a graph. In particular, Theorem \ref{thm:moment4}~(2) shows that the convergence of the fourth-moment characterizes the asymptotic normality of $S_{G_n}(\bm{X}_n)$ whenever $F$ has finite fourth-moment. This means that the fourth-moment phenomenon holds even in the intermediate regime $1 < \Ex_{X \sim F}[X^4] < 3$, which, to the best of our knowledge, is not covered by previous results. Our result also includes the case where $F$ does not have finite fourth-moment, where the asymptotic normality is characterized by a truncated fourth moment phenomenon (Theorem \ref{thm:moment4}~{(1)}). 

Theorem \ref{thm:moment4} is a consequence of the following proposition,  which we prove in Section \ref{sec:pfmoment4}. In addition to establishing the fourth-moment phenomenon, this proposition provides a structural characterization of graphs which satisfy the fourth-moment condition. Interestingly, the characterization depends on whether or not the common distribution $F$ of $(X_1, X_2, \ldots, X_n)$ is Rademacher. 

\begin{proposition}\label{ppn:nrad} Let  $\bm{X}_n=(X_1,X_2,\ldots,X_n)^\top$ be i.i.d.~mean $0$ and variance $1$ random variables with common distribution function $F$ and consider a sequence of graphs $\{G_n\}_{n \geq 1}$ with $G_n \in \mathscr{G}_n$. Then the following  hold: 
	
	\begin{itemize}
		
		\item[$(1)$]  If $F$ is not the Rademacher distribution, then the following are equivalent: 
		
		\begin{enumerate}[label=(\alph*)]
			\item \label{ba} For  
			$|\lambda_{1,n}|\ge |\lambda_{2,n}|\ge  \cdots \geq |\lambda_{n,n}|$ the eigenvalues of the adjacency matrix $A(G_n)$, 
			\begin{align}\label{eq:spectralF}
				\lim_{n \rightarrow \infty}  \frac{1}{\sqrt{\eg}} \max_{1\le u \le n} |\lambda_{u,n}| = 0 . 
			\end{align} 
			\item \label{eq:4m} $\lim_{M\to \infty} \lim_{n\to\infty} \Ex[(S_{G_n}(\bm{X}_{n, M}))^4] = 3$, where $\bm{X}_{n}^{(M)}$ is as defined in \eqref{def:trunc}. 
			
			\item \label{bc} $S_{G_n}(\bm{X}_n) \stackrel{D} \to N(0,1)$. 
		\end{enumerate}  
		Moreover, if $\Ex[X_1^4]<\infty$, then the above conditions are equivalent to $\lim_{n\to\infty} \Ex[(S_{G_n}(\bm{X}_{n}))^4] = 3$.

		\item[$(2)$]  If $F$ is the Rademacher distribution, then the following are equivalent: 
		\begin{enumerate}[label=(\alph*)]
			\item \label{bar} For $C_4$ denoting the 4-cycle and $N(C_4,G_n)$  the number of copies of $C_4$ in $G_n$, 
			\begin{align}\label{eq:C4}
				\lim_{n \rightarrow \infty} \frac{1}{\eg^{2}} N(C_4,G_n) = 0.
			\end{align}
			
			\item \label{eq:4mr} $ \lim_{n\to\infty} \Ex[(S_{G_n}(\bm{X}_n))^4] = 3$. 
			\item \label{bcr} $S_{G_n}(\bm{X}_n) \stackrel{D}{\to} N(0,1)$.
		\end{enumerate}

	\end{itemize} 
	
\end{proposition}

Proposition \ref{ppn:nrad} provides a useful way to verify the fourth-moment condition in examples. Incidentally, the fact that the 4-cycle condition \eqref{eq:C4} characterizes Gaussianity for the quadratic Rademacher chaos (Proposition \ref{ppn:nrad}~(2)) also follows from \cite[Theorem 1.3]{bbb_pd_sm}, where the statistic $S_{G_n}$ was studied in the context of graph coloring problems. 
However, when $F$ is not the Rademacher distribution, the asymptotic normality of $S_{G_n}$ is characterized by the spectral condition \eqref{eq:spectralF} instead. To prove Proposition \ref{ppn:nrad} we express the fourth moment of $S_{G_n}$ as a linear combination of the counts of the different multi-graphs with 4 edges (see Figure \ref{fig:momentgraphs}). Then it can be observed that when $F$ is not the Rademacher distribution both the 2-star and the 4-cycle  counts contribute to the leading order of the fourth-moment difference $\Ex[(S_{G_n}(\bm{X}_{n, M}))^4]-3$. This, in turn, can be expressed in terms of the sum of the fourth-powers of the eigenvalues of $G_n$ (see \eqref{eq:4con}), which leads to the spectral condition in \eqref{eq:spectralF}. On the other hand, if $F$ {is the} Rademacher distribution, the coefficient corresponding to the 2-star count vanishes (since $\vr[X_1^2]=0$ in \eqref{eq:nneg0}) and the leading order of the fourth-moment difference $\Ex[(S_{G_n}(\bm{X}_{n}))^4]-3$ is determined solely by the number of 4-cycles.  

\subsection{Asymptotic Notation} \label{sec:asymptotic_notation} Throughout we will use the following standard asymptotic notations. For two positive sequences $\{a_n\}_{n\geq 1}$ and $\{b_n\}_{n\geq 1}$, $a_n  \lesssim b_n$ means $a_n \leq C_1 b_n$, $a_n  \gtrsim b_n$ means $a_n \geq C_2 b_n$, and $a_n \asymp b_n$ means $C_2 b_n \leq a_n \leq C_1 b_n$, for all $n$ large enough and positive constants $C_1, C_2$. Moreover, subscripts in the above notation,  for example $\lesssim_\square$ or $\gtrsim_\square$,  denote that the hidden constants may depend on the subscripted parameters. Finally, for a sequence of random variables $\{X_n\}_{n \geq 1}$ and a positive sequence $\{a_n\}_{n \geq 1}$, the notation $X_n = O_P(a_n)$ means $X_n / a_n$ is stochastically bounded, that is, $\lim_{M \rightarrow \infty} \lim_{n \rightarrow \infty}\Pr( | X_n / a_n | \leq M ) = 1$, and $X_n = o_P(a_n)$ will mean $\lim_{n \rightarrow \infty}\Pr( | X_n / a_n | \geq \varepsilon ) = 0$, for every $\varepsilon > 0$.

\subsection*{Organization} The rest of the paper is organized as follows. In Section \ref{sec:example} we compute the limiting distribution in various examples. The proofs of Theorem \ref{thm:main} and Theorem \ref{thm:converse} are given in Section \ref{sec:pfquadratic}. Theorem \ref{thm:moment4} and Proposition \ref{ppn:nrad} are proved in Section \ref{sec:pfmoment4}. The universality of the limiting distribution is discussed in Section \ref{sec:universality}. \rev{In Section \ref{sec:open}, we discuss open problems and directions for future research.} A few technical lemmas are proved in Appendix \ref{sec:appendix}. 

\section{Examples} \label{sec:example}

In this section we apply Theorem \ref{thm:main} to obtain limiting distribution of $S_{G_n}$ for various graph ensembles. Recall that the precise condition on $G_n$ under which $S_{G_n}$ is asymptotically normal is given in Proposition \ref{ppn:nrad}. Consequently, in this section we will primarily focus on the situations where the limit of $S_{G_n}$ has a non-normal component.

\begin{example}[Dense Graphs] 
	
	In \eqref{eq:Kn} we constructed an example where the limit $S_{G_n}$ converges to a centered $\chi^2_1$ distribution. Note that in this example $a_{u, v} = 1$ for all $1 \leq u \ne v \leq N$, that is, $G_n$ is the complete graph on $n$ vertices. This phenomenon extends to any converging sequence of dense graphs. In order to  explain this we briefly recall the basic definitions about the convergence of graph sequences (see \cite{lov} for a detailed exposition). 
	
	For two graphs $F$ and $G$, define the homomorphism density of $F$ into $G$ as 
	$$t(F,G) :=\frac{|\hom(F,G)|}{|V (G)|^{|V (F)|},}$$
	where  $|\hom(F,G)|$ denotes the number of homomorphisms of $F$ into $G$. 
	A {\it graphon} is a measurable function from $[0, 1]^2$ into $[0, 1]$ that is symmetric $W(x, y) = W(y,x)$, for all $x, y \in [0, 1]$. This is the continuum  analogue of graphs, and we denote the space of all graphons by $\mathscr W$. A finite simple graph $G$ on the vertex set $\{1, 2, \ldots, n\}$ can also be represented as a graphon in a natural way: Define $W_G(x, y) :=\boldsymbol 1\{(\ceil{nx}, \ceil{ny})\in E(G)\}$, that is, partition $[0, 1]^2$ into $n^2$ squares of side length $1/n$, and define $W_G(x, y)=1$ in the $(u, v)$-th square if $(u, v)\in E(G)$ and 0 otherwise. The graphon $W_G$ will be referred to as the {\it empirical graphon} corresponding to the graph $G$. For a simple graph $F$ with $V (F)= \{1, 2, \ldots, |V(F)|\}$ and a graphon $W$, define 
	$$t(F,W) =\int_{[0,1]^{|V(F)|}}\prod_{(u,v)\in E(F)} W(x_u,x_v) \mathrm dx_1\mathrm dx_2\cdots \mathrm dx_{|V(F)|}.$$ The fundamental definition of graph limit theory \cite{graph_limits_I,graph_limits_II,lov} asserts that a sequence of graphs $\{G_n\}_{n\geq 1}$ {\it converge to a graphon $W \in \mathscr W$} if for every finite simple graph $F$, 
	\begin{equation}\label{eq:graph_limit}
		\lim_{n\rightarrow \infty} t(F, G_n) = t(F, W).
	\end{equation} 
	Furthermore, every function $W\in \mathscr W$ defines an operator $T_W: L^2[0, 1]\rightarrow L^2[0, 1]$, by 
	\begin{eqnarray}
		(T_Wf)(x)=\int_0^1W(x, y)f(y)\mathrm dy.
		\label{eq:TW}
	\end{eqnarray}
	$T_W$ is a Hilbert-Schmidt operator, which is compact and has a discrete spectrum, that is, a countable multiset of non-zero real eigenvalues $\{\lambda_s(W)\}_{s \in \mathbb N}$ (see \cite[Section 7.5]{lov}). In particular, every non-zero eigenvalue has finite multiplicity and $
	\sum_{s=1}^\infty\lambda_s^2(W) = \int_{[0, 1]^2}W(x, y)^2\mathrm dx\mathrm dy$.

	Suppose $\{G_n\}_{n \geq 1}$, with $G_n = (V(G_n), E(G_n))$ and $G_n \in \mathscr{G}_n$, is a sequence of graphs converging to a graphon $W$ as in \eqref{eq:graph_limit} such that $t(K_2, W) = \int_{[0, 1]^2} W(x, y) \mathrm dx \mathrm dy  > 0$. Note that 
	$$\frac{1}{\eg} \sum_{v=1}^n a_{s, v}a_{v, t} \leq \frac{n}{|E(G_n)|} \to 0,$$ since $t(K_2, G_n)=\frac{2 \eg}{n^2}\to t(K_2,W)>0$ implies $\eg \gtrsim n^2$.  Thus, $G_n$ satisfies \eqref{as:point} with $\sigma_{st}=0$ for all $s, t$. Next, observe that for any fixed $K$, if $G_n\backslash G_{n, K}$ denotes the graph obtained by removing the edges in $G_{n, K}$  from $G_n$ (recall the definition of $G_{n, K}$ from Assumption \ref{asum2}), then 
	$\frac{1}{n^2}|E(G_n\backslash G_{n, K})| \leq \frac{K }{n}.$
	This implies, as $n \rightarrow \infty$, 
	$$\int_{[0, 1]^2}|W_{G_n}(x, y) - W_{G_{n, K}} (x, y) | \mathrm d x \mathrm d y \rightarrow 0,$$
	that is, the $L^1$ distance between the empirical graphons $W_{G_{n, K}}$ and $W_{G_n}$ converges to zero, for any fixed $K$. Hence, by \cite[Equation (8.14)]{lov}, for any fixed $K$ the truncated graph $G_{n,K}$ also converges to the graphon $W$.   Then, recalling that $|\lambda_{n, K}^{(1)}| \ge |\lambda_{n, K}^{(2)}| \ge \cdots \geq |\lambda_{n, K}^{(n-K)}|$ are the eigenvalues of the adjacency matrix of $G_{n, K}$, by \cite[Theorem 11.53]{lov} $\frac1{n}\lambda_{n,K}^{(s)} \to \lambda_s(W)$,   
	for every $s \geq 1$ fixed, where $\{\lambda_s(W)\}_{s \geq 1}$ are the eigenvalues (ordered according to non-increasing absolute values) of the operator $T_W$ as defined in \eqref{eq:TW}. 
	Thus, \eqref{as:eigen} holds with $$\rho_s=\sqrt{\frac{2}{\int_{[0, 1]^2} W(x, y) \mathrm dx \mathrm dy}}\lambda_s(W) . $$
	Thus limiting distribution $Q$ in \eqref{eq:main} is of the form
	\begin{align}\label{graphon}
		Q \sim N\left(0,1-\frac{\int_{[0, 1]^2} W(x, y)^2 \mathrm dx \mathrm dy}{\int_{[0, 1]^2} W(x, y) \mathrm dx \mathrm dy}\right)+\sum_{s=1}^{\infty} \frac1{\sqrt{2\int_{[0, 1]^2} W(x, y) \mathrm dx \mathrm dy}}\lambda_s(W) \cdot Y_s,
	\end{align}
	where $\{Y_s\}_{s \geq 1}$ are a collection of independent $\chi_1^2 - 1$ random variables independent of the normal random variable. In the following we compute the limit in \eqref{graphon} for a few specific choices of $W$. 
	
	\begin{itemize}
		
		\item  {\it Dense  Erd\H os-R\'enyi Random Graphs}: The Erd\H os-R\'enyi random graph $\mathcal{G}(n,p)$ is a random graph on $n$ vertices where each edge is present independently with probability $p = p(n) \in (0,1]$. When $p \in(0, 1]$ is fixed, then $G_n \sim \mathcal{G}(n,p)$ converges in probability to the graphon $W_p \equiv p$, which is the constant function $p$. In this case, $\lambda_1(W_p) = p$ is the only non-zero eigenvalue of $W_p$ and 
		\begin{align}\label{er2}
			Q\sim N(0,1-p)+\sqrt{\tfrac{p}{2}} (\chi_1^2-1),
		\end{align}
		where the $\chi_1^2-1$ variable is independent of the $N(0,1-p)$ variable. In particular, if $p=1$, which corresponds to the complete graph $K_n$, the normal component in \eqref{er2} is degenerate and $Q \sim \frac{1}{\sqrt 2} (\chi_1^2-1)$ (which recovers the limit in \eqref{eq:Kn}).

		\item  {\it Stochastic Block Models}: Consider the graphon corresponding to the 2-block stochastic block model with equal block sizes, such that the  within-block probability is $p$ and the across-block probability is $q$:	
		$$W(x,y)=\begin{cases}p & x , y \in [0, \frac{1}{2}]^2 \cup [\frac{1}{2}, 1]^2 , \\ q & \mbox{otherwise}. \end{cases}$$
		Since  $\frac{p+q}{2}$ and $\frac{p-q}2$ are the only non-zero eigenvalues of $W$, following \eqref{graphon}, the limiting distribution is given by
		\begin{align*}
			Q \sim N\left(0,1-\tfrac{p^2+q^2}{p+q}\right)+ \tfrac{\sqrt{p+q}}{2} \cdot Y_1+\tfrac{p-q}{2\sqrt{p+q}} \cdot Y_2, 
		\end{align*} 
		where $Y_1, Y_2$ are independent $\chi^2_1-1$ random variables which are independent of the normal component. 
	\end{itemize}
\end{example}


Next, we consider the case where the limit is a normal variance mixture. This arises in the limit when the graph $G_n$ has a few `high' degree vertices. 
Towards this we consider the complete bipartite graph. 

\begin{example}[Complete bipartite graphs] Suppose $G_n=K_{a, n}$ is the complete bipartite graph with vertex set $V(K_{a, n}) = A \cup B$ such that $|A|=a$ and $|B|=n$. In this case, the limiting distribution of $S_{G_n}$ depends on whether whether $a$ is fixed or increasing with $n$. 
	
	\begin{itemize}
		
		\item {\it Suppose $a$ is fixed.} Note that $|E(K_{a, n})|=a n$ and $d_{s, t}:=\sum_{v=1}^n a_{s, v} a_{v, t} = n$, for $s, t \in A$ and $d_{s, t} = a$, for $s, t \in B$. Hence, \eqref{as:point} holds with $\sigma_{st}=\frac{1}{a}$, for $1 \leq s, t \le a$ and $0$ otherwise. Moreover, when $K>a$, then the graph $G_{n, K}$ is empty. Hence, \eqref{as:eigen} holds trivially with $\rho_s=0$, for all $s \geq 1$. Thus, the limiting distribution $Q$ in \eqref{eq:main} is 	\begin{align*}
			Q \sim N\left(0,\frac1 a\left(\sum_{s=1}^a X_s\right)^2\right), 
		\end{align*}
		which is normal variance mixture. Note that the above distribution is exactly a normal, that is, the variance $\frac1 a (\sum_{s=1}^a X_s)^2$ is a constant almost surely, only when $a=1$ and $X_1, X_2, \ldots, X_n$ are i.i.d. Rademacher random variables. This corresponds to choosing $G_n=K_{1, n}$ (the $n$-star) and $F$ the Rademacher distribution. 
		
		\item {\it Suppose $a=a(n) \to \infty$, as $n\to \infty$.} In this case, \eqref{as:point} holds with $\sigma_{st}=0$, for all $s, t$. Moreover, removing the highest $K$ degree vertices from $G_n=K_{a,n}$ gives $G_{n, K}=K_{a-K,n}$. It is well known that the adjacency matrix of the complete bipartite $K_{p, q}$, for $p, q \geq 1$, has only two non-zero eigenvalues $\sqrt{pq}$ and $-\sqrt{pq}$ (see \cite[Section 1.4.2]{eigenvalues_book}). Hence, \eqref{as:eigen} holds with $\rho_1=1$, $\rho_2=-1$ and $\rho_s=0$ for $s \ge 3$. 
		Thus, in this case the limiting distribution $Q$ in \eqref{eq:main} is 
		\begin{align*}
			Q \sim \tfrac12Y_1-\tfrac12Y_2,
		\end{align*}
		where $Y_1$ and $Y_2$ are independent $\chi_1^2 - 1$ random variables. 
	\end{itemize}
\end{example}

Next, we consider the case of the sparse Erd\H os-R\'enyi random graph where the limit turns out to be normal. 

\begin{example}[Sparse Erd\H os-R\'enyi random graphs] \label{errg} Consider the Erd\H os-R\'enyi graph $G_n \sim \mathcal{G}(n, p)$, where $p = p(n) \rightarrow 0$ such that $n^2p \to \infty$. This ensures $\Ex[\eg] \to \infty$ and, consequently $\eg=(1+o_P(1)) \Ex[\eg]$, which ensures that the convergence in \eqref{eq:dense} holds in probability. Let us now verify Assumptions \ref{asum1} and \ref{asum2} for $G_n$. 
	Towards this we claim that
	$$d_{\mathrm{max}} := \max_{1 \leq v \leq n} d_{v} = o(n^2 p)$$ in the regime $p = p(n) \rightarrow 0$ such that $n^2p \to \infty$. Clearly when $np\to \infty$, as $d_{\mathrm{max}} \le n$, we have $d_{\mathrm{max}} =o(n^2p)$. On the other hand, when $np=o(\log n)$ using \cite[Lemma 2.2 (a) and (b)]{BBG} gives $d_{\mathrm{max}} =o(n^2p)$ in this case as well. This implies,
	${\frac{d_{\mathrm{max}}}{\eg} \stackrel{P}{\to} 0}$, and consequently $\sigma_{st}=0$ for all $s, t$. Now, to verify \eqref{as:eigen} let $|\lambda_{n}^{(1)}| \ge |\lambda_{n}^{(2)}| \ge \cdots \ge |\lambda_{n}^{(n)}|$ be the eigenvalues of $G_{n}$.  Then using the eigenvalue interlacing theorem \cite[Corollary 2.5.2]{eigenvalues_book} and the leading order of the maximum eigenvalue in an Erd\H os-R\'enyi graph (see \cite[Theorem 1.1]{ben}), it follows that 
	$$\frac{\max_{1\le s\le n}|\lambda_{n, K}^{(s)}|}{\sqrt{\eg}}   \leq \frac{\max_{1 \leq s \leq n} |\lambda_{n}^{(s)} |}{\sqrt{\eg}} = (1+o_P(1)) \frac{\max\{\sqrt{d_{\max}},np\}}{\sqrt{\eg}}	 \rightarrow 0,$$ 
	since $\frac{d_{\max}}{\eg}\stackrel{P}{\to} 0$ and $\frac{np}{\sqrt{\eg}}\stackrel{P}{\to} 0$ (as $p\to 0$). This shows that \eqref{as:eigen} holds with $\rho_s=0$ for all $s \geq 1$. Thus, by Theorem \ref{thm:main} $$S_{G_n}(\bm{X}_n) \stackrel{D} \rightarrow Q\sim N(0, 1).$$ 	
\end{example}

Next, by combining the three examples above, we can construct a graph where all the three components $Q_1, Q_2,$ and $Q_3$ in \eqref{eq:main} are non-trivial. 

\begin{example}{(Coexistence)} \label{ex3} Fix $p\in (0,1)$. Let $G_n$ be a graph with $n^2+1$ vertices labeled $\{1, 2, \ldots, n^2+1\}$ as follows (see Figure \ref{fig:example123}): 
	
	\begin{itemize} 
		
		\item The vertex labeled 1 is connected to all the other $n^2$ vertices. 
		
		\item On the vertices $2, \ldots, n+1$ we have a realization of the Erd\H os-R\'enyi random graph $\mathcal G(n, \frac{1}{2})$. 
		
		\item On the remaining vertices $\{n+2, \ldots, n^2+1\}$ there is a realization of the Erd\H os-R\'enyi random graph $\mathcal{G}(n^2-n, \frac{1}{n^2})$. 
		
	\end{itemize} 
	\begin{figure}
		\centering
		\includegraphics[scale=0.7]{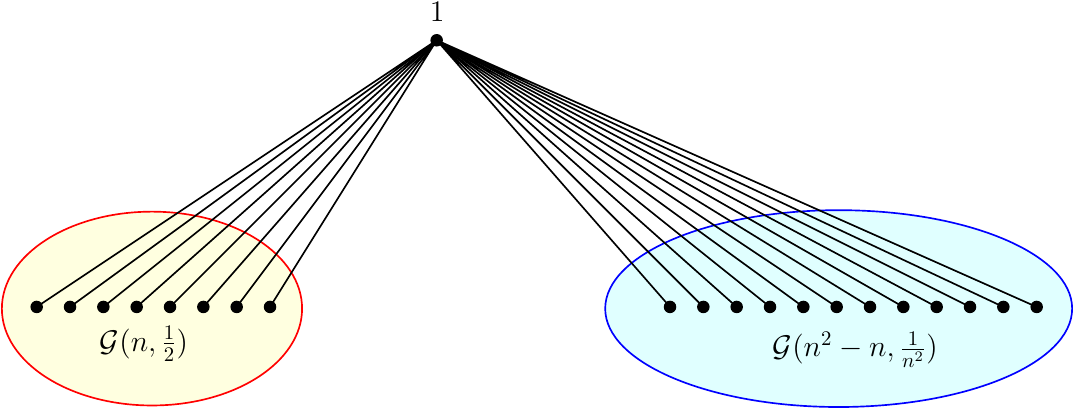}
		\caption{\small{Illustration for Example \ref{ex3}.}}
		\label{fig:example123}
	\end{figure} 
	{We now sketch a proof of the coexistence omitting some of the details for the sake of brevity.} Note that $\eg=\frac74n^2(1+o_P(1))$. In this case, \eqref{as:point} holds with $\sigma_{11}=\frac{4}{7}$ and $\sigma_{st}=0$ for $(s, t)\neq (1,1)$. To check Assumption \ref{asum2}, we remove the vertices $\{1, 2, \ldots, K\}$ from $G_n$ to obtain $G_{n,K}$. Note that after vertex $1$, with high probability, the next $K-1$ top vertices are the $K-1$ top vertices of $\mathcal{G}(n,\frac12)$. Let $H_{n,K}$ be the graph obtained by removing the top $K-1$ vertices from $\mathcal{G}(n,\frac12)$. Thus with high probability, $G_{n,K}$ is a disjoint union of two subgraphs which are isomorphic to  $\mathcal{G}(n^2-n, \frac{1}{n^2})$ and $H_{n,K}$. By \cite[Corollary 1.2]{ben}, the maximum eigenvalue $\lambda_{\mathrm{max}}(\mathcal{G}(n^2-n, \frac{1}{n^2}))= o_P(n)$.  Note that $\lambda_{\mathrm{max}}(\mathcal{G}(n, \frac{1}{2} )) = (1+o_P(1)) \frac{1}{2} n$ and the other eigenvalues of $\mathcal{G}(n, \frac{1}{2} )$ are $o_P(n)$. Since $H_{n,K}$ is a subgraph of $\mathcal{G}(n,\frac12)$, by interlacing of eigenvalues $\lambda_{\mathrm{max}}({H}_{n,K})\le (1+o_P(1))\frac{1}{2}n$ and other eigenvalues of $H_{n,K}$ are $o_p(n)$. Note that the largest eigenvalue of a graph is bounded below by the average degree of the graph. Thus,
	\begin{align*}
		\lambda_{\mathrm{max}}({H}_{n,K})\ge \frac{2}{n-K+1}|E(H_{n,K})| \ge \frac{2}{n-K+1}\left(|E(\mathcal{G}(n,\tfrac12))|- (K-1)(n-1) \right).
	\end{align*} 
	Since $E(\mathcal{G}(n,\frac12))=(1+o_p(1))\frac14n^2$, from the above inequality we have $\lambda_{\mathrm{max}}(H_{n,K})=(1+o_P(1)) \frac{1}{2}n$ and the other eigenvalues of $H_{n,K}$ are $o_P(n)$. Hence, \eqref{as:eigen} holds with $\rho_1= \frac{1}{\sqrt 7}$ and $\rho_s=0$ for $s \ge 2$. 
	Thus, by Theorem \ref{thm:main}, 
	$$S_{G_n}(\bm{X}_n) \stackrel{D} \rightarrow N(0,\tfrac{4}{7}X_1^2)+\tfrac{1}{2 \sqrt{7}} (\chi_1^2 - 1) + N(0,\tfrac{5}{14}),$$ 
	where the three terms above are mutually independent.
\end{example}

Finally, we construct an example where the $\Sigma$ matrix appearing in Assumption \ref{asum1} is infinite. For this we consider a disjoint union of star graphs with growing sizes. 

\begin{example} Suppose $G_n=\bigsqcup_{s=0}^n K_{1, 2^s}$, that is, $G_n$ is the disjoint union of $n$ stars where the $s$-th star has size $2^s$, for $1 \leq s \leq n$. We label the vertices in non-increasing order of their degrees. Then, since $\eg=\sum_{s=1}^n 2^s = 2^{n+1}-2 $, we have 
	$$\sigma_{ss} = 2^{-s}, \text{ for } s \geq 1 \text{ and } \sigma_{st}=0 \text{ for all }  s \neq t.$$ 
	Next, note that the graph $G_{n, K}$ obtained by removing the $K$ highest degree vertices from $G_n$ is still a disjoint unions of stars. In particular, $G_{n, K} = \bigsqcup_{s=1}^{n-K} K_{1, 2^s}$. This implies, the largest eigenvalue $\lambda_{\mathrm{max}}(G_{n, K})= 2^{(n-K)/2}$ and hence, 
	$$\frac{\lambda_{\mathrm{max}}(G_{n, K})}{\sqrt{\eg}} = \frac{2^{(n-K)/2}}{\sqrt{2^{n+1}-2}} = \frac{1}{\sqrt{2^{K+1} - 2^{K+1-n}}}\rightarrow 0,$$ 
	taking $n\to\infty$ followed by $K\to\infty$. This implies, \eqref{as:eigen} holds with $\rho_s=0$ for all $s \ge 1$. Thus, by Theorem \ref{thm:main}, in this case 
	$$S_{G_n}(\bm{X}_n) \stackrel{D} \rightarrow  N\left(0,\sum_{s=1}^{\infty} \frac1{2^s}X_s^2\right). $$
\end{example}

\section{Proofs of Theorem \ref{thm:main} and Theorem \ref{thm:converse}} \label{sec:pfquadratic}

Throughout the proof we will assume that the vertices of the graph $G_n$ are labelled by $ \lc n \rc := \{1, 2, \ldots, n \}$ in non-increasing order of the degrees. The first step in the proof of Theorem \ref{thm:main} is a truncation argument which shows that we can replace the random variables $\bm{X}_{n} = (X_{1}, \ldots, X_{n})^\top$ by their truncated versions (properly centered and scaled) without changing the asymptotic distribution of $S_{G_n}(\bm{X}_n)$. Towards this recall from \eqref{def:trunc} that $a_{M}:=\Ex [X_1\ind\{|X_1|\le M\}]$, $b_M:=\vr[X_1\ind\{|X_1|\le M\}]$, and $\bm{X}_{n, M} = (X_{1, M}, \ldots, X_{n, M})^\top$, where  $$X_{u, M}:=b_M^{-\frac{1}{2}}(X_u \ind \{| X_u |\le M\}-a_M ),$$ 
for $1 \leq u \leq n$. Then with $S_{G_n}$ as defined in \eqref{def:SG_n}, the following lemma shows that the asymptotic distributions of $S_{G_n}(\bm{X}_{n, M})$ and $S_{G_n}(\bm{X}_{n})$ are the same, as $n \rightarrow \infty$ followed by $M \rightarrow \infty$.

\begin{lemma}\label{l:trunc} For every $\e>0$, 
	\begin{align*} 
		\lim_{M\to\infty} \sup_{n\ge 1} \Pr(|S_{G_n}(\bm{X}_n)-S_{G_n}(\bm{X}_{n, M})|>\e) = 0.
	\end{align*}
\end{lemma}

\begin{proof} Note that
	\begin{align}\label{eq:trunc0}
		b_M^{-1}S_{G_n}(\bm{X}_n)-S_{G_n}(\bm{X}_{n, M}) & = \frac{b_M^{-1}}{\sqrt{\eg}}\sum_{1\le u < v \le n} a_{u, v}(X_uX_v-b_M X_{u, M}X_{v,M}) . 
	\end{align} 
	Now, since $X_u-b_M^{\frac{1}{2}}X_{u, M}=X_u\ind\{|X_u|>M\}+a_M$,  
	\begin{align*}
		X_uX_v-b_M X_{u, M}X_{v,M} & = X_uX_v-b_M^{\frac{1}{2}}X_{u, M}X_v+b_M^{\frac{1}{2}}X_{u, M}X_v-b_M X_{u, M}X_{v,M} \\ & = (X_u-b_M^{\frac{1}{2}}X_{u, M})X_v+b_M^{\frac{1}{2}}X_{u, M}(X_v-b_M^{\frac{1}{2}}X_{v,M}) \\ & = (X_u\ind\{|X_u|>M\}+a_M)X_v+b_M^{\frac{1}{2}}X_{u, M}(X_v\ind\{|X_v|>M\}+a_M).
	\end{align*}
	Thus, from \eqref{eq:trunc0}, 
	\begin{align*}
		b_M^{-1}S_{G_n}(\bm{X}_n)-S_{G_n}(\bm{X}_{n, M}) = T_1 + T_2, \nonumber 
	\end{align*}
	where 
	$$T_1 :=  \frac{b_M^{-1}}{\sqrt{\eg}}\sum_{1\le u < v \le n} a_{u, v} (X_u\ind\{|X_u|>M\}+a_M)X_v $$
	and 
	$$T_2 :=  \frac{b_M^{-\frac{1}{2}}}{\sqrt{\eg}}\sum_{1\le u < v \le n} a_{u, v}  X_{u, M}(X_v\ind\{|X_v|>M\}+a_M)$$
	Now, since the collection $\{X_u\}_{1 \leq u \leq n}$ are uncorrelated with mean 0 and variance $1$ and $\Ex[X_1 \bm 1\{ |X_1| > M \}] = - a_M$, 
	\begin{align*}
		\vr\left[ T_1 \right]  =  \frac{1}{b_M^2}\vr [X_1\ind\{|X_1|>M\}]  \quad \text{ and } \quad \vr[T_2]  = \frac{1}{b_M} \vr [X_1\ind\{|X_1|>M\}] .
	\end{align*}
	Thus, by Markov's inequality and using the fact that $\vr[A+B] \le 2\vr[A]+2\vr[B]$ we have
	\begin{align*}
		& \Pr(|S_{G_n}(\bm{X}_n)-S_{G_n}(\bm{X}_{n, M})|>\e) \\  
		& \le \Pr\left(\left|1- \frac{1}{b_M} \right| |S_{G_n}(\bm{X}_n)|>\tfrac\e2 \right)+\Pr\left( \left|b_M^{-1}S_{G_n}(\bm{X}_n)-S_{G_n}(\bm{X}_{n, M}) \right|>\tfrac\e2 \right) \\ 
		& \lesssim \frac{1}{\e^2} \left(1- \frac{1}{b_M}\right)^2 + \frac{1}{\e^2} \left( \frac{1}{b_M^{2}}+\frac{1}{b_M} \right)\vr [X_1\ind\{|X_1|>M\}] .
	\end{align*} 
	This gives the desired result, since $b_M\to 1$ and $\vr[X_1\ind\{|X_1|>M\}] \to 0$ as $M\to \infty$. 
\end{proof}

The above lemma shows that it suffices to derive the limiting distribution of $S_{G_n}(\bm{X}_{n, M})$. One of the advantages of working with $\bm{X}_{n, M}$ is that $S_{G_n}(\bm{X}_{n, M})$ has all finite moments. In fact, the moment generating function of $S_{G_n}(\bm{X}_{n, M})$ exists in an interval containing $0$ (see Lemma \ref{l:mgf}). Fix $M \geq 1$.   The proof now of Theorem \ref{thm:main} now proceeds by partitioning the set  of vertices $\lc n \rc=\{1, 2, \ldots, n\}$ into the following three parts: Towards this, fix $K_1, K_2\ge 1$ and define  
\begin{equation}\label{eq:V123}
	V_1 := \left \lc 1, \left \lfloor K_1 \right \rfloor \right \rc,  ~ V_2 := \left \lc \left \lfloor K_1 \right \rfloor+1, \left \lfloor K_2\sqrt{|E(G_n)|} \right \rfloor \right \rc, ~ \text{and} ~ V_3 := \left \lc \left \lfloor K_2\sqrt{|E(G_n)|} \right \rfloor+1 ,n \right \rc. 
\end{equation} 
Hereafter, for $1 \leq s \leq 3$, we set $n_s:=|V_s|$. 
Note that since the vertices of $G_n$ are arranged in non-increasing order of the degrees, we refer to the vertices in $V_1, V_2, V_3$ as the `high' degree, `medium' degree, and `low' degree vertices, respectively. The adjacency matrix $A(G_n)$ can then be decomposed as a $3\times 3$ block matrix as follows: 
\begin{align}\label{eq:Ablock}
	A(G_n) = \begin{pmatrix}
		A_{11} & A_{12} & A_{13} \\
		A_{21} & A_{22} & A_{23} \\
		A_{31} & A_{32} & A_{33} 
	\end{pmatrix} , 
\end{align}
where $A_{st}$ is a matrix of order $n_s \times n_t$ which encodes the edge connections between the vertex sets $V_s$ and $V_t$ in $G_n$, for $1 \leq s , t \leq 3$. (Recall that in the Introduction we referred to $A_{11}$, $A_{12}$, $A_{13}$, $A_{22}$, $A_{23}$, and $A_{33}$, as the high-to-high (\texttt{hh}), high-to-medium (\texttt{hm}), high-to-low (\texttt{hl}), medium-to-medium (\texttt{mm}), medium-to-low (\texttt{ml}), and low-to-low (\texttt{ll}) edges, respectively.) The proof of Theorem \ref{thm:main} involves deriving the asymptotic contributions of these 6 terms. Before proceeding to the technical details, we provide an outline of the proof below (refer to Figure \ref{fig:graphpf}):

\begin{figure}[h]
	\centering
	\begin{minipage}[l]{1.0\textwidth}
		\centering
		\includegraphics[width=5.25in]
		{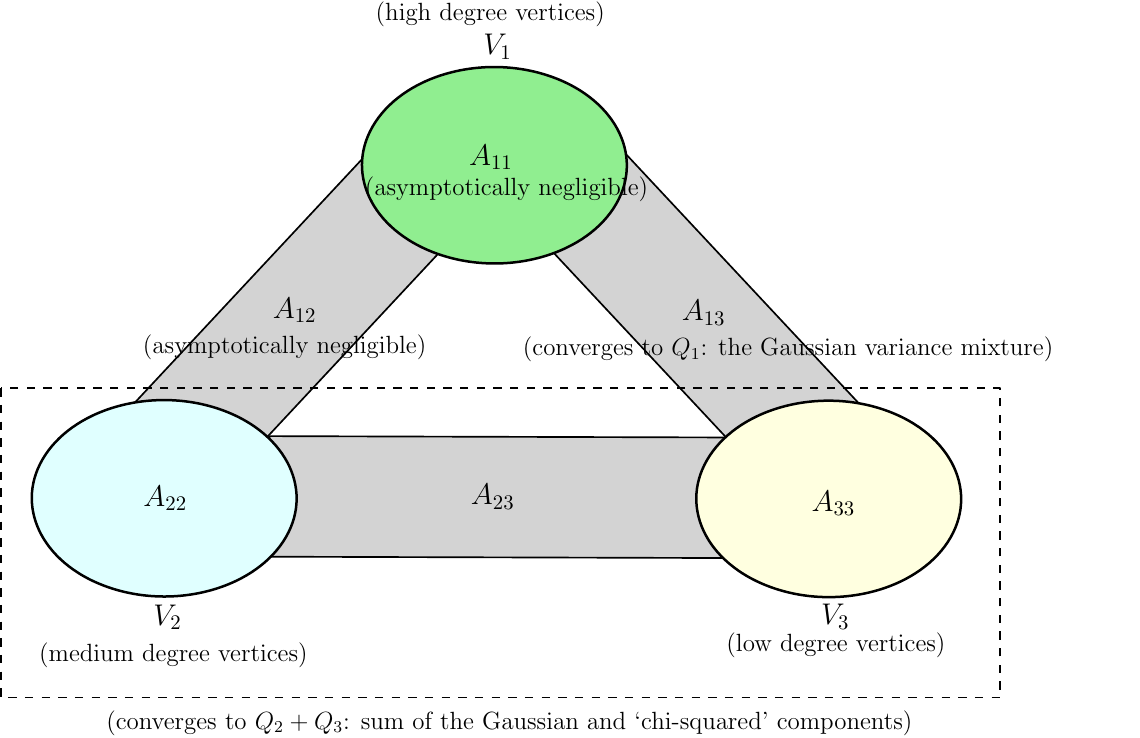}\\
	\end{minipage}
	\caption{\small{The partition of the vertices and the edges of the graph $G_n$ as in the proof of Theorem \ref{thm:main}}.}
	\label{fig:graphpf}
\end{figure}

\begin{enumerate} 
	
	\item The first step in the proof of Theorem \ref{thm:main} is to show that the variables corresponding to medium and low degree vertices, that is, those in the sets $V_2$ and $V_3$ can be replaced by standard Gaussian random variables without incurring any asymptotic error (see Proposition \ref{p:norm} in Section \ref{sec:gaussian_pf}).  The proof uses Lindeberg's method of replacing the non-Gaussian variables by Gaussian variables one by one and using Taylor expansion to get approximation bounds on the error.

	\item The next step is to show that the contribution to $S_{G_n}$ from $A_{11}$ and $A_{12}$ are negligible asymptotically. This follows by a simple second moment argument and the definitions of the sets $V_1, V_2$  (see Lemma \ref{lm:U1112} in Section \ref{sec:gaussian_pf}).

	\item Then in Section \ref{sec:independence} we show that the contributions to $S_{G_n}$ from $A_{13}$  is asymptotically independent in moments from the rest (see Proposition \ref{p:momind}). The proof of Proposition \ref{p:momind} proceeds in two steps: 
	
	\begin{itemize}
		
		\item First we show that the contribution from $A_{33}$ is asymptotically independent in moments from the joint contributions of $A_{13}$, $A_{22}$, and $A_{23}$ (Lemma \ref{momind}). This involves expressing moments as a sum over subgraphs of $G_n$ and then estimating the subgraph counts using the results from extremal combinatorics \cite{alon81,hypergraphcopies} and the relevant degree bounds on the sets $V_2, V_3$. 
		
		\item The second step is to show that the contribution from $A_{13}$ is asymptotically independent from contributions of $A_{22}$ and $A_{23}$ (Lemma \ref{asymp-ind}). For this we show that the covariance between $A_{13}$ and $A_{22} \bigcup A_{23}$ vanishes in the limit and then leverage properties of the Gaussian distribution (recall that the variables in $V_2$ and $V_3$ can be replaced by standard Gaussian variables by step (1) above) to establish the  asymptotic independence. This combined with the previous step establishes the independence of $A_{13}$ from the rest. 
		
	\end{itemize} 
	
	\item Next, in Section \ref{sec:limits1} we compute the limiting distribution corresponding to $A_{13}$ for a fixed truncation level $M$. Here, the limit is a normal variance mixture, where the random  variance is a quadratic form determined by the matrix $\Sigma$ (as defined in Assumption \ref{asum1}),  and the truncated variables $\{X_{u, M}\}_{u \geq 1}$.	
	
	\item  In Section \ref{sec:limits2} we compute the limiting distribution corresponding to $A_{22} \bigcup A_{23}$. Since the variables in $V_2 \bigcup V_3$ have been  replaced by standard Gaussians, we can use the spectral decomposition of $A_{n, K}$ (the adjacency matrix of the truncated graph $G_{n, K}$) and Assumption \ref{asum1} to show that this component converges to $Q_2+Q_3$ as defined in \eqref{eq:main}.

	\item Finally, in Section \ref{sec:minfty} we derive the limit of the $Q_{1, M}$ (the distribution obtained in step (4) above) as $M\to \infty$ limit, thus completing the proof of Theorem \ref{thm:main}. Using the result in Theorem \ref{thm:main} along a subsequence we also prove Theorem \ref{thm:converse} in Section \ref{sec:minfty}. 
\end{enumerate}


\subsection{Gaussian Replacements}
\label{sec:gaussian_pf}

Recall the partition of the vertex set $\lc n \rc=\{1, 2, \ldots, n\}$ into the sets  $V_1, V_2, V_3$ from \eqref{eq:V123}. Then we can partition the random vector $\bm{X}_{n, M}=(X_{1, M},X_{2, M},\ldots,X_{n, M} )^{\top}$ as 
$$\bm{X}_{n, M}=((\bm{X}_{n, M}^{(1)})^{\top},(\bm{X}_{n, M}^{(2)})^{\top},(\bm{X}_{n, M}^{(3)})^{\top})^{\top} , $$ where $\bm{X}_{n, M}^{(s)} = ((X_{j, M}))_{j \in V_s}$ is a $n_s$-dimensional vector, for $s \in \{1, 2, 3\}$. We now show that one can replace the variables in $\bm{X}_{n, M}^{(2)}$ and $\bm{X}_{n, M}^{(3)}$ with standard Gaussian random variables without changing the limiting distribution of $S_{G_n}(\bm{X}_{n, M})$. Towards this end, consider \begin{align}\label{def:Z}
	\bm{Z}_n^{\top} := (Z_1,\ldots,Z_n) := ((\bm{Z}_{n}^{(1)})^{\top},(\bm{Z}_{n}^{(2)})^{\top},(\bm{Z}_{n}^{(3)})^{\top}),
\end{align} 
where $Z_1, Z_2, \ldots$ are i.i.d. $N(0, 1)$ (which are also independent of the collection $\{X_i\}_{1 \leq i \leq n}$) and  $\bm{Z}_{n}^{(s)} = ((Z_{j}))_{j \in V_s}$ is a $n_s$-dimensional vector, for $1 \leq s \leq 3$.

\begin{proposition}\label{p:norm} Let $S_{G_n}$ be as defined in \eqref{def:SG_n} and $h:\mathbb{R} \to \mathbb{C}$ be a bounded, three times continuously differentiable function with $\norm{h'''} \le L <\infty$. {Suppose $F$ has finite third moment}.
	\begin{enumerate}
		\item [(a)] Then 
		\begin{align*}
			|\Ex[h(S_{G_n}({\bm  X}_n))]-\Ex [h(S_{G_n}({\bm  Z}_n))] |\lesssim_L \frac{\max_{1 \leq u \leq n}d_u}{|E(G_n)|}. 
		\end{align*}
		
		\item[(b)] Moreover, 
		\begin{align}\label{eq:norm}
			\lim_{K_1,K_2\to\infty} \sup_{n\ge K_1}\left|\Ex \left[h \left( S_{G_n}\left(\bm{X}_{n, M}\right)\right)\right]-\Ex \left[h\left( S_{G_n}\left(((\bm{X}_{n, M}^{(1)})^{\top},(\bm{Z}_{n}^{(2, 3)})^{\top})^{\top}\right)\right)\right]\right|=0 ,  
		\end{align} 
		where $\bm{Z}_{n}^{(2, 3)}  := ((\bm{Z}_{n}^{(2)})^{\top},(\bm{Z}_{n}^{(3)})^{\top})^\top$. 
	\end{enumerate}
\end{proposition}

\begin{proof}[Proof of Proposition \ref{p:norm}]

	For each $0 \leq u \leq n$, define
	\begin{align}\label{eq:W}
		\bm{W}_u & :=(X_1,X_2,\ldots,X_{n-u},Z_{n-u+1},\ldots,Z_{n})^{\top}, 
	\end{align}
	In other words, the vector $\bm W_u$ is obtained by replacing the last $u$ elements of the vector $\bm X$ with the last $u$ elements of the vector $\bm Z$. Note that $\bm W_0 = \bm X$ and $\bm W_n = \bm Z$, and the vectors $\{\bm{W}_u\}_{0 \leq u \leq n}$ serves as an interpolation between $\bm{X}_n$ and $\bm{Z}_n$, obtaining replacing one coordinate at a time. 
	
	Now, set 
	\begin{align}\label{eq:Tu}
		T_u:=S_{G_n}(\bm{W}_u), 
	\end{align} 
	and fix a bounded thrice continuously differentiable function $h: \mathbb R \to \mathbb C$ with $\norm{h'''}\le L$. 
	We claim that
	for each $0 \leq u \leq n$ we have
	\begin{align}\label{eq:claim}
		\left|\Ex[h(T_{u+1})]-\Ex[h(T_u)]\right| \lesssim_L \left(\frac{d_{n-u}}{\eg}\right)^{\frac{3}2} . 
	\end{align}
	Assuming \eqref{eq:claim} we first show how to complete the proof of Proposition \ref{p:norm}. To this end, summing \eqref{eq:claim} over $u$ from $0$ to $n-1$ gives
	\begin{align*}
		|\Ex h(S_{G_n}({\bm  X}_n))-\Ex h(S_{G_n}({\bm  Z}_n))|=|\Ex h(T_{(0)})-h(T_{(n)})| & \lesssim_L \sum_{u=1}^n\Big(\frac{d_u}{|E(G_n)|}\Big)^{3/2} \nonumber \\ 
		& \le 2\sqrt{ \frac{\max_{1 \leq u \leq n}d_u}{|E(G_n)|}},
	\end{align*}
	which verifies part (a).

	For part (b), we will invoke \eqref{eq:Tu} with $X_n$ replaced by $X_{n,M}$.
	With this choice, observe that $$T_{(n-\lceil K_1\rceil)} = S_{G_n}\left( \left((\bm{X}_{n, M}^{(1)})^{\top}, (\bm{Z}_{n}^{(2, 3)})^{\top} \right)^\top \right)  \quad \text{ and } \quad T_{(0)} = S_{G_n}\left(\bm{X}_{n, M} \right).$$
	Therefore, summing over $u$ from 0 to $n-\lceil K_1 \rceil-1$ on both sides of \eqref{eq:claim} gives, 
	\begin{align}\label{eq:sumc}
		& \nonumber \left|\Ex \left[h \left( S_{G_n}\left(\bm{X}_{n, M}\right)\right)\right]-\Ex \left[ h\left( S_{G_n}\left(((\bm{X}_{n, M}^{(1)})^{\top},(\bm{Z}_{n}^{(2, 3)})^{\top})^{\top}\right)\right) \right]\right| \\ & = \left|\Ex[h(T_{(n-\lceil K_1 \rceil)})]-\Ex[h(T_{(0)})]\right| \nonumber \\ 
		& \lesssim_L \sum_{u=\lceil K_1 \rceil +1}^n\left(\frac{d_{u}}{\eg}\right)^{\frac{3}2} .
	\end{align}
	As the vertices $\lc n \rc$ are labeled such that the degrees of $G_n$ are arranged in non-increasing order, for all $u \ge \lceil K_1 \rceil +1$, 
	\begin{align}\label{eq:Kdegree}
		d_u \le \frac{1}{u} \sum_{v=1}^u d_v \le \frac{1}{u} \sum_{v=1}^n d_v = \frac{2 \eg}{u} \le \frac{2 \eg}{\lceil K_1 \rceil+1}.
	\end{align}
	Thus,
	\begin{align*}
		\sum_{u=\lceil K_1 \rceil +1}^n\left(\frac{d_{u}}{\eg}\right)^{\frac{3}2} \lesssim  \frac1{\sqrt{\lceil K_1 \rceil+1}}\sum_{u=\lceil K_1 \rceil +1}^n\frac{d_{u}}{\eg}  \lesssim \frac1{\sqrt{\lceil K_1 \rceil +1}}.
	\end{align*} 
	Using the above estimate in the RHS of \eqref{eq:sumc} and taking the appropriate limits the result in \eqref{eq:norm} follows. 
\end{proof}

\begin{proof}[Proof of \eqref{eq:claim}] For $0 \leq u \leq n$ recall the definition of $\bm W_u$ from \eqref{eq:W}. In this proof, we will denote the $v$-th coordinate of $\bm W_u$ as $W_{u, v}$, for $1 \leq v \leq n$. For $0 \leq u \leq n$, define the following notations: 
	\begin{align*}
		Q_u & :=\frac1{\sqrt{\eg}} \sum_{\substack{1 \leq v \leq n \\  v \neq n-u}}  a_{n-u, v} W_{u, v}, \quad R_u:=\frac1{\sqrt{\eg}}\sum_{\substack{1 \leq v < v' \leq n \\  v, v' \neq n-u}} a_{v, v'} W_{u, v} W_{u, v'},
	\end{align*}	
	Now, recall the definition of $T_u$ from \eqref{eq:Tu} and observe the following: 
	\begin{enumerate}	
		\item $T_{u}=X_{n-u}Q_u+R_u$ and $T_{u+1}=Z_{n-u}Q_u+R_u$,
		\item $X_{n-u}$ and $Z_{n-u}$ are independent of $(Q_u, R_u)$.
	\end{enumerate}	
	The second observation implies $$\Ex\left[{Z_{n-u}Q_u}h'\left({R_u}\right)\right]  =\Ex\left[{X_{n-u}Q_u}h'\left({R_u}\right)\right]$$ 
	and $$\Ex\left[{Z_{n-u}^2 Q_u^2}h''\left({R_u}\right)\right]  =\Ex\left[{X_{n-u}^2 Q_u^2}h''\left({R_u}\right)\right].$$
	Using the above identifies and the triangle inequality gives 
	\begin{align} 
		& \left|\Ex[h(T_{u+1})-h(T_{u})]\right|  \nonumber \\  
		& \le \left|\Ex\left[h(T_{u+1})-h(R_u)-Z_{n-u}Q_uh'(R_u)-\tfrac12Z_{n-u}^2(Q_u)^2h''(R_u)\right]\right| \nonumber \\ 
		& \hspace{2.5cm}  + \left|\Ex\left[h(T_{u})-h(R_u)-X_{n-u}Q_uh'(R_u)-\tfrac12X_{n-u}^2(Q_u)^2h''(R_u)\right]\right| \nonumber \\  
		& \le \tfrac{L}{6} \left\{ \Ex\left|Z_{n-u}^3 Q_u^3\right|+ \Ex\left|X_{n-u}^3 Q_u^3\right| \right \} \tag*{(since $\norm{ h ''' } \leq L$)} \nonumber \\ 
		& \lesssim_L \Ex\left| Q_u^3\right|(\Ex|X_{n-u}^3|+\Ex|Z_{n-u}^3|) . 
		\label{eq:h1} 
	\end{align}
	Note that $\{W_{u, v}\}_{1 \leq v \leq n}$ are independent mean $0$ variance $1$ random variables with finite fourth moments. Thus, by H\"older's inequality, 
	\begin{align}\label{eq:h2}
		{\eg}^{\frac32}\Ex \left|Q_u^3\right| & \le {\eg}^{\frac32}\left(\Ex \left[Q_u^4\right]\right)^{\frac{3}4} \nonumber \\ 
		& \lesssim \left[ \sum_{\substack{1 \leq v \leq n \\  v \neq n-u}} a_{n-u, v} \Ex[ W_{u, v}^4 ] + \sum_{\substack{1 \leq v \ne v' \leq n \\  v, v' \neq n-u}} a_{n-u, v} a_{n-u, v'} \Ex[W_{u, v}^2] \Ex[W_{u, v'}^2]   \right]^{\frac34} \nonumber \\
		& \lesssim \left[ \sum_{\substack{1 \leq v \leq n \\  v \neq n-u}} a_{n-u, v}  +\left( \sum_{\substack{1 \leq v \leq n \\  v \neq n-u}} a_{n-u, v}  \right)^2\right]^{\frac34} \nonumber \\ 
		& \lesssim d_{n-u}^{\frac32} . 
	\end{align} 
	As $X_{n-u}$ and $Z_{n-u}$ have finite third moments, combining \eqref{eq:h1} and \eqref{eq:h2} the result in \eqref{eq:claim} follows. 
\end{proof}

Now, recalling the block decomposition of the matrix $A(G_n)$ from \eqref{eq:Ablock} gives, 
\begin{align}\label{eq:U23} 
	S_{G_n} & \left(((\bm{X}_{n, M}^{(1)})^{\top},(\bm{Z}_{n}^{(2, 3)})^{\top})^{\top}\right) \nonumber \\ 
	& := U_{11,n,M} + U_{12,n,M} + U_{13,n,M} + U_{22,n} + U_{23,n} + U_{33,n}, 
\end{align} 
where 
\begin{align}
	U_{11,n,M} & : = \frac{(\bm{X}_{n, M}^{(1)})^{\top} A_{11} \bm{X}_{n, M}^{(1)}}{ 2 \sqrt{|E(G_n)|} }, & U_{12,n,M}  & : = \frac{(\bm{X}_{n, M}^{(1)})^{\top} A_{12} \bm{Z}_{n}^{(2)} } { \sqrt{|E(G_n)|} }, & U_{13,n,M}  & : = \frac{ \bm{X}_{n, M}^{(1)}A_{13} (\bm{Z}_{n}^{(3)})^{\top} }{  \sqrt{|E(G_n)|} } , \nonumber \\ 
	U_{22,n} & := \frac{(\bm{Z}_{n}^{(2)})^{\top} A_{22} \bm{Z}_{n}^{(2)}}{ 2 \sqrt{|E(G_n)|}} , & U_{23,n}  & := \frac{(\bm{Z}_{n}^{(2)})^{\top} A_{23} \bm{Z}_{n}^{(3)}}{ \sqrt{|E(G_n)|}} , & U_{33,n}  & := \frac{(\bm{Z}_{n}^{(3)})^{\top} A_{33} \bm{Z}_{n}^{(3)}}{ 2 \sqrt{|E(G_n)|}} . \nonumber  
\end{align} 
(Note that $U_{13,n,M}, U_{22,n}, U_{23,n},$ and $U_{33,n}$ depends on $K_1$ and $K_2$ as well (recall \eqref{eq:V123}), but we suppress this dependence for notational convenience.) 

The proof of Theorem \ref{thm:main} now proceeds by analyzing the six terms in \eqref{eq:U23}. The first step is to show that the terms $U_{11,n,M} $ and $U_{12,n,M}$ are negligible asymptotically.  

\begin{lemma}\label{lm:U1112} For any fixed $M,K_1,K_2$ with $U_{11,n,M}$ and $U_{12,n,M}$ as defined above the following hold: 
	\begin{align*} 
		\lim_{n \to\infty}  \Pr(|U_{11,n,M} + U_{12,n,M} |>\e) = 0.
	\end{align*}
\end{lemma}

\begin{proof} Using the fact that $\vr[A+B] \le 2\vr[A]+2\vr[B]$ and an application of Markov inequality  shows that 
	\begin{align}
		\Pr\left(|U_{11,n,M} + U_{12,n,M} |>\e\right) & \lesssim \frac1{\e^2\eg}\left\{ \vr[(\bm{X}_{n, M}^{(1)})^{\top} A_{11} \bm{X}_{n, M}^{(1)} ]+\operatorname{Var}[(\bm{X}_{n, M}^{(1)})^{\top} A_{12} \bm{Z}_{n}^{(2)}] \right\} \nonumber \\ 
		&  \lesssim \frac1{\e^2\eg}\left\{\tr(A_{11}^2)+ \tr(A_{12}^2) \right\} \nonumber \\ 
		& \le \frac1{\e^2\eg}\left\{ |V_1|^2 +  |V_1| |V_2| \right\} \nonumber \\ 
		& \le \frac1{\e^2\eg}\left[K_1^2+K_1K_2\sqrt{\eg}\right] . \nonumber 
	\end{align} 
	The expression above goes to zero as $n\to \infty$ followed, for all $K_1,K_2$ fixed. This completes the proof of the lemma.  
\end{proof}

\subsection{Asymptotic Independence of $U_{13,n,M}$ and $U_{22,n}+U_{23,n}+U_{33,n}$}\label{sec:independence}

The goal of this section is to show that $U_{13,n,M}$ is asymptotically independent of $U_{22,n}+U_{23,n}+U_{33,n}$ in moments. This is formalized in the following proposition:

\begin{proposition}\label{p:momind} Fix $K_2$ and $M$ large enough such that $U_{13,n,M}$ is well defined in \eqref{eq:U23}. Then for all non-negative integers $a, b$, 
	\begin{align*}
		\lim_{K_1\to\infty}\limsup_{n\to\infty} \left|\Ex\left[U_{13,n,M}^{a}(U_{22,n}+U_{23,n}+U_{33,n})^{b}\right]-\Ex\left[U_{13,n,M}^{a}\right]\Ex\left[(U_{22,n}+U_{23,n}+U_{33,n})^{b}\right]\right|=0.
	\end{align*}
\end{proposition}

\begin{proof}[Proof of Proposition \ref{p:momind}] 
	
	The first step in the proof is to show that $U_{33,n}$ is asymptotically independent of the triplet $(U_{13,n,M},U_{22,n},U_{23,n})$ in moments. This is shown in the following lemma which is proved in Section \ref{sec:momind1}.  
	
	\begin{lemma}\label{momind} For all non-negative integers $a,b,c,d$ the following holds: 
		\begin{align*}
			\lim_{K_1,K_2\to\infty}\limsup_{n\to\infty} \left|\Ex\left[U_{13,n,M}^{a}U_{22,n}^{b}U_{23,n}^{c}U_{33,n}^d \right]-\Ex\left[U_{13,n,M}^{a}U_{22,n}^{b}U_{23,n}^c \right]\Ex\left[U_{33,n}^{d}\right]\right|=0.
		\end{align*}
	\end{lemma}
	
	Next, we show that $U_{13,n,M}$ is asymptotically independent of $U_{22,n}+U_{23,n}$ using the Gaussian structure. 
	This is described in the following lemma which is proved in Section \ref{sec:indpf}.

	\begin{lemma}\label{asymp-ind} 
		For any $t_1, t_2\in \R$,	
		\begin{align*}
			\lim_{K_1\to\infty}\limsup_{n\to\infty}\left|\Ex[e^{\i  t_1U_{13,n,M}+i t_2(U_{22,n}+U_{23,n})}] -\Ex[e^{\i t_1 U_{13,n,M}}]\Ex[e^{\i  t_2(U_{22,n}+U_{23,n})}] \right|=0.
		\end{align*} 
	\end{lemma}

	Given the above two lemmas, the proof of Proposition \ref{p:momind} can be completed easily. Recall that $$U_{13,n,M}:=\frac1{2\sqrt{\eg}}\bm{X}_{n,M}^{(1)}A_{13}(\bm{Z}_n^{(3)})^\top.$$ For each fixed $M$, the truncated random variables $\{X_{u, M}\}_{1 \leq u \leq n}$ and $\{Z_u\}_{1 \leq u \leq n}$ are centered, bounded, and hence, sub-Gaussian. Therefore,  the matrix $\frac1{2\sqrt{\eg}}A_{13}$ satisfies the condition \eqref{mgf-cond} in Appendix \ref{sec:appendix}. Similarly, since $U_{22,n},U_{23,n},$ and $U_{33,n}$ involve Gaussian random variables, the corresponding matrices satisfies \eqref{mgf-cond}.  Thus, by Lemma \ref{l:mgf} there exists $\delta>0$ such that 
	\begin{align}\label{eq:momfin}
		\sup_{|t|\le \delta} \sup_{K_1,K_2\ge 1} \sup_{n\ge 1} \left\{\Ex[e^{tU_{13,n,M}}]+\Ex[e^{tU_{22,n}}]+\Ex[e^{tU_{23,n}}]+\Ex[e^{tU_{33,n}}]\right\}<\infty.
	\end{align}
	{Now, let us denote 
	$$\mathcal R_{n,K_1}:= U_{13,n,M} \quad \text{and} \quad \mathcal S_{n,K_1} := U_{22,n}+U_{23,n} . $$
	By \eqref{eq:momfin}, we have  $\{(\mathcal R_{n,K_1}, \mathcal S_{n,K_1})\}_{n,K_1\ge 1}$ is a tight sequence of random variables. Therefore, passing through a double subsequence (in both $n,K_1$) we may assume $(\mathcal R_{n,K_1}, \mathcal S_{n,K_1})$ converges in distribution, that is, there exists random variables $\mathcal R, \mathcal S$ on $\mathbb R$ such that $$(\mathcal R_{n,K_1}, \mathcal S_{n,K_1})\stackrel{D}{\to} (\mathcal R, \mathcal S).$$ Now, Lemma \ref{asymp-ind} implies $\mathcal R$ and $\mathcal S$ must be independent.  Therefore, by uniform integrability, for $a, b \geq 1$ integers, 
$$\lim_{K_1\to\infty} \lim_{n\to\infty}\Ex[\mathcal R_{n,K_1}^a \mathcal S_{n,K_1}^b]=\Ex [\mathcal R^a \mathcal S^b]=\Ex[\mathcal R^a]\Ex[\mathcal S^b]=\lim_{K_1\to\infty}\lim_{n\to\infty} \Ex[\mathcal R_{n,K_1}^a]\Ex[\mathcal S_{n,K_1}^b],$$
and hence, 
\begin{align}\label{eq:mmln}
 \lim_{K_1\to\infty}	\lim_{n\to \infty}\left|\Ex[\mathcal R_{n,K_1}^a \mathcal S_{n,K_1}^b]-\Ex[\mathcal R_{n,K_1}^a]\Ex[\mathcal S_{n,K_1}^b]\right| = 0.
	\end{align}
Combining \eqref{eq:mmln} with Lemma \ref{momind} and binomial theorem, the result in Proposition \ref{p:momind} follows. }
\end{proof}

\subsubsection{Proof of Lemma \ref{momind}} \label{sec:momind1}

In the method of moment calculations of Lemma \ref{momind}, we will need to invoke various results from extremal graph theory.  Towards this, we begin by recalling some basic definitions. 

\medskip

\noindent{\it Extremal Graph Theory Background}: 
For any graph $H$, denote the neighborhood of a set $S\subseteq V(H)$ by $N_H(S)=\{v\in V(H): \exists u \in S \text{ and } (u, v)\in E(H) \}$. 
Next, given two simple graphs $G=(V(G), E(G))$ and $H=(V(H), E(H))$, denote by $N(H, G)$ the number of 
isomorphic copies of $H$ in $G$, that is, 
$$N(H, G)=\sum_{S\subset E(G):|S|=|E(H)|}\pmb 1\{G[S]\cong H\},$$ 
where the sum is over subsets $S$ of $E(G)$ with $|S|=|E(H)|$, and $G[S]$ is the subgraph of $G$ induced by the edges of $S$. 

One of the fundamental problems in extremal graph theory is to estimate, for any fixed graph $H$, the quantity $N(H, G)$ over the class of graphs $G$ with a specified number of edges. More formally, for a positive integer $\ell\geq |E(H)|$, define $$N(H, \ell):=\sup_{G: |E(G)|=\ell}N(H, G).$$
For the complete graph $K_h$, Erd\H os \cite{erdoscompletesubgraphs} determined the asymptotic behavior of $N(K_h, \ell)$ as $\ell$ tends to infinity,  which is also a special case of the celebrated Kruskal-Katona theorem. For a general graph $H$ this problem was settled by Alon \cite{alon81} and later extended to hypergraphs by Friedgut and Kahn \cite{hypergraphcopies}. To explain this result we need the following definition: 

\begin{definition}\cite{schrijver} \label{gammah}
	For any graph $H = ( V(H), E(H) )$ define its  {\it fractional stable number} $\gamma(H)$ as: 
	\begin{equation}\label{eq:gamma}
		\gamma(H)=\max_{\substack{\phi\in [0,1]^{V(H)}, \\ \phi(x)+\phi(y)\le 1 \text{ for }(x,y)\in E(H)}} \sum_{v\in V(H)} \phi(v) , 
	\end{equation}
	where $[0,1]^{V(H)}$ is the collection of all functions $\phi: V(H)\rightarrow [0, 1]$.   Moreover, the polytope $\mathscr{P}(H)$ defined by the set of constraint $$\{\phi\in [0,1]^{V(H)}: \phi(u) + \phi(v) \leq 1~\textrm{for all}~ \{u,v\} \in E(H)\}$$ is called the \textit{fractional stable set polytope} of the graph $H$. 
\end{definition}

We now state Alon's result \cite{alon81} result (in the language of \cite{hypergraphcopies}) in the following theorem for ease of referencing:

\begin{theorem}[\cite{alon81,hypergraphcopies}]For a fixed graph $H$, there exist two positive constants $C_1=C_1(H)$ and $C_2=C_2(H)$ such that for all $\ell\geq |E(H)|$,
	\begin{equation*}
		C_1\ell^{\gamma(H)}\leq N(H, \ell) \leq C_2\ell^{\gamma(H)},
	\end{equation*}
	where $\gamma(H)$ is {\it fractional stable number} of $H$ as in Definition \ref{gammah}. 
	\label{thm:alonexponent}
\end{theorem}

In the proof of Lemma \ref{momind} we will often need to deal with counts of multi-graphs (instead of simple graphs) containing an edge in $E_{3,3}$.  A  {\it multi-graph} $H$ is a graph with no self-loops, where there might be more than one edge between two vertices.  In this case, the set of edges $E(H)$ is a multi-set where the edges are counted with their multiplicities.

\medskip

\noindent\textbf{\textit{Representing the Joint Moments in Terms of Multigraph Counts}}: With the above results we now proceed with the proof of Lemma \ref{momind}. To begin with note that the result in Lemma \ref{momind} is trivial if $d=0$. Hence, we assume that $d$ is positive. Throughout this proof we will drop the subscript $n$ from $U_{22,n}, U_{23,n}, U_{33,n}$ and the subscripts $n$ and $M$ from $U_{13, n, M}$.

Now, let $E_{s, t}$ denote the set of all edges in $G_n$ with one vertex in $V_s$ and other in $V_t$, for $1 \leq s, t \leq 3$ (recall the definitions of $V_1, V_2, V_3$ from \eqref{eq:V123}). Note that
\begin{align}\label{eq:iden}
	\Ex\left[U_{13}^{a}U_{22}^{b}U_{23}^{c} U_{33}^d \right]-\Ex\left[U_{13}^{a}U_{22}^{b}U_{23}^c\right]\Ex\left[U_{33}^{d}\right]  
	& = \frac1{{\eg}^{\frac{a+b+c+d}2}}\sum_{\bm u} B_{\bm u}.
\end{align}
where the sum is over all tuples of the form:
\begin{align}\label{eq:uabcd}
	\bm u & := \left(\{u_1,v_1\},\ldots,\{u_a,v_a\},\{u_1',v_1'\},\ldots,\{u_b',v_b'\},\right. \nonumber \\ 
	& \hspace{5cm}\left.\{u_1'',v_1''\},\ldots,\{u_c'',v_c''\},\{u_1''',v_1'''\},\ldots,\{u_d''',v_d'''\}\right),
\end{align}
with $\{u_1,v_1\},\ldots,\{u_a, v_a\} \in E_{1,3}$, $\{u_1',v_1'\},\ldots,\{u_b',v_b'\} \in E_{2,2}$, $\{u_1'',v_1''\},\ldots,\{u_c'',v_c''\} \in E_{2,3}$, and $\{u_1''',v_1'''\}, \ldots,\{u_d''',v_d'''\} \in E_{3,3}$; and 
\begin{align*}
	B_{\bm u} & :=\mathbb{E}\left[\prod_{j=1}^a X_{u_j, M} Z_{v_j} \prod_{j=1}^b Z_{u_j'} Z_{v_j'}\prod_{j=1}^c Z_{u_j''} Z_{v_j''}\prod_{j=1}^d Z_{u_j'''} Z_{v_j'''}\right] \\ & \hspace{4cm} - \mathbb{E}\left[\prod_{j=1}^a X_{u_j, M} Z_{v_j} \prod_{j=1}^b  Z_{u_j'} Z_{v_j'}\prod_{j=1}^c Z_{u_j''} Z_{v_j''}\right]\mathbb{E}\left[\prod_{j=1}^d Z_{u_j'''} Z_{v_j'''}\right] . 
\end{align*}
{Let $\mathcal{U}_{a, b, c, d}$ be the collection of all tuples $\bm u$ as in \eqref{eq:uabcd}, which arises in the sum \eqref{eq:iden}, i.e.~$\mathcal{U}_{a,b,c,d}$ is the collection of $a+b+c+d$ edge tuples with each edge from $E(G_n)$, such that
\begin{itemize}
\item
the number of edges in $\bm u$ which are in $E_{1,3}$ is $a$;
\item
the number of edges in $\bm u$ which are in $E_{2,2}$ is $b$;
\item
the number of edges in $\bm u$ which are in $E_{2,3}$ is $c$;
\item
the number of edges in $\bm u$ which are in $E_{3,3}$ is $d$.
\end{itemize}
 Let $\mathcal{U}_{a,b,c,d}^{(0)}\subset \mathcal{U}_{a,b,c,d}$ be the set of all $\bm u$ such that $B_{\bm u}\ne 0$.
Note, since all moments of $X_{1, M}$ are bounded, $|B_{\bm u}| \lesssim_{a, b, c, d, M} 1$. Therefore, \eqref{eq:iden} can be bounded as: 
\begin{align}\label{eq:U}
	\left|\Ex\left[U_{13}^{a}U_{22}^{b}U_{23}^{c} U_{33}^d \right]- \Ex\left[U_{13}^{a}U_{22}^{b}U_{23}^c\right]\Ex\left[U_{33}^{d}\right]  \right| & \lesssim_{a, b, c, d, M} \frac1{|E(G_n)|^{\frac{a+b+c+d}2}}|\mathcal{U}^{(0)}_{a, b, c, d}| .  
\end{align}
For every $\bm u \in \mathcal{U}_{a, b, c, d}$, let $\mathcal{H}({\bm u})$ denote the multigraph formed by the edges in $\bm u$. Note that the multigraph $\mathcal{H}({\bm u})$ has $a+b+c+d$ multi-edges, for each $\bm u \in \mathcal{U}_{a, b, c, d}$. As in simple graphs, we define the degree $d_v(H)$ of a vertex $v$ in a multigraph $H=(V(H), E(H))$ as the number of multi-edges of $H$ incident on the vertex $v$. Also, let $d_{\min}(H) := \min\{d_v(H): v \in V(H)\}$ denote the minimum degree of $H$. 
\begin{lemma}\label{lm:Hgamma} For $H=\mathcal{H}(\bm u)$ with $\bm u \in \mathcal{U}_{a, b, c, d}$, with connected components $H_1, H_2, \ldots, H_\nu$, the following hold: 
	\begin{enumerate}
		\item[$(1)$]  For all $1 \leq i \leq \nu$, $\gamma(H_i) \geq |V(H_i)|/2$. 
		\item[$(2)$] If  ${\bm u}\in \mathcal{U}^{(0)}_{a, b, c, d}$, then $d_{\min}(H)\ge 2$.
		\item[$(3)$] If  $d_{\min}(H) \geq 2$, then  $|V(H_i)| \leq |E(H_i)|$ for all $i\in [\nu]$. Moreover, if there exists a vertex $v'\in V(H_i)$ such that $d_{v'}(H_i) \ge 3$, then $|V(H_i)|< |E(H_i)|$. 
		\item[$(4)$] If  $d_{\min}(H) \geq 2$, then for all $1 \leq i \leq \nu$ we have $\gamma(H_i) \le |E(H_i)|/2 $. Moreover, if there exists a vertex $v'\in V(H_i)$ such that $d_{v'}(H_i) \ge 3$ and $\varphi(v') > 0$, where $\varphi$ is an optimal solution to \eqref{eq:gamma}, then $|\gamma(H_i)|< |E(H_i)|/2$. 
		\item[$(5)$] If ${\bm u}\in \mathcal{U}^{(0)}_{a, b, c, d}$, then $H_S$ contains a $2$-star as a subgraph with one edge in $E_{3,3}$. 
	\end{enumerate}
\end{lemma}
\begin{proof} The proof of (1) follows by taking $\phi\equiv \frac12$ in \eqref{eq:gamma} for the graph $H_i$. 
	\\
	For (2), since $X_{1, M}$ and $Z \sim N(0, 1)$ are both mean zero random variables, $B_{\bm u}$ is non zero only when each vertex index in $\bm u$ appears at least twice, and so $d_{\min}(H)\ge 2$.
\\	
	Proceeding to (3), we have $d_{\min}(H_i) \geq d_{\min}(H) \geq 2$, for $1 \leq i \leq \nu$. Hence, 
	$$2 |V(H_i)| \leq |V(H_i)| d_{\min}(H_i) \leq \sum_{v \in V(H_i)} d_v(H_i) = 2|E(H_i)|,$$
	that is, $|V(H_i)| \leq |E(H_i)|$. Now, suppose there exists $v' \in V(H_i)$ such that $d_v(H_i) \ge 3$, then 
	$$2 (|V(H_i)|-1) + 3  \leq \sum_{v \in V(H_i)} d_v(H_i) = 2|E(H_i)| ,$$
	that is, $|V(H_i)| < |E(H_i)$. 
	\\	
  For verifying (4), note that if $d_{\min}(H_i)\geq 2$, then for any function $\phi \in [0, 1]^{V(H_i)}$ with $\phi(x) + \phi(y) \leq 1$, for $(x, y) \in E(H_i)$,  
	$$\sum_{x\in V(H_i)} \phi(x)\leq \frac{1}{d_{\min}(H_i)}\sum_{(x, y)\in E(H_i)}\{\phi(x)+\phi(y)\}\leq\frac{1}{2}|E(H_i)|.$$
	which gives $\gamma(H_i)\leq \frac{1}{2}|E(H_i)|$. Now, suppose there exists $v' \in V(H_i)$ such that $d_{v'}(H_i) \geq 3$ and $\varphi(v')\ne 0$, where $\varphi$ is an optimal solution to \eqref{eq:gamma}. Then
	\begin{align*}
		|E(H_i)|  \geq  \sum_{(x, y)\in E(H_i)}\{\varphi(x)+\varphi(y)\}=\sum_{x\in V(H_i)} d_x(H_i)\varphi(x) & \geq  3\varphi(v')+2\sum_{x\in V(H_i)\backslash \{v'\}}\varphi(x) \nonumber \\ 
		& =2\gamma(H_i)+\varphi(v'),
	\end{align*}
	and the result in (3) follows since $\varphi(v')>0$. 
\\
	Finally, note that if $H_S$ does not contain a 2-star with an edge in $E_{3, 3}$, then the edge set of each connected component of $H_S$ must be either totally contained in $E_{3,3}$ or totally contained in $E_{3,3}^c$, which implies $B_{\bm u} = 0$. 
\end{proof} 
We now define three disjoint subsets of $\mathcal{U}_{a, b, c, d}$:
\begin{enumerate} 
	\item $\mathcal{U}_{a, b, c, d}^{(1)}$ is the collection of all 
	$\bm u \in \mathcal{U}_{a, b, c, d}$ such that $H = \mathcal{H}(\bm u)$ has connected components $H_1, H_2, \ldots, H_\nu$ which satisfy $\gamma(H_i) \leq |E(H_i)|/2$ for all $1 \leq i \leq \nu$ and $\gamma(H_j) < |E(H_j)|/2$ for some $1 \leq j \leq \nu$. 
	\item $\mathcal{U}_{a, b, c, d}^{(2)}$ is the collection of all 
	$\bm u \in \mathcal{U}_{a, b, c, d}$ such that $H = \mathcal{H}(\bm u)$ has connected components $H_1, H_2, \ldots, H_\nu$ which satisfiy $\gamma(H_i) = |E(H_i)|/2$ for all $1 \leq i \leq \nu$, and there exists $1 \leq j \leq \nu$ such that $\gamma(H_j) = |E(H_j)|/2= |V(H_j)|/2$ and $(H_j)_S$ has a 2-star with one edge in $E_{3, 3}$. 
	\item $\mathcal{U}_{a, b, c, d}^{(3)}$ is the collection of all $\bm u \in \mathcal{U}_{a, b, c, d}$ such that $H = \mathcal{H}(\bm u)$, has connected components $H_1, H_2, \ldots, H_\nu$ which satisfy $\gamma(H_i) = |E(H_i)|/2$ for all $1 \leq i \leq \nu$ and there exists $1 \leq j \leq \nu$ such that $\gamma(H_j) = |E(H_j)|/2 > |V(H_j)|/2$ and $(H_j)_S$ has a 2-star with one edge in $E_{3, 3}$. 
\end{enumerate}
Note that  $\mathcal{U}_{a, b, c, d}^{(1)}$, $\mathcal{U}_{a, b, c, d}^{(2)}$, and $\mathcal{U}_{a, b, c, d}^{(3)}$ are disjoint events which cover all possibilities in $\mathcal{U}^{(0)}_{a, b, c, d}$ by Lemma \ref{lm:Hgamma}.  
\\
 For an unlabeled multi-graph $H$ define
$M(H,G_n)$ to be the cardinality of the set of all $\bm u\in \mathcal{U}_{a,b,c,d}$ such that the multigraph $\mathcal{H}(\bm u)$ is isomorphic to $H$, and  has a 2-star with one edge in $E_{3,3}$.
It is easy to see that $M(H, G)\lesssim_{a,b,c,d} N(H_S, G)$, where $H_S$ is the simple graph obtained from $H$ by replacing the edges between the vertices which occur more than once by a single edge. Moreover, the definition of the fractional stable number in \eqref{eq:gamma} extends verbatim to any multi-graph $H$, by setting  
$\gamma(H):=\gamma(H_S)$. Theorem \ref{thm:alonexponent} then implies,  
\begin{align}\label{eq:alon}
M(H, G)\lesssim_{a,b,c,d} N(H_S, G) \lesssim_{a,b,c,d}  |E(G)|^{\gamma(H_S)}=|E(G)|^{\gamma(H)}.
\end{align}
Finally, denote by $\mathcal H_{a, b, c, d}$ the collection of all unlabelled multigraphs $H = (V(H), E(H)) $ with $a+b+c+d$ edges, which has a cardinality free of $n$ (in fact it depends only on $a+b+c+d$).
\begin{lemma}\label{lm:H1} $|\mathcal{U}_{a, b, c, d}^{(1)}\cap \mathcal{U}_{a,b,c,d}^{(0)}| = o\left(|E(G_n)|^{(a+b+c+d)/2}\right)$. 
\end{lemma}
\begin{proof} For any $H = \mathcal{H}(\bm u)$ with $\bm u \in \mathcal{U}_{a, b, c, d}^{(1)}\cap \mathcal{U}_{a,b,c,d}^{(0)}$, using definition of $\mathcal{U}^{(1)}_{a,b,c,d}$ we have
$\gamma(H) = \sum_{i=1}^\nu \gamma(H_i) < \sum_{i=1}^\nu |E(H_i)|/2 = |E(H)|/2 = (a+b+c+d)/2$. Then by \eqref{eq:alon}, 
	$$M(H, G_n) \lesssim_{a, b, c, d}  |E(G_n)|^{\gamma(H)} = o\left(|E(G_n)|^{(a+b+c+d)/2}\right).$$ 
	 Then, with $\mathcal{H}_{a,b,c,d}$ as introduced before the lemma, we have
	$$|\mathcal{U}_{a, b, c, d}^{(1)}\cap \mathcal{U}_{a,b,c,d}^{(0)}| \leq \sum_{H \in \mathcal H_{a, b, c, d}} M(H, G_n) \bm 1\{H \cong \mathcal{H}(\bm u), \bm u \in \mathcal{U}_{a, b, c, d}^{(1)}\cap \mathcal{U}_{a,b,c,d}^{(0)}\} = o\left(|E(G_n)|^{(a+b+c+d)/2}\right),$$
	since $|\mathcal H_{a, b, c, d}| \lesssim_{a, b, c , d} 1$. 
\end{proof}
\begin{lemma}\label{lm:H2} $|\mathcal{U}_{a, b, c, d}^{(2)}\cap \mathcal{U}^{(0)}_{a, b, c, d}| \lesssim_{a, b, c, d}\frac{1}{K_2} |E(G_n)|^{(a+b+c+d)/2}$, where $K_2$ is as in \eqref{eq:V123}.  
\end{lemma}
\begin{proof}[Proof of Lemma \ref{lm:H2}] 
	For any $H = \mathcal{H}(\bm u)$ with $\bm u \in \mathcal{U}_{a, b, c, d}^{(2)}\cap \mathcal{U}_{a,b,c,d}^{(0)}$, consider the connected component $H_j$, for $1 \leq j \leq \nu$, such that $\gamma(H_j)=|E(H_j)|/2 = |V(H_j)|/2$ and $(H_j)_S$ has a 2-star with one edge in $E_{3, 3}$ (this exists by definition of $\mathcal{U}_{a, b, c, d}^{(2)}$).  Without loss of generality, suppose $j=1$. The following lemma shows that $H_1$ is a cycle or a doubled edge. 
	\begin{observation}
		Let $F=(V(F), E(F))$ be any connected multi-graph with $d_{\min}(F)\geq 2$ and $|V(F)|=|E(F)|$. Then $F$ is a cycle or a doubled edge. 
		\label{obs:cycle_edge_vertex}
	\end{observation}
	\begin{proof}
		Given the graph $F$, denote by $F_S$ the underlying simple graph. Since $F$ is connected, $|E(F_S)| \le |E(F)|$ and $|E(F_S)| \ge |V(F_S)|-1 = |V(F)|-1 = |E(F)|-1$.
		Hence, $|E(F_S)|=|E(F)|=|V(F)|$ or $|E(F_S)|=|E(F)|-1=|V(F)|-1$.
		If $|E(F_S)|=|V(F)|=|E(F)|$, then $F$ itself is a simple graph with  $d_{\min}(F)\geq 2$, which implies that $F$ is a cycle of length $|V(F)|$.
		On the other hand, if $|E(F_S)|=|V(F)|-1=|E(F)|-1$, then $F_S$ is a tree. Note that any tree has at least two degree 1 vertices. If any two of the degree 1 vertices in $F_S$ are not connected by an edge, then adding one extra (double) edge cannot add to both their degrees. Hence, all the degree 1 vertices must be adjacent. 
		This implies that $F_S$ is an isolated edge, and $F$ is a doubled edge.  
	\end{proof}
	Since $|E(H_1)| = |V(H_1)|$ and $d_{\mathrm{min}}(H_1) \geq 2$ (by Lemma \ref{lm:Hgamma} (2)), $H_1$ is a cycle or a doubled edge by Observation \ref{obs:cycle_edge_vertex}. By assumption $(H_1)_S$ contains a $2$-star with one edge in $E_{3,3}$, which implies,  $(H_1)_S$ is a cycle with at least one edge in $E_{3, 3}$. In the following observation we estimate the number of such cycles. 
	\begin{observation}\label{obs:C} 
	With $C_g$ denoting a cycle of length $g$, we have
\begin{align*}
	M(C_g, G_n) \lesssim_{g} \frac{1}{K_2}|E(G_n)|^{g/2}.
\end{align*}
\end{observation}
\begin{proof} To begin with suppose $g=2L+1$, for some $L \geq 1$, is odd.  Choose a cyclic enumeration $s_1,\ldots,s_{2L+1}, s_1$ of the vertices of $C_g$, such that the edge $\{s_1,s_2\}$ is in $E_{3,3}$. Clearly, $\{s_1,s_2\}$ can be chosen in at most $|E(G_n)|$ ways. Once this edge is fixed, the edge $\{s_2,s_3\}$ of $C_g$ has at most $\frac2{K_2}\sqrt{|E(G_n)|}$ choices, since $s_2 \in V_3$ and $\max_{v \in V_3} d_v \leq \frac2{K_2}\sqrt{|E(G_n)|}$ (recall the definition of the set $V_3$ from \eqref{eq:V123} and note that the vertices of $G_n$ are arranged in non-decreasing order of the degrees). Now, for each $2 \leq \ell \leq L$, the edge $\{s_{2\ell},s_{2\ell+1}\}$ has $|E(G_n)|$ choices. Once these edges are chosen the cycle is completely determined. Hence, 
$$M(C_g, G_n) \lesssim \frac1{K_2}|E(G_n)|^{L+\frac{1}{2}}.$$
Next, suppose $g = 2L$, for some $L \geq 2$, is even. Once again, fix a cyclic enumeration $s_1,\ldots,s_{2L}, s_1$ of the vertices of $C_g$, where the edge $\{s_1,s_2\}$ in $C_g$ is in $E_{3,3}$. As before, the edge $\{s_1,s_2\}$ can be chosen in $|E(G_n)|$ ways. Once this edge is chosen, the edges $\{s_2,s_3\}$ and $\{s_1,s_{2L}\}$ of $C_g$ can each be chosen in $\frac2{K_2}\sqrt{|E(G_n)|}$ ways, since $s_1,s_2\in V_3$. Now, for each $2\leq \ell \leq L-1$, the edge $\{s_{2\ell},s_{2\ell+1}\}$ of $ C_g$ has $|E(G_n)|$ choices. Once these edges are chosen the cycle is completely determined. Hence, 
$$M( C_g, G_n) \lesssim \frac1{K_2^2} |E(G_n)|^L.$$ 
This completes the proof of the observation.  
\end{proof}
To complete the proof of Lemma \ref{lm:H2} chose any $H = \mathcal{H}(\bm u)$, with $\bm u \in \mathcal U^{(2)}_{a, b, c, d}\cap \mathcal{U}^{(0)}_{a,b,c,d}$, with connected components $H_1, H_2, \ldots, H_\nu$, such that $H_1$ is a cycle with at least one edge mapped to $E_{3,3}$. 
	Then using  \eqref{eq:alon} and Lemma \ref{lm:Hgamma} (4) we have
	\begin{align*} 
	M(H, G_n) & \lesssim_{a, b, c, d} M(H_1, G_n)  \prod_{i=2}^\nu N\Big((H_i)_S,G_n\Big)\nonumber\\
	&\leq M(H_1, G_n) |E(G_n)|^{ \sum_{i=2}^ \nu \gamma(H_i)} \nonumber \\ 
	 & \leq M(H_1, G_n) |E(G_n)|^{ \frac{1}{2} \sum_{i=2}^ \nu |E(H_i)|} \nonumber \\ 
	 	 & \lesssim_{a, b, c, d} \frac{1}{K_2} |E(G_n)|^{ \frac{1}{2} \sum_{i=1}^ \nu |E(H_i)|}  = \frac{1}{K_2} |E(G_n)|^{(a+b+c+d)/2},
	\end{align*}
	where the last inequality uses Observation \ref{obs:C}.
	Hence, for $\mathcal H_{a, b, c, d}$ as defined before Lemma \ref{lm:H1}, we have
	$$|\mathcal{U}_{a, b, c, d}^{(2)}\cap \mathcal{U}_{a, b, c, d}^{(0)}| \leq \sum_{H \in \mathcal H_{a, b, c, d}} M(H, G_n) \bm 1\{H \cong \mathcal{H}(\bm u) , \bm u \in \mathcal{U}_{a, b, c, d}^{(2)}\cap  \mathcal{U}_{a, b, c, d}^{(0)}\} \lesssim_{a, b, c , d}  \frac{1}{K_2} |E(G_n)|^{(a+b+c+d)/2},$$
	since $|\mathcal H_{a, b, c, d}| \lesssim_{a, b, c , d} 1$.
	%
	%
	%
	%
	%
	%
	%
	%
	%
	%
	%
	%
	%
	%
\end{proof}
\begin{lemma}\label{lm:H3} 
	$|\mathcal{U}_{a, b, c, d}^{(3)}\cap \mathcal{U}^{(0)}_{a, b, c, d}| \lesssim_{a, b, c, d}\Big(o(1)+\frac{1}{K_2}\Big) |E(G_n)|^{(a+b+c+d)/2}$. 
\end{lemma}
For proving Lemma \ref{lm:H3}, we first prove the following proposition, which deals with the special case when $H=\mathcal{H}(\bm u)$ is connected.
\begin{proposition}\label{prop:combine}
Let $H = \mathcal{H}(\bm u)$ for some $\bm u \in \mathcal{U}_{a, b, c, d}^{(3)}$,  such that $H$ is connected, and $d_{\min}(H)\ge 2$. Then we have 
$$\frac{M(H,G_n)}{|E(G_n)|^{|E(H)|}}=o(1)+O\left(\frac{1}{K_2}\right).$$
Further, this estimate is uniform in $\bm u$. 
\end{proposition}
\begin{proof}
Using the definition of $\mathcal{U}^{(3)}_{a,b,c,d}$, we have $\gamma(H)=|E(H)|/2 > |V(H)|/2$, and $H$ has a 2-star with one edge in $E_{3, 3}$.  
Further, $|V(H)| \geq 3$, since $H_S$ has a 2-star with one edge in $E_{3,3}$. Let $\varphi:V(H)\rightarrow [0, 1]$ an optimal solution of \eqref{eq:gamma}.  Then it is well-known that $\varphi \in \{0, \frac{1}{2}, 1\}$ \cite[Proposition 2.1]{fractionalindependentset}. Partition $V(H)=V_0(H)\cup V_{1/2}(H)\cup V_1(H)$, where $V_a(H)=\{v\in V(H): \varphi(v)=a\}$, for $a\in \{0, 1/2, 1\}$. We will need the following lemma about the structure of the subgraphs of $H$ induced by this partition of the vertex set. 
	\begin{lemma} \cite[Lemma 9]{alon81}, \cite[Lemma 4.2]{bbb_pd_sm} Let $H$ be a multi-graph with no isolated vertex and $\gamma(H)> |V(H)|/2$. If $\varphi:V(H)\rightarrow [0, 1]$ is an optimal solution to the linear program (\ref{eq:gamma}), then the following holds:
		\begin{enumerate}
			\item The bipartite graph $H_{01}=(V_0(H)\cup V_1(H), E(H_{01}))$, where $E(H_{01})$ is the set of edges from $V_0(H)$ to $V_1(H)$, has a matching which saturates every vertex in $V_0(H)$.\footnote{A {\it matching} $M$ in a graph $H=(V(H), E(H))$ is subset of edges of $E(H)$ without common vertices. The matching $M$ is said to saturate $A\subset V(H)$, if, for every vertex $a\in A$, there exists an edge in the matching $M$ incident on $a$.} 
			\item The subgraph of $H$ induced by the vertices of $V_{1/2}(H)$ has a spanning subgraph which is a disjoint union of cycles and isolated edges.
		\end{enumerate}
		\label{lm:Hgamma12}
	\end{lemma}
	Note that $\gamma(H)> |V(H)|/2$ implies that no optimal $\varphi$ is not identically equal to 1/2. Depending on the size of $V_{1/2}(H)$ the following cases arise: 
	\begin{enumerate}
		\item $|V_{1/2}(H)|\ne 0$: Let $H_{01}$ be the graph with vertex set $V_0(H)\cup V_1(H)$ and edge set $E(H_{01})$, where $E(H_{01})$ is the set of edges from $V_0(H)$ to $V_1(H)$. Let $H_{1/2}$ be the subgraph of $H$ induced by the vertices of $V_{1/2}(H)$. 
		Decompose $H$ into subgraphs $H_{01}$ and $H_{1/2}$. By Lemma \ref{lm:Hgamma12}, $H_{01}$, has a matching which saturates every vertex in $V_0(H)$. Therefore,
		\begin{equation}
			N\Big((H_{01})_S, G_n\Big) \lesssim_{a, b, c, d} |E(G_n)|^{|V_1(H)|} \lesssim_{a, b, c, d} |E(G_n)|^{|E(H_{01})|/2},
			\label{eq:H01}
		\end{equation}
		since $d_{\min}(H)\geq 2$ implies $|E(H_{01})|\geq 2|V_1(H)|$, because $V_1(H)$ is as independent set and there is no edge from $V_1(H)$ to $V_{1/2}(H)$.  
		Moreover, the subgraph $F$ of $H$ induced by the vertices of $V_{1/2}(H)$ has a spanning subgraph which is a disjoint union of cycles and isolated edges. Denote the connected components of $F$ by $F_1, \ldots, F_\nu$.  Observe that $d_v(H) = 2$, for every $v \in V_{1/2}(H)$ by Lemma \ref{lm:Hgamma} (4). Hence, if $F_i$ is a cycle or a double edge, for some $i$, then there cannot be an edge from a vertex in $F_i$ to $V(H)\backslash V(F_i)$, which contradicts the connectedness of $H$, since $V_0(H) \cup V_1(H)$ is non-empty. Hence, $F_i$ is an isolated edge for all $1 \leq i \leq \nu$.  
		This implies, for all $1 \leq i \leq \nu$, 
		\begin{equation}
			N(F_i, G_n)\lesssim_{a, b, c, d} |E(G_n)|.
			\label{eq:H1/2}
			\end{equation} 
			Also, each of the 2 vertices in $V(F_i)$ must be connected to at least some other vertex in $H_{1/2}$ or $H_0$, which means $|E(H_{0}, H_{1/2})| + |E(H_{1/2})| \geq 3 \nu$, where $E(H_{0}, H_{1/2})$ is the set of edges between $H_{0}$ and $H_{1/2}$.
		Therefore, by \eqref{eq:H01} and \eqref{eq:H1/2}, 		
		\begin{align}\label{eq:MHGn1}
			\frac{M(H, G_n)}{|E(G_n)|^{|E(H)|/2}} & \lesssim_{a, b, c, d} \frac{N\Big((H_{01})_S, G_n\Big)\prod_{i=1}^{\nu}N(F_i, G_n)}{|E(G_n)|^{|E(H)|/2}} \nonumber\\
			& \lesssim_{a, b, c, d} \frac{|E(G_n)|^{|E(H_{01})|/2 + \nu}}{|E(G_n)|^{|E(H_{01})|/2 + |E(H_{1/2})|/2 + |E(H_0, H_{1/2})|/2}}  \nonumber\\ 
			& = \frac{|E(G_n)|^{\nu}}{|E(G_n)|^{|E(H_{1/2})|/2 + |E(H_0, H_{1/2})|/2}} = |E(G_n)|^{-\nu/2}  = o(1) ,  
		\end{align}
since $\nu \geq 1$. 
		\item $|V_{1/2}(H)| = 0$: In this case  $\gamma(H) = |V_1(H)|$. 
		By the definition of $\mathcal U^{(3)}_{a, b, c, d}$, 
				\begin{align}\label{eq:VH01}
			|E(H)|=2 \gamma(H) = 2 |V_1(H)| . 
		\end{align}
		Since every vertex in $V_1(H)$ has degree 2 (by Lemma \ref{lm:Hgamma} (4)) and $V_1(H)$ is an independent set, \eqref{eq:VH01} means there is no edge between the vertices in $V_0(H)$. This implies the graph $H$ is bipartite. Also, by assumption, at least one edge of $H$ is mapped to $E_{3, 3}$ (recall the definition of $\mathcal U^{(3)}_{a, b, c, d}$). Denote the pre-image of this edge by $\{u, v\}$. 
		 Without loss of generality assume $\varphi(u)=0$ and $\varphi(v)=1$. Note that since $d_v(H)=2$, the edge $\{u, v\}$ can be repeated at most twice. We consider two cases depending on whether the edge $(u, v)$ appears once or twice.  	
		\begin{itemize}		
			\item The edge $\{u, v\}$ appears twice in $H$. This means $d_v(H)=2$. Now, since the edge $\{u,v\}$ must be part of a 2-star, there exists $w \in V_1(H)$ such that $\{u, w\} \in E(H_{01})$ (this makes $(v,u,w)$ a 2-star with root vertex $u$). Take $S \subseteq V_0(H)\setminus \{u\}$ and let $N(S)$ be the set of vertices in $V_1(H)\setminus \{v, w\}$ that are adjacent to some vertex in $S$. Denote by $E(S, V_1(H)\setminus \{v, w\})$ the number of edges in $H$ with one endpoint in $S$ and the other endpoint in $V_1(H) \setminus \{v, w\}$. Note that there is no edge in $E(H)$ joining $v$ to a vertex in $S$ (since $d_v(H)=2$, as the edge $\{u,v\}$ appears twice, and $u\notin S$), and there can be at most one edge in $E(H)$ joining $w$ to a vertex in $S$ (because $d_w(H)=2$ and $w$ is already joined to $u$). Hence, 
			$$|E(S, V_1(H)\setminus \{v, w\})| \geq 2|S|-1.$$ 
			Also, since every vertex in $N(S)$ has degree 2 in $H$, $|E(S, V_1(H)\setminus \{v, w\})| \leq 2|N(S)|$. Therefore, $2|N(S)| \geq 2|S|-1$ which implies, $|N(S)| \geq |S|$. Hence, by Hall's marriage theorem \cite{lpmatching}, there exists a matching between $V_0(H)\setminus \{u\}$ and $V_1(H) \setminus \{v, w\}$ that saturates every vertex in $V_0(H)\setminus \{u\}$. Therefore, we can count the number of copies of $H$ in $G_n$  as follows: 
			\begin{itemize} 
				\item First choose the 2-star $(v,u,w)$ in at most $O(\frac1{K_2} |E(G_n)|^{3/2})$ ways. 
				\item Then create the matching between $V_0(H)\setminus \{u\}$ and $V_1(H) \setminus \{v, w\}$ in $$|E(G_n)|^{|V_0(H)\setminus \{u\}|}$$ ways. 
				\item Next, choose the remaining (non-matched) vertices in $V_1(H) \setminus \{v, w\}$ in $$|E(G_n)|^{|V_1(H) \setminus \{v, w\}|-|V_0(H)\setminus \{u\}|}$$ ways. 
			\end{itemize}
			 This gives, 
			\begin{align}\label{eq:MHGn2}
				M(H, G_n)  \lesssim_{a, b, c, d} \frac1{K_2} |E(G_n)|^{\frac{3}{2}+|V_1(H) \setminus \{v, w\}|}  
				& =  \frac1{K_2} |E(G_n)|^{|V_1(H)| -\frac{1}{2} } \nonumber \\ 
				& \leq \frac1{K_2} |E(G_n)|^{|E(H)|/2} , 
			\end{align}
			since, by \eqref{eq:VH01}, $|V_1(H)|= |E(H)|/2$. 			
			\item The edge $\{u, v\}$ appears only once in $E(H)$. In this case, since $\min\{d_{u}(H), d_v(H)\} \geq 2$, there exists $v' \in V_1(H)\setminus\{v\}$ and $u' \in V_0(H)\setminus \{u \}$, such that $\{u, v'\}$ and $\{v,u'\} \in E(H)$. By an exactly similar argument using Hall's marriage theorem as in the previous case, it follows that there exists a matching between $V_0(H)\setminus \{u,u'\}$ and $V_1(H) \setminus \{v, v'\}$ that saturates every vertex in $V_0(H)\setminus \{u, u'\}$. Therefore, we can count the number of copies of $H$ in $G_n$  as follows: 
			\begin{itemize}				
				\item First chose the path $(v'uvu')$ in at most $O(\frac1{K_2^2} |E(G_n)|^{2}) = O(\frac1{K_2} |E(G_n)|^{2})$ ways. 
				\item Then create the matching between $V_0(H)\setminus \{u, u'\}$ and $V_1(H) \setminus \{v, v'\}$ in $$|E(G_n)|^{|V_0(H)\setminus \{u, u'\}|}$$ ways.				
				\item Next, choose remaining (non-matched) vertices in $V_1(H) \setminus \{v, v'\}$ in $$|E(G_n)|^{|V_1(H) \setminus \{v, v'\}|-|V_0(H)\setminus \{u, u'\}|}$$ ways. 			
			\end{itemize}
			This gives, 
			\begin{align}\label{eq:MHGn3}
				M(H, G_n)  \lesssim_{a, b, c, d} \frac1{K_2} |E(G_n)|^{2+|V_1(H) \setminus \{v, v'\}|}  
				& =  \frac1{K_2} |E(G_n)|^{|V_1(H)| } \nonumber \\ 
				& = \frac1{K_2} |E(G_n)|^{|E(H)|/2} 
			\end{align}
			since, by \eqref{eq:VH01}, $|V_1(H)|= |E(H)|/2$.		
		\end{itemize} 	
	\end{enumerate} 
	Combining \eqref{eq:MHGn1}, \eqref{eq:MHGn2}, and \eqref{eq:MHGn3}, the conclusion of the proposition follows.
	\end{proof}
	\begin{proof}[Proof of Lemma \ref{lm:H3}] 
		For $H = \mathcal{H}(\bm u)$ with $\bm u \in \mathcal{U}_{a, b, c, d}^{(3)}\cap \mathcal{U}_{a,b,c,d}^{(0)}$, consider the connected component $H_j$, for $1 \leq j \leq n$, such that $\gamma(H_j)=|E(H_j)|/2 > |V(H_j)|/2$ and $(H_j)_S$ has a 2-star with one edge in $E_{3, 3}$ (this exists by definition of $\mathcal{U}^{(3)}_{a,b,c,d}$). Without loss of generality, suppose $j=1$. 
Let $\bm u'$ denote the subset of vertices in $\bm u$ which spans the component $H_1$. Let $a',b',c',d'$ be, respectively,
\begin{itemize}
\item
the number of edges in $\bm u'$ which are in $E_{1,3}$;
\item
the number of edges in $\bm u'$ which are in $E_{2,2}$;
\item
the number of edges in $\bm u'$ which are in $E_{2,3}$;
\item
the number of edges in $\bm u'$ which are in $E_{3,3}$.
\end{itemize}
Then we claim that $\bm u'\in \mathcal{U}^{(3)}_{a',b',c',d'}$. Indeed, this follows by the construction of $H_1=\mathcal{H}(\bm u')$, which ensures that $H_1=\mathcal{H}(\bm u')$ is connected, with $\gamma(H_1)=|E(H_1)|/2>|V(H_1)|/2$, and $(H_1)_S$ has a two star with one edge in $E_{3,3}$, and so $d'>0$. Also, since $\bm u\in \mathcal{U}^{(0)}_{a,b,c,d}$, it follows using Lemma \ref{lm:Hgamma} (1) that $d_{\min}(H)\ge 2$, and so $d_{\min}(H_1)\ge 2$. Thus invoking 
Proposition \ref{prop:combine} we have 
$$\frac{M(H_1,G_n)}{|E(G_n)|^{|E(H_1)|}}=o(1)+O\Big(\frac{1}{K_2}\Big),$$
which on using \eqref{eq:alon} and Lemma \ref{lm:Hgamma} (4) gives
	\begin{align}
	M(H, G_n) & \lesssim_{a, b, c, d} M(H_1, G_n) |E(G_n)|^{ \sum_{i=2}^ \nu |\gamma(H_i)|} \nonumber \\ 
	& \leq M(H_1, G_n) |E(G_n)|^{\frac{1}{2}\sum_{i=2}^ \nu |E(H_i)|}  \nonumber \\ 
	& \lesssim_{a, b, c, d}\Big(o(1)+ \frac{1}{K_2}\Big) |E(G_n)|^{|E(H)|/2} . \nonumber  
	\end{align} 
	Hence, for $\mathcal H_{a, b, c, d}$ as defined before Lemma \ref{lm:H1}, we have
	$$|\mathcal{U}_{a, b, c, d}^{(3)}\cap \mathcal{U}_{a, b, c, d}^{(0)}| \leq \sum_{H \in \mathcal H_{a, b, c, d}} M(H, G_n) \bm 1\{H \cong \mathcal{H}(\bm u), \bm u \in \mathcal{U}_{a, b, c, d}^{(3)}\cap  \mathcal{U}_{a, b, c, d}^{(0)}\} =  \Big(o(1)+\frac{1}{K_2} \Big)|E(G_n)|^{\frac{a+b+c+d}{2}},$$
	since $|\mathcal H_{a, b, c, d}| \lesssim_{a, b, c , d} 1$, and $|E(H)|=a+b+c+d$. This completes the proof of Lemma \ref{lm:H3}.
\end{proof} 
The proof of Lemma \ref{momind} can now be completed easily. For this, recalling \eqref{eq:U} and using Lemmas \ref{lm:H1}, \ref{lm:H2}, and \ref{lm:H3} gives, 
\begin{align}
	\left|\Ex\left[U_{13}^{a}U_{22}^{b}U_{23}^{c} U_{33}^d \right]- \Ex\left[U_{13}^{a}U_{22}^{b}U_{23}^c\right]\Ex\left[U_{33}^{d}\right]  \right| 
	& \lesssim_{a, b, c, d, M} \frac1{|E(G_n)|^{\frac{a+b+c+d}2}} \sum_{s=1}^3 |\mathcal{U}^{(s)}_{a, b, c, d}\cap  \mathcal{U}_{a, b, c, d}^{(0)}|	\nonumber \\
	& \lesssim_{a, b, c, d, M} \frac{1}{K_2} + o(1) , \nonumber 
\end{align}
which goes to zero as $n \rightarrow \infty$ followed by $K_2 \rightarrow \infty$. }

\subsubsection{Proof of Lemma \ref{asymp-ind}} 
\label{sec:indpf}
In this section we will show that the joint cdf of $(U_{13,n,M},U_{22,n}+U_{23,n})$ factorize in the limit. To begin with observe that 
$$((U_{22,n}+U_{23,n})\ , \ U_{13,n,M})^{\top} \mid \bm{X}_{n},\bm{Z}_{n}^{(2)} \sim N_2\left(\bm \mu_n,\Sigma_n\right),$$ 
where
\begin{align*}
	\bm \mu_n:=\left(\tfrac1{2\sqrt{|E(G_n)|}}(\bm{Z}_{n}^{(2)})^{\top} A_{22} \bm{Z}_{n}^{(2)}\ , \ 0\right)^{\top} ,
\end{align*}
and 	
\begin{align*}
	\Sigma_n := \begin{pmatrix}   \Sigma_n^{(1,1)} &   \Sigma_n^{(1, 2)} \\  \Sigma_n^{(1, 2)} &  \Sigma_n^{(2,2)} 
	\end{pmatrix}
	:=\frac1{|E(G_n)|}\begin{pmatrix} (\bm{Z}_{n}^{(2)})^{\top}A_{23}A_{32}\bm{Z}_{n}^{(2)} & (\bm{Z}_{n}^{(2)})^{\top}A_{23}A_{31}\bm{X}_{n,M}^{(1)} \\ (\bm{Z}_{n}^{(2)})^{\top}A_{23}A_{31}\bm{X}_{n,M}^{(1)} & (\bm{X}_{n,M}^{(1)})^{\top}A_{13}A_{31}\bm{X}_{n,M}^{(1)} . 
	\end{pmatrix} 
\end{align*}

Note that $\vr((\bm{Z}_{n}^{(2)})^{\top} A_{22} \bm{Z}_{n}^{(2)}) = \ind^\top  A_{22}\ind \le \eg$. Also, 
\begin{align*}
	\Ex\left|(\bm{Z}_{n}^{(2)})^\top A_{23}A_{32}\bm{Z}_{n}^{(2)}\right|  & \lesssim \sum_{\substack{u, v\in \vm, w\in \vl}} a_{u, w}a_{w, v}    \le {\eg}
\end{align*}
and 	
\begin{align*}
	\Ex\left|(\bm{X}_{n,M}^{(1)})^{\top} A_{13}A_{31}\bm{X}_{n,M}^{(1)}\right|  & \lesssim_M \sum_{\substack{u, v\in \vh, w\in \vl}} a_{u, w}a_{w, v}    \le {\eg}. 
\end{align*} 
This shows that the conditional means $\{\bm \mu_n\}_{n \geq 1}$ and  the conditional variances $\{\Sigma_n^{(1,1)}\}_{n \geq 1}$ and $\{\Sigma_n^{(2, 2)}\}_{n \geq 1}$ are tight. Next, we claim that the covariance of $\Sigma_n^{(1, 2)}$ is negligible under the double limit, that is, 
\begin{align}\label{eq:variance12}
	\lim_{K_1\to \infty}\lim_{n\to\infty}\frac1{{\eg}^2}\operatorname{Var}((\bm{Z}_{n}^{(2)})^{\top}A_{23}A_{31}\bm{X}_{n,M}^{(1)}) =0.
\end{align}
Note that \eqref{eq:variance12} implies, 
$$\frac1{\eg}(\bm{Z}_{n}^{(2)})^\top A_{23}A_{31}\bm{X}_{n,M}^{(1)} \stackrel{P}{\to} 0.$$ As $(\mu_n^{(1)},\Sigma_n^{(1,1)})$ and $(\mu_n^{(2)},\Sigma_n^{(2,2)})$ are independent, assuming \eqref{eq:variance12}, the result in Lemma \ref{asymp-ind} then follows from Lemma \ref{tech} in Appendix \ref{sec:appendix}.

The remainder of the proof is devoted in proving \eqref{eq:variance12}. Note that
\begin{align} 
	\operatorname{Var}((\bm{Z}_{n}^{(2)})^{\top}A_{23}A_{31}\bm{X}_{n,M}^{(1)}) & = \sum_{u \in \vh, v \in \vm} \left(\sum_{w \in \vl} a_{u, w}a_{w, v}\right)^2 \nonumber \\ 
	& \le K_1\sum_{v \in \vm} d_v^2 \tag*{(since $|V_1| \leq K_1$)} \nonumber \\ 
	& \le K_1\sum_{v =1}^n d_v^2\ind\set{d_v \le \tfrac{2 \eg}{K_1}}, 
	\label{eq:varZn} 
\end{align}	 
where the last step uses \eqref{eq:Kdegree}. Now, for every integer $v \geq 1$ fixed, denote by 
\begin{align}\label{eq:fnv}
	f_{n, K_1}(v):= \frac{K_1d_{v}^2}{{\eg}^2}\ind \left\{\frac{d_{v}}{\eg} \le \frac{2}{K_1} \right\}. 
\end{align} 
Since, the sequence  $\{\frac{d_v}{\eg}\}_{n\ge 1}$ is bounded by 1 for every $n, v \geq 1$, passing through a subsequence (which we also index by $n$ for notational convenience) without loss of generality we can assume $$\lim_{n \rightarrow \infty} \frac{d_v}{|E(G_{n})|}=:\lambda(v) \quad \mbox{ and } \quad \lim_{n \rightarrow \infty} \ind\left\{\frac{d_v}{|E(G_{n})|}\le \frac{2}{K_1} \right\}=:\eta_{K_1}(v)$$ exists simultaneously for all $v \geq 1, K_1\ge 1$.  Thus for all $K_1\in \Z_{>0}$  and for all $v\ge 1$,
\begin{align}\label{eq:fnKv}
	\lim_{n \rightarrow \infty} f_{n, K_1}(v) = K_1\lambda(v)^2\eta_{K_1}(v).
\end{align}
Furthermore, since the vertices of $G_n$ are ordered according to the degrees (recall Definition \ref{defg}), for any $v \geq 1$, 
$$\frac{d_v}{\eg} \le \frac1{v\eg}(d_1+d_2+\cdots+d_v) \le \frac{2}{v}.$$ 
Thus $f_{n,K_1}(v) \le \frac{4K_1}{v^2}$  which is summable over $v$. Hence, by the Dominated Convergence theorem,
\begin{align*}
	\lim_{n\to\infty} \sum_{v=1}^n f_{n,K_1}(v) = K_1 \sum_{v=1}^{\infty} \lambda(v)^2\eta_{K_1}(v).
\end{align*}
Moreover, from the definition of $\eta_{K_1}(v)$, it is clear that $\eta_{K_1}(v)=0$ or $1$, and in the latter case $\lambda(v) \le \frac{2}{K_1}$. Thus, $$\lim_{K_1\to \infty} K_1 \lambda(v)^2\eta_{K_1}(v) \lesssim \lim_{K_1\to \infty} \frac{1}{K_1} = 0,$$ $K_1 \lambda(v)^2 \eta_{K_1}(v) \le 2\lambda(v)$, and 
$$\sum_{v=1}^{\infty}\lambda(v) = \lim_{r\to\infty} \sum_{v=1}^{r}\lambda(v)  = \lim_{r\to\infty}\lim_{n\to\infty} \sum_{v=1}^r \frac{d_v}{\eg} \le 2.$$
Therefore, by another application of Dominated Convergence Theorem and recalling \eqref{eq:varZn}, \eqref{eq:fnv}, and \eqref{eq:fnKv} gives, 
\begin{align*} 
	\lim_{K_1\to \infty}\lim_{n\to\infty}\frac1{{\eg}^2}\operatorname{Var}((\bm{Z}_{n}^{(2)})^{\top}A_{23}A_{31}\bm{X}_{n,M}^{(1)}) & \leq \lim_{K_1\to \infty} \lim_{n \rightarrow \infty} \sum_{v=1}^n f_{n,K_1}(v) \nonumber \\ 
	& = \lim_{K_1\to \infty} K_1 \sum_{v=1}^{\infty} \lambda(v)^2\eta_{K_1}(v)=0. 
\end{align*}  
This completes the proof of \eqref{eq:variance12}.


\subsection{Distributional Limit of $U_{13,n,M}$} \label{sec:limits1} 

Having established the asymptotic independence of $U_{13,n,M}$ and $U_{22,n}+U_{23,n}+U_{33,n}$, we now proceed to derive the distributional limit of $U_{13,n,M}$. We begin with the following general lemma: 

\begin{lemma}\label{l1lim} Let $Y_1,Y_2,\ldots$ be i.i.d.~mean $0$ variance $1$ random variables and $\Sigma=((\sigma_{st}))$ be as in Assumption \ref{asum1}. For $K \geq 1$, denote by $\bm Y_K = (Y_1, Y_2, \ldots, Y_K)^\top$ and $\Sigma_K = (\sigma_{st})_{1 \leq s, t \leq K}$. Then,  the sequence of non-negative random variables 
	\begin{align*}
		W_K := \bm{Y}_{K}^\top\Sigma_K \bm Y_{K} = \sum_{1\le s, t \le K} \sigma_{s, t} Y_s Y_t 
	\end{align*}
	converges in $L^1$, as $K \rightarrow \infty$, to some random variable $W_{\infty}:=\bm{Y}_{\infty}^\top\Sigma \bm Y_{\infty} := \sum_{1\le u, v \le \infty} \sigma_{s, t} Y_s Y_t 
	$. 
	\end{lemma}
	
	\begin{proof} The non-negativity of $W_K$ follows from the definition of $\sigma_{st}$ in \eqref{as:point}. Indeed, observe that for each $K \geq 1$, 
		\begin{align*}
			W_K:=\sum_{1\le s, t \le K} \sigma_{st} Y_s Y_t = \lim_{n\to\infty} \sum_{1\le s, t \le K} \sum_{v=1}^n a_{s, v}a_{v, t}Y_s Y_t = \lim_{n\to\infty} \sum_{v=1}^n\left(\sum_{s=1}^K a_{s, v} Y_s \right)^2 \ge 0 .
		\end{align*}
		We now show $\{W_K\}_{K \geq 1}$ is Cauchy in $L^1$. Towards this, observe that by Fatou's Lemma
		\begin{align}\label{eq:sumbd}
			\sum_{s=1}^\infty \sigma_{ss} \le \liminf_{n\to\infty} \frac1{\eg} \sum_{s=1}^n\sum_{v=1}^n a_{s, v} \le 2, 
		\end{align}
		and
		\begin{align}\label{eq:sqbd}
			\sum_{1\le s, t<\infty} \sigma_{st}^2 & \le \liminf_{n\to \infty} \sum_{1\le s, t \le n} \frac1{{\eg}^2}\sum_{1 \leq v, v' \leq n} a_{s, v}a_{s, v'}a_{v, t}a_{v', t} \nonumber \\ 
			& \le \liminf_{n\to \infty} \sum_{1\le s, t \le n} \frac1{{\eg}^2}\sum_{1 \leq v, v' \leq n} a_{s, v}a_{v', t} \nonumber \\ 
			& \le 4.
		\end{align}
		Thus, using the fact that $Y_1, Y_2, \ldots, $ have zero mean and variance $1$ we get, 
		\begin{align*}
			\Ex\left|W_{K+K'}-W_K\right| & \le  \sum_{s=K+1}^{K+K'}\sigma_{ss}\Ex [Y_s^2] + \sqrt{\Ex\bigg[\bigg(\sum_{1\le s \neq t \le K+K'} \sigma_{st} Y_s Y_t -\sum_{1\le s \neq t \le K} \sigma_{st} Y_s Y_t \bigg)^2\bigg]} \\ & 
			= \sum_{s=K+1}^{K+K'}\sigma_{ss} + \sqrt{\sum_{1\le s \neq t \le K+K'} \sigma_{st}^2 -\sum_{1\le s \neq t \le K}\sigma_{st}^2},
		\end{align*}
		which goes to zero as $K' \to \infty$ followed by $K \to \infty$. Thus $W_{K}$ is Cauchy in $L^1$. As $L^1$ is complete, $W_K$ converges to some $W_{\infty}\in L^1$ which denote as $\bm{Y}_{\infty}^{\top}\Sigma \bm{Y}_{\infty}$. 
	\end{proof}

	Using the above lemma we can now derive in the following proposition the limit of $U_{13,n,M}$ (as defined in \eqref{eq:U23}) under Assumptions \ref{asum1} and \ref{asum2}. The limit turns out to be normal random variable with random variance, where the random variance is an infinite dimensional quadratic form.

	\begin{proposition}\label{p:lim1} Fix $M$ large enough such that $U_{13,n,M}$  is well defined. Consider $\bm{X}_{\infty, M}=(X_{1, M},X_{2, M},\ldots)^{\top}$ where $\{X_{u, M}\}_{u\geq 1}$ is the i.i.d.~truncated sequence as defined in \eqref{def:trunc}. Then as $n\to \infty$, followed by $K_1,K_2\to \infty$, 
		$$U_{13,n,M} \to Q_{1, M}$$ 
		in distribution and in all moments, where $Q_{1, M}\sim N(0,(\bm{X}_{\infty, M})^\top\Sigma\bm{X}_{\infty, M} ) )$, with $\Sigma$ defined in Assumption \ref{asum1}. Furthermore, the moment generating function (mgf) of $Q_{1, M}$ exists in an open interval containing zero. 
	\end{proposition}

	\begin{proof}[Proof of Proposition \ref{p:lim1}]  Recalling the definition of  $U_{13,n,M}$ from \eqref{eq:U23}, note that 
		\begin{align*}
			U_{13,n,M}\mid \bm{X}_{\infty, M} \sim N\left(0,{V}_{n,K_1,K_2,M}\right), \text{ where }  {V}_{n,K_1,K_2,M}:=\frac1{\eg}(\bm{X}_{n, M}^{(1)})^\top A_{13}A_{31}\bm{X}_{n, M}^{(1)}.
		\end{align*}
		Fix $K_1,K_2\ge 1$. Note that as $n \rightarrow \infty$,	
		\begin{align}\label{eq:V12M}
			{V}_{n,K_1,K_2,M} & = \sum_{1\le u, v \le K_1} X_{u, M}X_{v,M}\frac1{\eg}\sum_{w \in \vl} a_{u, w}a_{w, v} \nonumber \\
			& \stackrel{a.s.}\to \sum_{1\le u, v \le K_1} \sigma_{uv} X_{u, M}X_{v,M} := V_{K_1,M},  
		\end{align}
		where $\sigma_{uv}$ is defined in \eqref{as:point} and using the observation 
		$$\limsup_{n \rightarrow \infty }\frac{1}{|E(G_n)|}\sum_{w \in \vh\cup\vm} a_{u, w}a_{w, v} \le \limsup_{n \rightarrow \infty } \frac{K_2}{\sqrt{\eg}} = 0,$$ 
		since $|V_1 \cup V_2| \leq K_2 \sqrt{|E(G_n)|}$	(recall \eqref{eq:V123}). Moreover, since $$|{V}_{n,K_1,K_2,M} | \leq W:=  \sum_{1\le u, v \le K_1} |X_{u, M}||X_{v,M}|$$ and $\mathbb E[W] \lesssim_{K_1, M} 1$, by the Dominated Convergence Theorem the convergence in \eqref{eq:V12M} is also in $L^1$.   
		
		Next, invoking Lemma \ref{l1lim} $V_{K_1, M}$ converges to a random variables $V_M:= (\bm{X}_{\infty, M})^\top\Sigma\bm{X}_{\infty, M}$ in $L^1$ as  $K_1 \to \infty $. Therefore, as $n\to\infty$ followed by $K_1,K_2\to \infty$, $V_{n,K_1,K_2,M}$ converges to  $V_M$ in $L^1$. Thus under this iterated limit,  
		\begin{align*}
			\Ex[e^{\i t U_{13,n,M}}]= \Ex\bigg[e^{-\tfrac12t^2V_{n,K_1,K_2,M}}\bigg] \to \Ex\bigg[e^{-\tfrac12t^2V_M} \bigg].
		\end{align*}
		This establishes $U_{13,n,M} \to Q_{1, M}	$ in distribution. By \eqref{eq:momfin}, the exponential moments of $U_{13,n,M}$ are uniformly bounded in $n$, in a neighborhood of zero. Thus $U_{13,n,M} \to Q_{1, M}$ in all moments by uniform integrability. This also implies the finiteness of mgf of $Q_{1, M}$, since by Fatou's Lemma 
		\begin{align*}
			\Ex [e^{tQ_{1, M}}] \le \sup_{n,K_1,K_2\ge 1} \Ex[e^{tU_{13,n,M}}] <\infty ,
		\end{align*} 
		where the last inequality  holds for $|t|$ small enough via \eqref{eq:momfin}. 
		This completes the proof.  
	\end{proof}

	\subsection{Distributional Limit of $(U_{22,n}+U_{23,n}+U_{33,n})$} \label{sec:limits2}
	
	In this section we obtain the  distributional limit of $U_{22,n}+U_{23,n}+U_{33,n}$, where $U_{22,n}, U_{23,n}, U_{33,n}$ are defined as in  \eqref{eq:U23}. We begin with the following general result.

	\begin{lemma} \label{ccond} Suppose $r_n\to\infty$ and $\{c_{s,n}\}_{n\ge 1, 1\le s \le r_n}$ is a triangular sequence of arrays satisfying the following conditions:  
		\begin{enumerate}[label=(\alph*)]
			\item \label{c0} $|c_{1,n}|\ge |c_{2,n}| \ge \ldots \ge |c_{r_n,n}|$ for each $n\geq 1$ and $\sup_n \sum_{s=1}^{r_n} c_{s,n}^2 \le L$ for some $L>0$,
			\item \label{ca} $\lim_{n \rightarrow \infty} c_{s,n} \to c_s$, for each $s \geq 1$, 
			\item \label{cb} $\ds \lim_{K\to\infty} \limsup_{n\to\infty} \sum_{s=K+1}^{r_n} c_{s,n}^2=\lim_{K\to\infty} \liminf_{n\to\infty} \sum_{s=K+1}^{r_n} c_{s,n}^2=\rho^2$.
		\end{enumerate}
		Let $\{Y_s\}_{s \geq 1}$ be a sequence of i.i.d mean $0$ and variance $1$ random variables and $X_n:=\sum_{s=1}^{r_n} c_{s, n}Y_s$. Then the random variable $\sum_{s=1}^{\infty} c_sY_s$ is well-defined and $X_n$ converges weakly to the random variable
		$$X:=\sum_{s=1}^{\infty} c_sY_s+ \rho Z, $$
		where $Z\sim N(0,1)$ is independent from $\{Y_s\}_{s\ge 1}$. 
	\end{lemma}
	
	\begin{proof} First, note that as $\sum_{s=1}^{r_n} c_{s,n}^2 \le L$ for all $n$, by Fatou's Lemma $\sum_{s=1}^{\infty} c_s^2 \le L$. Thus, $\sum_{s=1}^{\infty} c_s Y_s$ is well-defined by Kolmogorov's three series theorem. 
		
		To establish the weak convergence we consider the following two cases: 	
		\begin{enumerate}
			
			\item $\rho=0$: Note that	\begin{align*}
				\vr\left(X_n-\sum_{s=1}^K c_{s,n}Y_s\right)= \vr\left(\sum_{s=K+1}^{r_n} c_{s,n}Y_s \right)=\sum_{s=K+1}^{r_n} c_{s,n}^2. 
			\end{align*}
			This means, as $\lim_{K\to\infty}\limsup_{n\to\infty}\sum_{s=K+1}^{r_n} c_{s, n}^2=0$, $X_n-\sum_{s=1}^K c_{s, n}Y_s \stackrel{P}\to 0$, as $n\to \infty$ followed by $K \to \infty$. However, as $n\to\infty$, $\sum_{s=1}^K c_{s, n}Y_s \to \sum_{s=1}^K c_s Y_s$ (by assumption (a) in Lemma \ref{ccond}), which converges almost surely to $\sum_{s=1}^\infty c_s Y_s$, as $K\to \infty$ (since $\sum_{s=1}^{\infty} c_s^2 \le L$). This proves the result for $\rho=0$.
			
			\item $\rho>0$: Note that by parts \ref{ca} and \ref{cb}, $\sum_{s=1}^{r_{n}} c_{s, n}^2\to \kappa :=\rho^2+\sum_{s=1}^{\infty} c_s^2$ as $n\to \infty$. In this case 
			\begin{align*}
				\lim_{n \rightarrow \infty} \max_{K+1\le s\le n} \frac{c_{s,n}^2}{\sum\limits_{s=K+1}^{r_n} c_{s,n}^2} = \lim_{n \rightarrow \infty} \frac{c_{K+1,n}^2}{\sum\limits_{s=K+1}^{r_n} c_{s,n}^2}  = \frac{c_{K+1}^2}{\kappa-\sum\limits_{s=1}^K c_s^2},
			\end{align*}
			which goes to zero as $K\to \infty$. Here, we used the fact that $\rho^2=\kappa-\sum\limits_{s=1}^{\infty} c_s^2>0$. Thus the constants $\{c_{s,n}\}_{K+1}^{r_n}$ satisfies the H\'ajek-\v{S}id\'ak condition \cite[Theorem 3.3.6]{ss} under the iterated limit. Therefore, under the double limit $n\to \infty$ followed by $K\to \infty$, 
			$\sum_{s=K+1}^{r_{n}} c_{s, n}Y_s$ converges to $\rho Z$. On the other hand, $\sum_{s=1}^K c_{s, n}Y_s$ converges to $\sum_{s=1}^{\infty} c_s Y_s$ under the iterated limit. 
		\end{enumerate}		
	\end{proof}

	We are now ready to derive the  distributional limit of $U_{22,n}+U_{23,n}+U_{33,n}$.  
	
	\begin{proposition} \label{p:lim2} Under Assumptions \ref{asum1} and \ref{asum2}, as $n\to \infty$ followed by $K_1\to\infty$ 
		\begin{align*}
			U_{22,n}+U_{23,n}+U_{33,n} \to Q_2+Q_3
		\end{align*}
		in distribution and in all moments, where $Q_2,$ and $Q_3$ are independent random variables with \begin{itemize}
			\item[--] $Q_2 \sim N(0,\rho^2)$,  where $\rho^2=1-\sum_{s=1}^{\infty} (\sigma_{ss}+\frac12\rho_s^2) \in [0, \infty)$ with $\sigma_{ss}$ and $\rho_s$ as defined in \eqref{as:point} and \eqref{as:eigen}, respectively, and 
			\item[--] $Q_3 \sim \sum_{s=1}^\infty \frac12\rho_s Y_s$, where $\{Y_s\}_{s \geq 1}$ is an i.i.d. collection of $\chi_1^2-1$ random variables. (In particular, $Q_3$ is well defined under Assumption \ref{asum2}). 
		\end{itemize} 
		Furthermore, there exists a constant $C>0$ such that
		\begin{align*}
			\sup_{|t|\le C}\Ex[e^{t(Q_2+Q_3)}] <\infty. 
		\end{align*}
	\end{proposition}

	
	\begin{proof}[Proof of Proposition \ref{p:lim2}] Following \eqref{eq:U23}, note that
		\begin{align}\label{req}
			U_{22,n}+U_{23,n}+U_{33,n} &= \frac1{2\sqrt{E(G_n)}}((\bm{Z}_{n}^{(2)})^{\top},(\bm{Z}_{n}^{(3)})^{\top})\begin{pmatrix}
				A_{22} & A_{23} \\ A_{32} & A_{33} 
			\end{pmatrix} ((\bm{Z}_{n}^{(2)})^{\top},(\bm{Z}_{n}^{(3)})^{\top})^{\top} . 
		\end{align} 
		For simplicity, let us write $K=\lfloor K_1\rfloor$ for the rest of the proof. Let $G_{n,K}$ be the induced graph on vertex set $[K+1,n]$ with adjacency matrix $$A^{(K)}:=\begin{pmatrix}
			A_{22} & A_{23} \\ A_{32} & A_{33} 
		\end{pmatrix} . $$
		By the spectral decomposition, we can write $A^{(K)}=P^{\top}\Lambda P$, where $\Lambda=\operatorname{diag}(\lambda_{n,K}^{(1)},\lambda_{n,K}^{(2)},\ldots, \lambda_{n,K}^{(n-K)})$ are the eigenvalues of $A^{(K)}$ (such that $|\lambda_{n,K}^{(1)}| \ge |\lambda_{n,K}^{(2)}| \ge \cdots \ge |\lambda_{n,K}^{(n-K)}|$) and $P$ the matrix of eigenvectors.  Note that $\tr(A^{(K)})=0$ and 
		\begin{align*}
			\tr(A^{(K)}(A^{(K)})^\top)=\sum_{s=1}^{n-K}(\lambda_{n,K}^{(s)})^2=2|E(G_{n,K})|. 
		\end{align*}  
		Let $c_{s,n}= \frac12 (E(G_n))^{-\frac{1}{2}}\lambda_{n,K}^{(s)}$. Then by \eqref{req} and the spectral decomposition, we see that $$U_{22,n}+U_{23,n}+U_{33,n}=\sum_{s=1}^{n-K} c_{s,n}Y_s,$$ where $\{Y_s\}_{\geq 1}$ are centered $\chi_1^2$ random variables. Note that $\sum_{s=1}^{n-K} c_{s,n} =0$ and $\sum_{s=1}^{n-K} c_{s,n}^2 \le 1$. Furthermore, recalling that  $A_{11}$ is the adjacency matrix corresponding to the set of vertices $[[1,K]]$, we have the following identity: For $m \geq 1$
		\begin{align*}
			\sum_{s=m+1}^{n-K} c_{s, n}^2 & = \frac1{2\eg}\left[2|E(G_{n,K})| -\sum_{s=1}^{m} (\lambda_{n,K}^{(s)})^2 \right]  \\ 
			& = \frac1{2\eg}\left[\left(2|E(G_{n})|-2\sum_{s=1}^{K} d_s+\tr(A_{11}A_{11}^{\top})\right) -\sum_{s=1}^{m} (\lambda_{n,K}^{(s)})^2 \right].  
		\end{align*}
		Then by \eqref{as:point} and \eqref{as:eigen}, as $n\to\infty$ followed by $K\to\infty$ and then $m\to\infty$ we have
		\begin{align*}
			\sum_{s=m+1}^{n-K} c_{s,n}^2 \to 1-\sum_{s=1}^{\infty} (\sigma_{ss}+ \tfrac{1}{2} \rho_s^2) .
		\end{align*}
		Thus $c_{s,n}$ satisfy conditions \ref{ca} and \ref{cb} of Lemma \ref{ccond} under the iterated limit. Hence, by Lemma \ref{ccond},  $U_{22,n}+U_{23,n}+U_{33,n} \stackrel{D} \to Q_2+Q_3 $,  where $Q_2$ and $Q_3$ are as defined in Lemma \ref{ccond}. 
		
		By \eqref{eq:momfin}, the exponential moments of $U_{22,n},U_{23,n}$, and $U_{33,n}$ are uniformly bounded in $n$, in a neighborhood of zero. 
		Thus convergence in all moments is guaranteed by uniform integrability. This also implies the finiteness of the mgf of $Q_2+Q_3$, since by Fatou's Lemma and H\"older's inequality, 
		\begin{align*} 
			\left(\Ex [e^{t (Q_2+Q_3)}]\right)^3 & \le  \sup_{n,K_1,K_2\ge 1} \left(\Ex [e^{t (U_{22,n}+U_{23,n}+U_{33,n}) }]\right)^3 \nonumber \\ 
			& \leq  \sup_{n,K_1,K_2\ge 1} \left(\Ex [e^{3t U_{22,n} }]\right) \left(\Ex [e^{3t U_{23,n} }]\right) \left(\Ex [e^{3t U_{33,n} }]\right) <\infty ,
		\end{align*} 
		where the last inequality  holds for $|t|$ small enough via \eqref{eq:momfin}. 
	\end{proof}

	\subsection{Completing the Proofs of Theorem \ref{thm:main} and Theorem \ref{thm:converse}}\label{sec:minfty}

	We now combine the results of the previous sections and complete the proof of Theorem \ref{thm:main}. One last ingredient of our proof is the existence of $M\to \infty$ limit of $Q_{1, M}$ defined in Proposition \ref{p:lim1}, which we record in Lemma \ref{mim} below. Towards this, denote 
	$$V_M:= (\bm{X}_{\infty, M})^\top\Sigma\bm{X}_{\infty, M},$$
	which is well-defined by Lemma \ref{l1lim}, and recall from Proposition \ref{p:lim1} that $Q_{1, M}\sim  N(0, V_M)$.

	\begin{lemma}\label{mim} The sequence of random variables $\{ V_M \}_{M \geq 1}$ {converges} in $L^1$ to some random variable $V:= \bm{X}_{\infty}^{\top}\Sigma\bm{X}_{\infty}$. This implies, as $M\to \infty$, 
		\begin{align*}
			Q_{1, M} \to Q_1
		\end{align*}
		in distribution, where $Q_1\sim N(0,V)$ is a well defined random variable with finite second moment.
	\end{lemma}
	
	\begin{proof}
		Note that $\Ex[e^{\i tQ_{1, M}}]=\Ex [e^{-\frac12t^2V_M}]$, since $Q_{1, M}\sim  N(0, V_M)$. Furthermore, by Lemma \ref{p:lim1}, $V_M$ is the $L^1$ limit of $V_{K_1,M}=\sum_{1\le u, v \le K_1} \sigma_{u v}X_{u, M}X_{v, M}$ (recall \eqref{eq:V12M}). We first claim that $V_{M}$ is Cauchy in $L^1$. To this end, observe that for each $M,N, K_1>0$ we have
		\begin{align*}
			V_{K_1,M+N}-V_{K_1,M} & =\sum_{u=1}^{K_1} \sigma_{uu}(X_{u,M+N}^2 - X_{u, M}^2 )+\sum_{1\le u \neq v \le K_1} \sigma_{u v} X_{v,M}(X_{u,M+N}-X_{u, M})\\ & \hspace{2cm}+\sum_{1\le u \neq v \le K_1} \sigma_{u v} X_{v, M+N}(X_{u, M+N}-X_{u, M}).
		\end{align*} 
		Thus by Fatou's Lemma followed by {the} Cauchy-Schwarz inequality, 
		\begin{align*}
			\Ex|V_{M+N}-V_{M}| & \le \liminf_{K_1 \to\infty} \Ex|V_{K_1,M+N}-V_{K_1,M}| \\ & \le \Ex \left|X_{1,M+N}^2-X_{1, M}^2\right|\sum_{u=1}^{\infty}\sigma_{u u} \\ & \hspace{1cm}+\liminf_{ K_1 \to\infty } \sqrt{\Ex\bigg[\bigg(\sum_{1\le u \neq v \le K_1 } \sigma_{u v} X_{v,M}(X_{u,M+N}-X_{u, M})\bigg)^2\bigg]} \\ & \hspace{2cm}  + \liminf_{ K_1 \to\infty }\sqrt{\Ex\bigg[\bigg(\sum_{1\le u \neq v \le K_1 } \sigma_{uv} X_{u,M+N}(X_{v,M+N}-X_{v,M})\bigg)^2\bigg]} \\ & = \Ex \left|X_{1,M+N}^2-X_{1, M}^2\right|\sum_{u=1}^{\infty}\sigma_{uu} + 2\sqrt{\Ex\bigg[(X_{1,M+N}-X_{1, M})^2\bigg]\sum_{1\le u \neq v <\infty} \sigma_{u v}^2}\\ & \le 2\Ex \left|X_{1,M+N}^2-X_{1, M}^2\right| + 2\sqrt{4\Ex\bigg[(X_{1,M+N}-X_{1, M})^2\bigg]}
		\end{align*} 
		where the last inequality follows from \eqref{eq:sumbd} and \eqref{eq:sqbd}. Thus, to conclude $V_M$ is Cauchy in $L^1$, it suffices to show 
		\begin{align}\label{eq:unic}
			\lim_{M\to \infty} \Ex\bigg[(X_{1,M+N}-X_{1, M})^2\bigg] = 0 \quad \text{ and } \quad  \lim_{M\to \infty} \Ex\left|X_{1,M+N}^2-X_{1, M}^2\right| =  0, 
		\end{align} 
		uniformly in $N\ge 0$. For this, recall the definition of the truncation from \eqref{def:trunc}.Then for the first term in \eqref{eq:unic} a simple computation shows that
		\begin{align*}
			& \Ex\bigg[(X_{1,M+N}-X_{1, M})^2\bigg] \\ & =2-2\Ex[X_{1,M+N}X_{1, M}] \\ & =2-2b_M^{-\frac{1}{2}}b_{M+N}^{-\frac{1}{2}}\Ex[(X_1\ind\{|X_1|\le M+N\}-a_{M+N})(X_1\ind\{|X_1|\le M\}-a_{M})] \\ & =2-2b_M^{-\frac{1}{2}}b_{M+N}^{-\frac{1}{2}}\left[\Ex[X_1^2\ind\{|X_1|\le M\}]-a_{M}a_{M+N}\right] \\ & =2-2b_M^{-\frac{1}{2}}b_{M+N}^{-\frac{1}{2}}\left[b_M+a_M^2-a_{M}a_{M+N}\right] .
		\end{align*} 
		Clearly, as $M\to \infty$, uniformly in $N\ge 0$, $b_M,b_{M+N} \to 1$ and $a_M,a_{M+N}\to 0$. Thus, following the above computation, 
		$$\lim_{M\to \infty} \Ex\bigg[(X_{1,M+N}-X_{1, M})^2\bigg] = 0, $$ 
		uniformly in $N\ge 0$.  
		
		For the second term in \eqref{eq:unic}, by triangle inequality we have
		\begin{align}\label{eq:T12}
			\Ex \left|X_{1,M+N}^2-X_{1, M}^2\right| & \le  T_1 + T_2 , 
		\end{align}
		where 
		$$T_1 := b_M^{-1}\Ex\left| (X_1\ind\{|X_1|\le M+N\}-a_{M+N})^2-(X_1\ind\{|X_1|\le M\}-a_M)^2\right| $$
		and 
		$$T_2 := \left|1- \frac{b_{M+N}}{b_{M}}\right|\Ex [X_{1,M+N}^2].$$ 
		Clearly, $T_2 \rightarrow 0$, as $M\to\infty$, uniformly in $N\ge 0$. For $T_1$ using the identity $x^2-y^2=(x+y)(x-y)$ and applying Cauchy-Schwarz inequality multiple times gives 
		\begin{align*}
			T_1 
			& \le b_M^{-1}\sqrt{2 \left(\Ex[X_1^2\ind \{M<|X_1|\le M+N\}]+(a_M-a_{M+N})^2 \right)} \\ & \hspace{2cm} \cdot \sqrt{3\left(\Ex[X_1^2\ind\{|X_1|\le M+N\}]+\Ex[X_1^2\ind\{|X_1|\le M\}]+ (a_M+a_{M+N})^2\right)} 
		\end{align*}
		Observe that the term inside the first square root above goes to zero as $M\to\infty$, uniformly in $N\ge 0$. The rest of the factors are bounded for all $M\ge M_0, N\ge 0$ for some $M_0$ (so that $b_M^{-1}$ is well defined for $M\ge M_0$). This shows  (recall \eqref{eq:T12})
		$$\lim_{M\to \infty} \Ex \left|X_{1,M+N}^2-X_{1, M}^2\right| =0 ,$$ uniformly in $N\ge 0$. Thus, $V_M$ is Cauchy in $L^1$. Hence $V_{M} \to V$ in $L^1$ where $V:=\bm{X}_{\infty}^{\top}\Sigma\bm{X}_{\infty}$. This implies, $\Ex[e^{-t^2V_M/2}] \to \Ex[e^{-t^2V/2}]$ for each $t>0$. This shows $Q_{1, M} \to Q_1$ in distribution.
	\end{proof}
	
	\begin{proof}[Completing the Proof of Theorem \ref{thm:main}] 
		From Proposition \ref{p:lim1} and \ref{p:lim2} we know $U_{13,n,M}$ and $U_{22,n}+U_{23,n}+U_{33,n}$ converges to $Q_{1, M}$ and $Q_2+Q_3$ in all moments, for every fixed $M$ large enough. Thus, appealing to  Proposition \ref{p:momind}, we see that for positive integers $a, b$,  
		\begin{align*}
			\lim_{K_1, K_2\to \infty}\limsup_{n\to\infty}\Ex[U_{13,n,M}^a(U_{22,n}+U_{23,n}+U_{33,n})^b]=\Ex[(Q_{1, M})^a(Q_2+Q_3)^b].
		\end{align*}
		As $Q_{1, M}$ and $Q_2+Q_3$ have finite mgfs in an interval containing zero (Proposition \ref{p:lim1} and \ref{p:lim2}), moments of $Q_{1, M}$ and $Q_2+Q_3$ uniquely determine their distribution. Thus, as $n\to\infty$, followed by $K_1, K_2\to \infty$, $U_{n, M}$ converges weakly (and in all moments) to $Q_{1, M}+Q_2+Q_3$.  Then using Lemma \ref{lm:U1112} and Proposition \ref{p:norm} we get 
		\begin{align}\label{eq:trunc1}
			S_{G_n}\big(\bm{X}_{n,M}\big) \stackrel{D} \rightarrow Q_{1, M}+Q_2+Q_3,
		\end{align}
		as $n\to\infty$. Finally, since as $M\to\infty$, $Q_{1, M} \to Q_1$  weakly (Lemma \ref{mim}) and $Q_{1, M}$ is independent of $Q_2+Q_3$, using Lemma \ref{l:trunc} we conclude 
		\begin{align}\label{eq:trunc2}
			S_{G_n}\big(\bm{X}_{n}\big) \stackrel{D} \rightarrow Q_1+Q_2+Q_3,
		\end{align}
		where $Q_1, Q_2, Q_3$ are as defined in Theorem \ref{thm:main}. 
	\end{proof}

	\begin{proof}[Completing the Proof of Theorem \ref{thm:converse}] Suppose $S_{G_n}(\bm{X}_n)$ converges weakly to some random variable $Q'$. We will show that $Q'\stackrel{D}{=}Q$, where $Q$ is defined in \eqref{eq:main}. For this note that for every $n$, the infinite matrix $$\left(\left(\frac{1}{|E(G_n)|}\sum_{v=1}^na_{s,v}a_{v,t}\right)\right)_{s,t\ge 1}$$ has entries in $[0,1]$, as 
		$\sum_{v=1}^na_{s,v}a_{v,t}\le \min\{d_s, d_t\} \le |E(G_n)|$.  Similarly, the entries of the infinite vector $$\left(\frac{1}{\sqrt{|E(G_n)|}} \lambda_{n,K}^{(s)}\right)_{s\ge 1}$$ take values in $[-\sqrt{2},\sqrt{2}]$, as $|\lambda_{n,K}^{(s)}|\le \sum_{s=1}^{n-K}(\lambda_{s,K}^{(s)})^2=2|E(G_{n,K})|\le 2|E(G_n)|$. 
		Since $[0,1]^{\mathbb{N}\times \mathbb{N}}\times [-\sqrt{2},\sqrt{2}]^\mathbb{N}$ is compact under the product topology, by Tychonoff's theorem, it follows that there is a subsequence along which, 
		$$\frac{1}{|E(G_n)|}\sum_{v=1}^na_{s,v}a_{v,t}\to \sigma_{s,t},\quad  \text{and} \quad \frac{1}{\sqrt{|E(G_n)|}} \lambda_{n,K}^{(s)}\to \rho_{K,s}, $$
		for every $s, t\ge 1$, for some $\sigma_{s,t}\in [0,1]$ and $\rho_{K,s}\in [-\sqrt{2},\sqrt{2}]$. By another application of Tychonoff's theorem, there is a subsequence in $K$ along which, 
		$$\lim_{K\to\infty}\rho_{K,s}=\rho_s,$$
		for all $s\ge 1$, for some $\rho_s\in [-\sqrt{2},\sqrt{2}]$. Thus, both Assumptions  \ref{asum1} and \ref{asum2} hold along a subsequence in $n$, and a subsequence in $K$. Thus, by Theorem \ref{thm:main}, along this subsequence we have $S_{G_{n}}(\bm{X}_{n})\stackrel{D}{\to} Q$ where $Q$ is defined in \eqref{eq:main}. This implies, $Q'\stackrel{D}{=}Q$, thus completing the proof.
	\end{proof}
	
	\section{Proofs of Theorem \ref{thm:moment4} and Proposition \ref{ppn:nrad}}
	\label{sec:pfmoment4}
	
	Note that Theorem \ref{thm:moment4} follows directly from Propostion \ref{ppn:nrad}. Hence, it suffices to prove Proposition \ref{ppn:nrad}. The proof is presented over two sections: In Section \ref{sec:pfnrad1} we consider the case where $F$ is not Rademacher. The Rademacher case is  in Section \ref{sec:pfnrad2}. 
	
	\subsection{Proof of Proposition \ref{ppn:nrad}~(1)}
	\label{sec:pfnrad1}
	
	As $X_1$ does not follow the Rademacher distribution, $\vr[X_1^2]>0$. Thus,  for large enough $M$ we have $\vr[X_{1, M}^2]>0$ as well. Hereafter, we will fix such a $M$. 
	First we will show that conditions \ref{ba} and \ref{eq:4m}  are equivalent.   Using the definition of $S_{G_n}$ from \eqref{def:SG_n} we have \begin{align}\label{eq:expand}
		\Ex[(S_{G_n}(\bm{X}_{n, M}))^4]=\frac1{\eg^2}\sum_{(u_1, v_1),(u_2, v_2), (u_3, v_3), (u_4, v_4)\in E(G_n) }\Ex\left[\prod_{\ell=1}^4 X_{u_{\ell}, M}X_{v_{\ell}, M}\right] .
	\end{align}
	Now, consider the multigraph formed by the edges $(u_1, v_1),(u_2, v_2), (u_3, v_3), (u_4, v_4)$. Note that whenever a vertex appears only once in the multigraph, the corresponding expectation is zero. 	
	\begin{figure*}[h]
		\centering
		\begin{minipage}[c]{1.0\textwidth}
			\centering
			\includegraphics[width=5.25in]
			{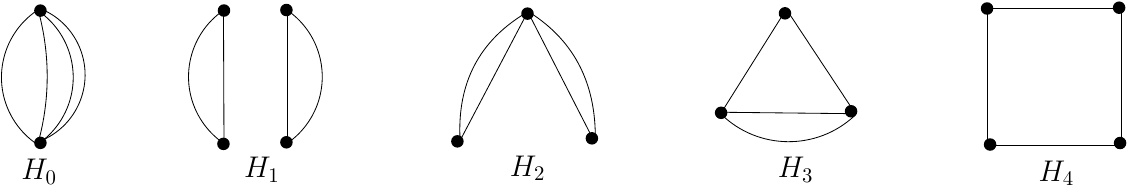}\\
		\end{minipage}
		\caption{\small{The graphs which contribute to the fourth moment.}}
		\label{fig:momentgraphs}
	\end{figure*}
	Hence, the only multigraphs which have a non-zero contribution to the sum in the RHS of \eqref{eq:expand} are the ones shown in Figure \ref{fig:momentgraphs}. 
	Note that, for $H_0$ and $H_3$ as in Figure \ref{fig:momentgraphs}, 
	$$N(H_0, G_n) \lesssim N(K_2,G_n)=\eg$$ and $$N(H_3, G_n) \lesssim N(\triangle,G_n) \lesssim \eg^{3/2} . $$ 
	Also, for the multigraphs $H_1$ and $H_2$, one needs to divide the 4 edges into 2 groups which can be done in ${4 \choose 2} =6$ possible ways. Thus, the RHS of \eqref{eq:expand} simplifies to
	\begin{align*}
		\Ex[S_{G_n}(\bm{X}_{n, M})^4] = 6 \frac{N(K_2 \bigsqcup K_2,G_n)}{\eg^2} +6\frac{N(K_{1,2},G_n)}{\eg^2}\Ex [X_{1, M}^4]+\frac{N(C_4,G_n)}{\eg^2}+o(1) , 
	\end{align*}
	where $K_2 \bigsqcup K_2$ denotes the disjoint union of two edges. (Note that $K_2 \bigsqcup K_2$ is the simple graph underlying $H_1$, $K_{1, 2}$ is the simple graph underlying $H_2$, and $H_4=C_4$ is the 4-cycle.)  Clearly, as $N(K_2 \bigsqcup K_2, G_n)+N(K_{1,2},G_n)=\binom{\eg}{2}$.  Therefore, as $n \to\infty$, 
	\begin{align}\label{eq:nneg0}
		\Ex[S_{G_n}(\bm{X}_{n, M})^4]  = 3+6\frac{N(K_{1,2},G_n)}{\eg^2}\vr[X_{1, M}^2 ]+\frac{N(C_4,G_n)}{\eg^2}+o(1).
	\end{align}
	It is well known that the number of homomorphisms of the 4-cycle in a graph can be expressed as the 
	trace of the $4$-th power of its adjacency matrix (see Example 5.11 in \cite{lov}). This implies, 
	\begin{align} \label{eq:4con} 
		\frac1{\eg^2}\left[ N(K_{1,2},G_n)+N(C_4,G_n)\right] & \asymp \frac1{\eg^2}\left[N(K_{2},G_n)+N(K_{1,2},G_n)+N(C_4,G_n)\right] \nonumber \\ 
		& \asymp	 \sum_{u=1}^{n} c_{u, n}^4 ,  
	\end{align} 
	where $c_{u, n}=\frac{1}{\sqrt{\eg}} \lambda_{u,n}$, for $1 \leq u \leq n$.  Hence, if \ref{ba} holds, then since $$\sum_{u=1}^{n} c_{u, n}^4 \le \max_{1 \leq u \leq n} c_{u, n}^2 \sum_{u=1}^{n} c_{u, n}^2 \le 2\max_{1 \leq u \leq n} c_{u, n}^2 , $$ 
	the RHS of \eqref{eq:4con} goes to zero. Therefore, via \eqref{eq:nneg0}, $\Ex[ S_{G_n}(\bm{X}_{n, M})^4]\to 3$. This shows \ref{ba} implies \ref{eq:4m}. 	Next, suppose \ref{eq:4m} holds. Then, as $\vr[X_{1, M}^2]>0$, by \eqref{eq:nneg0} we must have 
	$$N(K_{1,2},G_n)=o(\eg^2) \quad \text{ and } \quad N(C_4,G_n)=o(\eg^2).$$ 
	Hence, the RHS ~of \eqref{eq:4con} to goes to zero as $n \rightarrow \infty$ and, since $\max_{1 \leq u \leq n} |c_{u, n}| \leq \sum_{u=1}^{n} c_{u, n}^4$, \ref{ba} holds. 
	
	Next, we show \ref{ba} implies \ref{bc}. We denote by $A= A(G_n)$ the adjacency matrix of the graph $G_n$. For this note that for each $1 \leq u \leq n$ and $\bm e_u$ denoting the $u$-th basis vector, 
	\begin{align*}
		\frac{d_u}{\eg}=\frac1{\eg}\sum_{v =1}^n a_{u, v}  = \frac1{\eg} \bm{e}_u^{\top} A^\top A \bm{e}_u    & \le \frac1{\eg}\max_{\norm{\bm{w}}=1} (\bm{w}^{\top}A^2 \bm{w}) \\ 
		& = \max_{1\le u \le n} c_{u, n}^2 \to 0.
	\end{align*}
	Thus, for each $u, v \in V(G_n)$, by Cauchy Schwarz inequality, 
	\begin{align}\label{eq:cs}
		\frac1{\eg}\sum_{w=1}^n a_{u, w}a_{w, v} \le \frac{1}{\eg}\sqrt{\sum_{w=1}^n a_{u, w} \sum_{w=1}^n a_{w, v}} =  \frac{1}{\eg}\sqrt{ d_u d_v } \to 0.
	\end{align}
	This implies, \eqref{as:point} holds with $\sigma_{st}=0$ for all $s, t \geq 1$. 
	Also, under \ref{ba}, using the eigenvalue interlacing theorem \cite[Corollary 2.5.2]{eigenvalues_book}, it follows that for all $s \geq 1$, 
	$$\frac{1}{\sqrt{\eg}} |\lambda_{n, K}^{(s)}|  \leq \max_{1 \leq u \leq n} |c_{u, n}| \rightarrow 0.$$
	This implies \eqref{as:eigen}  holds with $\rho_s=0$, for all $s \geq 1$. Hence, by Theorem \ref{thm:main}, $S_{G_n}(X_n) \to Q_2 \sim N(0,1)$. This establishes \ref{bc}.

	Finally, we show that \ref{bc} implies \ref{eq:4m}. As in the proof of Theorem  \ref{thm:converse}, there exists a subsequence $\{n_m\}_{m \geq 1}$ such that \eqref{as:point} holds for some matrix $\Sigma$ along this subsequence. { Now, since  \ref{bc} holds, the variance of $Q_1$ must be constant. } This implies, $\bm X_{\infty}^\top \Sigma \bm X_{\infty}$, defined via Lemma \ref{l1lim}, has to be constant. Setting $\bm X_{-1}=(X_2,X_3,\ldots)$ note that 
	\begin{align}\label{eq:cbpf}
		\bm X_{\infty}^\top \Sigma \bm X_{\infty}=\sigma_{11}X_1^2+ X_1f(\bm{X}_{-1})+g(\bm{X}_{-1}),
	\end{align}
	for some functions $f$ and $g$. {If $\sigma_{11} > 0$ and the  support of $X_1$ has at least three distinct points, then given $\bm{X}_{-1}$, the random variable in the RHS of \eqref{eq:cbpf} can not be a constant (since a quadratic polynomial has at most two roots).} On the other hand, if support of $X_1$ has at most two points, under the assumption $\Ex[X_1]=0$ and $\vr[X_1]=1$, it forces $X_1$ to be Rademacher which we have ruled out in this case. Therefore, the random variable in the RHS of \eqref{eq:cbpf} conditioned on $\bm{X}_{-1}$ is not constant, whenever $\sigma_{11} > 0$. Thus, we can assume that $\sigma_{11} =0$. In fact, by symmetry, this implies $\sigma_{ss}=0$, for all $s \geq 1$. Now, applying the Cauchy-Schwarz inequality as in \eqref{eq:cs} shows $\sigma_{st}=0$, for all $s, t \geq 1$. Thus, the random variable $Q_{1, M}$ defined in Proposition \ref{p:lim1} is zero. Therefore, by \eqref{eq:trunc1} and \eqref{eq:trunc2} we see that for large enough $M>0$,  both $S_{G_n}(\bm{X}_{n, M})$ and $S_{G_n}(\bm{X}_{n})$ converges to $Q_2+Q_3$ as defined in Theorem \ref{thm:main}. On the other hand, by the hypothesis \ref{bc}, we have $Q_2+Q_3 \stackrel{D}{=} N(0,1)$. Denote $\tilde A = \frac1{2\sqrt{\eg}}A(G_n)$ and recall that 
	\begin{align}\label{eq:SGnXM}
		S_{G_n}(\bm{X}_{n, M}) := \bm{X}_{n,M}^{\top} \tilde A \bm{X}_{n,M}.
	\end{align} 
	Note that for each fixed $M$, the truncated random variables are centered, bounded, and hence, sub-Gaussian, and the matrix $\tilde A$ satisfies the conditions of  Lemma \ref{l:mgf}. Thus by Lemma \ref{l:mgf}, for each fixed $M>0$, the mgf of $S_{G_n}(\bm{X}_{n, M})$ is uniformly bounded in $n$, in a neighborhood of zero. This implies, $S_{G_n}(\bm{X}_{n, M})$ converges in moments to $N(0,1)$, which proves \ref{eq:4m}. 
	
	Finally, if one assumes $\Ex[X_1^4]<\infty$, to show that \ref{ba}, \ref{eq:4m}, and \ref{bc}, is equivalent to $\Ex[(S_{G_n}(\bm{X}_{n}))^4] \rightarrow 3$, it suffices to show that 
	\begin{align*}
		\lim_{M\to\infty}\lim_{n\to \infty} \left[\Ex[S_{G_n}(\bm{X}_{n})^4]-\Ex[S_{G_n}(\bm{X}_{n,M})^4]\right]=0. 
	\end{align*}
	To  this end, by calculations similar to \eqref{eq:nneg0} gives
	\begin{align*}
		\Ex[S_{G_n}(\bm{X}_{n})^4]  = 3+6\frac{N(K_{1,2},G_n)}{\eg^2}\vr[X_{1}^2 ]+\frac{N(C_4,G_n)}{\eg^2}+o(1).
	\end{align*}
	Comparing the above display with \eqref{eq:nneg0}, it suffices to show that $\vr[X_{1,M}^2] \to \vr[X_1^2]$, as $M\to\infty$. This follows on using Dominated Convergence Theorem on letting $M\to\infty$, as $\Ex[X_1^4]<\infty$.

	\subsection{Proof of Proposition \ref{ppn:nrad}~(2)}
	\label{sec:pfnrad2}
	
	As $F$ is Rademacher and hence, sub-Gaussian, and the matrix $\tilde A$ (as defined above \eqref{eq:SGnXM}) satisfies the assumptions of Lemma \ref{l:mgf}. Then there exists $\delta>0$ such that $$\sup_{|t|\le \delta} \sup_{n\ge 1}\Ex\left[e^{tS_{G_n}(\bm{X}_n)}\right] < \infty.$$  Thus \ref{bcr} implies \ref{eq:4mr} by the boundedness of all moments. Also, since $\vr[X_1^2]=0$ when $X_1$ is Rademacher, \ref{eq:4mr} implies \ref{bar} follows from \eqref{eq:nneg0}. Therefore, it remains to show \ref{bar} implies \ref{bcr}. To this end, for $1 \leq u \ne v \leq n$, define $d_{u, v} := \sum_{w=1}^n a_{u, w}a_{w, v}$ as the number of common neighbors of the vertices $u$ and $v$ (the co-degree of $u$ and $v$). Note that $$N(C_4, G_n) \asymp  \sum_{1 \leq u \ne v \leq n} \binom{d_{u, v}}{2}.$$ Therefore, by \ref{bar}, 
	$$ \frac{1}{|E(G_n)|^2} \max_{1 \leq u \ne v \leq n} d_{u, v}^2 \lesssim 
	\frac{1}{|E(G_n)|^2} \sum_{1 \leq u \ne v \leq n} \binom{d_{u, v}}{2}	\lesssim \frac{N(C_4,G_n)}{\eg^2} \rightarrow 0 .$$
	Referring to \eqref{as:point}, this implies, $\sigma_{st}=0$ for $s\neq t$. Moreover, passing to a subsequence we can assume $d_s/\eg \to \sigma_{ss}$ exists, for all $s \geq 1$. (Note that $\sigma_{ss}$ may not be zero.) 
	
	Next, we will show that $\rho_s=0$, for all $s \geq 1$. For this recall the definition of the graph $G_{n, K}$ and $\{\lambda_{n, K}^{(s)}\}_{1 \leq s \leq n-K}$ from Assumption \ref{asum2}. Using the bound in \eqref{eq:Kdegree} gives, 
	$$N(K_{1,2},G_{n, K}) \lesssim \sum_{u=K+1}^n d_u^2 \le \frac1K\eg^2.$$ Thus, as $n\to \infty$ followed by $K\to\infty$,  
	\begin{align*}
		\frac1{\eg^2}\sum_{s=1}^{n-K} (\lambda_{n, K}^{(s)})^4 &  = \frac1{\eg^2}(\tr(A^{(K)})^4)  \\ 
		& \lesssim \frac1{\eg^2}[N(K_2,G_{n, K})+N(K_{1,2},G_{n, K})+N(C_4,G_{n, K})] \\ 
		& \le \frac1{\eg}+\frac1{K}+\frac{N(C_4,G_n)}{\eg^2} \to 0, 
	\end{align*} 
	where the last step uses \ref{bar}. This implies, 
	\begin{align*}
		\left(\max_{1 \leq s \leq n-K} \frac{1}{\sqrt{\eg}}\lambda_{n, K}^{(s)}\right)^4 \le \frac1{\eg^2}\sum_{s=1}^{n - K} (\lambda_{n, K}^{(s)}){^4} \to 0. 
	\end{align*}
	Hence $\rho_s=0$, for all $s \geq 1$. Hence, $Q_1, Q_2,$ and $Q_3$ in Theorem \ref{thm:main} are $N(0,\sum_{s=1}^{\infty} \sigma_{s s})$, $N(0,1-\sum_{s=1}^{\infty} \sigma_{ss})$, and zero, respectively. This establishes \ref{bcr}. 
	
	\section{Universality} 
	\label{sec:universality}

	In this section we discuss conditions under which the limiting distribution of $S_{G_n}({\bm X}_n)$ is `universal', that is, it does not depend on the marginal law of ${\bm X}$. Universality of random homogeneous sums has been extensively studied in the literature (see, for example, \cite{chatterjee_invariance,stability,weiner_chaos,rotar2} and {in the context of directed polymers \cite{akq,carcot,polychaos,marrel,csz2d}}). In particular, \cite[Theorem 4.1]{weiner_chaos} shows that a random homogeneous sum and its corresponding Gaussian counterpart are asymptotically close in law, when the maximum `influence' of the underlying independent random variables are controlled, where the influence of a variable roughly quantifies its contribution to  the overall configuration of the homogeneous sum (see \cite[Equation (1.5)]{weiner_chaos} for the formal definition). For random quadratic forms as in $S_{G_n}({\bm X}_n)$, the maximum influence turns out to be the maximum degree of the graph $G_n$ scaled by $|E(G_n)|$. Consequently, the aforementioned results imply that the limiting distribution of $S_{G_n}({\bm X}_n)$ is universal whenever the maximum average degree of $G_n$ is $o(|E(G_n)|)$. In the following proposition we recover this result as a corollary of Theorem \ref{thm:main} and also show that this condition is tight, in the sense that universality does not hold when the maximum degree is of the same order as $|E(G_n)|$. 
	
	\begin{corollary} 
		Suppose $\bm{X}_n=(X_1,X_2,\ldots,X_n)^\top$ be i.i.d.~mean $0$ and variance $1$ random variables with common distribution function $F$ and consider a sequence of graphs $\{G_n\}_{n \geq 1}$ with $G_n \in \mathscr{G}_n$  where $\mathscr{G}_n$ is as in Definition \ref{defg}.  {Consider the metric
		\begin{align*}
			d_L(X,Y):=\sup_{h \ : \ |h|+|h'|+|h''|+|h'''|\le L} \left|\Ex[h(X)]-\Ex[h(Y)]\right|
		\end{align*}}
		Then the following hold: 
		
		\begin{enumerate}
			\item[$(1)$]
			If $\max_{1 \leq u \leq n}  \frac{d_u}{|E(G_n)|}\to 0$, then for all $L>0$,
			\begin{align*}
				\lim_{n \rightarrow \infty} d_L(S_{G_n}({\bm X}_n),S_{G_n}({\bm Z}_n)) = 0,
			\end{align*}
			where ${\bm Z}_n=(Z_1,\ldots,Z_n)$ is a vector of i.i.d. $N(0,1)$ random variables. 
			
			\item[$(2)$]
			Conversely, if $\liminf_{n\to\infty}\max_{1 \leq u \leq n}\frac{d_u}{|E(G_n)|}>0$, then for any random variable $X_1$ satisfying $\Ex |X_1|^3<\infty$ and $\Ex[X_1^4]=\infty$, 
		{	\begin{align*}
				\liminf_{n\to\infty}d_L(S_{G_n}({\bm X}_n),S_{G_n}({\bm  Z}_n)) >0,
			\end{align*}
		for some $L>0$.}
		\end{enumerate} 
	\end{corollary}
	
	\begin{proof} {Fix $\e>0$ and any $h$ with $|h|+|h'|+|h''|+|h'''|\le L$. By first-order Taylor expansion
	\begin{align*}
		\left|\Ex[h(S_{G_n}({\bm  X}_n))]-\Ex[h(S_{G_n}({\bm  X}_{n,M}))] \right|  \lesssim_L \Ex\left|S_{G_n}({\bm  X}_n)-S_{G_n}({\bm  X}_{n,M})\right| 
	\end{align*}
which by Lemma \ref{l:trunc} goes to zero as $n\to \infty$ followed by $M\to \infty$. Thus one can choose $M$ large enough so that 
			\begin{align*}
		\limsup_{n\to\infty}\left|\Ex[h(S_{G_n}({\bm  X}_n))]-\Ex[h(S_{G_n}({\bm  X}_{n,M}))] \right| \le \e.
		\end{align*}	
		Using part (a) of Proposition \ref{p:norm} we have
		\[|\Ex[h(S_{G_n}({\bm  X}_{n,M})]-\Ex[h(S_{G_n}({\bm  Z}_n))] |\lesssim_{M,L} \frac{\max_{1 \leq u \leq n}d_u}{|E(G_n)|}\to 0, \] 
		since $\frac{\max_{1 \leq u \leq n} d_u}{|E(G_n)|}\to 0$ by assumption. Thus by triangle inequality we have that 
		\begin{align*}
			\limsup_{n\to\infty}\left|\Ex[h(S_{G_n}({\bm  X}_n))]-\Ex[h(S_{G_n}({\bm  Z}_{n}))] \right| \le \e.
		\end{align*}
		Since the above two estimates are uniform over all possible $h$ with $|h|+|h'|+|h''|+|h'''|$ less than $L$, and as $\e$ is arbitrary,  we arrive at (1). }
		
		To show (2) note that by Theorem \ref{thm:moment4} any subsequential limit of $S_{G_n}({\bm  X}_n)$ and $S_{G_n}({\bm  Z}_n)$ is of the form $Q_1+Q_2+Q_3$, where $Q_1,Q_2,Q_3$ are mutually independent and as described in Theorem \ref{thm:main}. It follows from their definitions that $Q_2$ and $Q_3$ have finite mgfs in a neighborhood of $0$. Also, if $F=N(0,1)$, then $Q_1$ is an infinite sum of chi-squares (see Lemma \ref{lem:normalqf}), and hence, has a finite mgf in a neighborhood of $0$. This implies, any subsequential limit of $S_{G_n}({\bm  Z}_n)$ will have a finite mgf in a neighborhood of $0$. Therefore, to establish (2) it suffices to show if $\Ex|X_1|^3<\infty$ and $\Ex[X_1^4]=\infty$, then $Q_1$ does not have finite exponential moment in a neighborhood of zero. In fact, we will show the  stronger conclusion that $\Ex[Q_1^4]=\infty$. For this, setting $\bm X_{-1}=(X_2,X_3,\ldots)$ note that as in \eqref{eq:cbpf}, 
		\begin{align*}
			\bm X_{\infty}^\top \Sigma \bm X_{\infty}=\sigma_{11}X_1^2+ X_1f(\bm{X}_{-1})+g(\bm{X}_{-1}),
		\end{align*}
		for some functions $f$ and $g$. This gives, since $Q_1 \sim N(0, \bm X_{\infty}^\top \Sigma \bm X_{\infty})$, 
		\begin{align*}
			\Ex[Q_1^4]=3\Ex[({\bm  X}_\infty^\top \Sigma{\bm  X}_\infty)^2]
			=3\Ex\left[(\sigma_{11}X_1^2+ X_1f(\bm{X}_{-1})+g(\bm{X}_{-1}))^2\right] =  +\infty , 
		\end{align*}
		where the last step follows from the observation    
		\[\Ex \left[(\sigma_{11}X_1^2+ X_1f(\bm{X}_{-1})+g(\bm{X}_{-1}))^2| \{X_u, u\ge 2 \} \right]\stackrel{a.s.}{=} + \infty.\]
		This completes the proof of part (2).
	\end{proof}

\section{Conclusion and Future Directions} 
\label{sec:open}

\rev{ In this paper, we provide a complete characterization of all possible distributional limits of random quadratic forms with $\{0,1\}$-valued coefficients. We also provide necessary and sufficient conditions for the asymptotic normality of $\{0,1\}$-valued random quadratic forms, connecting it to the fourth-moment phenomenon. The natural next step is to consider quadratic forms with \textit{arbitrary} weights. We expect our arguments to hold when all the non-zero coefficients lie within an interval $[C^{-1},C]$, for some $C\ge 1$. However, for more general weights the combinatorial arguments in Section \ref{sec:independence} do not apply, and new ideas are needed. Establishing rates of convergence for our limit theorems is another interesting future direction. }

\rev{Going beyond quadratic forms, an important question is to understand all possible distributional limits for homogeneous sums of order $r \ge 3$. Considering $\{0, 1\}$-valued coefficients, it would be interesting to explore if extremal results for hypergraphs can be applied 
to characterize distributional limits in this setting.} \




\subsection*{Data availability statement} Data sharing not applicable to this article as no datasets were generated or analyzed during the current study.

\subsection*{Funding} Bhaswar B. Bhattacharya was partially supported by NSF CAREER grant DMS 2046393 and a Sloan Research Fellowship. Sayan Das was partially supported by supported by NSF grant DMS-1928930 and Fernholz Foundation's ``Summer Minerva Fellows'' program. Sumit Mukherjee was partially supported by NSF grants DMS 1712037 and DMS 2113414. The authors also thank the anonymous referees for their helpful comments. 

\appendix 

	\section{Proofs of Technical Lemmas}
\label{sec:appendix}

In this section we collect the proofs of various technical lemmas. We begin with a lemma about the asymptotic independence of two random variables which are conditionally jointly Gaussian.

\begin{lemma}\label{tech} Let $\{(U_n,V_n)\}_{n \geq 1}$ be a sequence of  bivariate random vectors and $\{\mathcal{F}_n\}_{n \geq 1}$ be a sequence of $\sigma$-fields. Assume the following conditions:
	\begin{enumerate} 
		\item[$(1)$] $(U_n,V_n) \mid \mathcal{F}_n \sim N_2({\bm \mu}_n,\Sigma_n)$,
		where the random variables 
		$$\bm \mu_n := \begin{pmatrix}   \mu_n(1) \\  \mu_n(2) 
		\end{pmatrix} \in \R^2
		\quad \text{and} \quad 
		\Sigma_n := \begin{pmatrix}   \Sigma_n(1,1) &   \Sigma_n(1, 2) \\  \Sigma_n(1, 2) &  \Sigma_n(2,2) 
		\end{pmatrix} \in \R^{2 \times 2}, 
		$$ 
	are tight in $\R^2$ and $\R^{2 \times 2}$, respectively.

	\item[$(2)$] $\mathrm{Cov}(U_n, V_n|\mathcal{F}_n)\stackrel{P}{\to}0.$
	
	\item[$(3)$] $\{  \mu_n(1),\Sigma_n(1,1)\}$ and $\{\mu_n(2),\Sigma_n(2,2)\}$ are mutually independent. 
	
\end{enumerate} 
Then $U_{n}$ and $V_{n}$ are asymptotically independent, that is, for any $t_1,t_2\in \R$ 
$$\lim_{n\to\infty}\Big| \Ex\left[e^{\i (t_1U_n+t_2V_n)}\right]-\Ex\left[e^{\i t_1U_n}\right] \Ex\left[ e^{\i t_2V_n}\right]\Big|=0.$$ 
\end{lemma}

\begin{proof} 
Without loss of generality by passing to a subsequence, assume that $({\bm \mu}_n,\Sigma_n)\stackrel{D}{\to}({\bm \mu},\Sigma)$, where $(\bm \mu,\Sigma)$ is a random variable on $\R^2\times \R^{2\times 2}$. Then using the first assumption and setting ${\bm t}:=(t_1,t_2)$ gives 
\begin{align*}
	\Ex\left[e^{\i (t_1U_n+t_2V_n)}\right]=\Ex\left[ e^{\i {\bm t}^\top{\bm \mu}_n-\frac{1}{2}{\bm t}^\top \Sigma_n {\bm t}} \right]
	\to\Ex\left[ e^{\i {\bm t}^\top {\bm \mu}-\frac{1}{2}{\bm t}^\top \Sigma {\bm t}} \right].
\end{align*}
Using the second assumption now gives $\Sigma_n(1,2)=\mathrm{Cov}(U_n,V_n|\mathcal{F}_n)\stackrel{P}{\to}0$, that is, $\Sigma(1,2)=0$ almost surely. This implies, 
\begin{align}\label{eq:t1t2}
	\Ex\left[e^{\i {\bm t}^\top{\bm \mu}-\frac{1}{2}{\bm t}^\top \Sigma {\bm t}} \right] & = \Ex\left[ e^{\i t_1 \mu(1)-\frac{1}{2}t_1^2 \Sigma(1,1)} \right] \Ex\left[ e^{\i t_2 \mu(2)-\frac{1}{2} t_2^2 \Sigma(2, 2)} \right] , 
\end{align}
since $\{  \mu_n(1),\Sigma_n(1,1)\}$ and $\{\mu_n(2),\Sigma_n(2,2)\}$ are mutually independent by the third assumption. Since the RHS of \eqref{eq:t1t2} factorizes in $t_1,t_2$, it follows that $(U_n,V_n)$ are asymptotically independent, as desired. 
\end{proof}

The next lemma shows the finiteness of the moment generating function of 
a bilinear/quadratic function of sub-Gaussian random variables.

\begin{lemma}\label{l:mgf}	 Let $C_n$ be an $n\times n$ matrix (not necessarily symmetric) which satisfies 
\begin{align}\label{mgf-cond}
	\beta:=\sup_{n \ge 1}\|C_n\|_F <\infty,
\end{align}
where $\| \cdot \|_F$ denotes the Frobenius norm.
Suppose ${\bm Y}:=(Y_1,Y_2,\ldots,Y_n)^\top$ and ${\bm Z}:=(Z_1,Z_2, \ldots, Z_n)^\top$ are mutually independent mean $0$ random variables which are uniformly sub-Gaussian, that is,~there exists $\gamma>0$, such that for all $t\in \R$  
\begin{align}\label{eq:mgf_assumption}
	\sup_{n\ge 1} \max_{1\le u \le n} \max\Big\{\Ex\left[e^{tY_u }\right], \Ex\left[e^{t{Z}_u}\right]\Big\} \le e^{\gamma^2t^2/2}.    \end{align}
\begin{enumerate}
	\item[$(1)$] 
	Then there exists $\delta>0$ (depending only on $\beta,\gamma)$, such that
	\begin{align*}
		\sup_{|t|\le \delta}\sup_{n\ge 1} \Ex\left[e^{t{\bm Y}^\top C_n{\bm Z}}\right] <\infty.
	\end{align*}
	
	\item[$(2)$]
	If, furthermore, 
	\begin{align}\label{mgf-cond2}
		\alpha:=\sup_{n\ge 1}|{\rm tr}(C_n)| <\infty,
	\end{align}
	then there exists $\delta>0$ (depending only on $\alpha,\beta,\gamma)$, such that 
	\begin{align*}
		\sup_{|t|\le \delta}\sup_{n\ge 1} \Ex\left[e^{t{\bm Y}^\top C_n{\bm Y}}\right] <\infty.
	\end{align*}
\end{enumerate}
\end{lemma}

\begin{proof} Using the mutual independence and the sub-Gaussian assumption \eqref{eq:mgf_assumption}, 
\begin{align*}
	\Ex\left[e^{t {\bm Y}^\top C_n{\bm Z} } \right] = \Ex\left[\Ex\left[e^{t {\bm Y}^\top C_n{\bm Z} }\Big | \bm Z \right] \right]   \le \Ex \left[e^{\frac{t^2\gamma^2}{2} {\bm Z}^\top C_n^\top C_n{\bm Z}}\right].
\end{align*}
Now, use \eqref{mgf-cond} and \eqref{eq:mgf_assumption} to note that
$$\Ex[{\bm Z}^\top  C_n^\top C_n{\bm Z}]\le K \gamma^2{\rm tr}(C_n^\top C_n) \le K\beta^2 \gamma^2,$$
for some universal constant $K$. To complete the proof, invoking the Hanson-Wright's inequality for quadratic forms (\cite[Theorem 1.1]{rud}), it suffices to show that 
$$\sup_{n\ge 1}\|C_n^\top C_n\|_F<\infty \quad \text{and} \quad \sup_{n\ge 1}\|C_n^\top C_n\|_2<\infty,$$
where $\|\cdot\|_2$ denotes the operator norm. This follows from \eqref{mgf-cond} and noting that 
$$\|C_n^\top C_n\|_F\le \|C_n\|_F^2\le \beta^2 \quad \text{and} \quad \|C_n^\top C_n\|_2\le \|C_n^\top C_n\|_F\le \beta^2.$$ 
This completes the proof of (1). 

For (2) by the Hanson-Wright's inequality and the sub-Gaussian assumption \eqref{eq:mgf_assumption}, it suffices to check the following conditions:
$$\sup_{n\ge 1}\Ex[{\bm Y}^\top C_n{\bm Y}]<\infty, \quad \sup_{n\ge 1}\|C_n\|_F<\infty, \quad \text{and} \quad \sup_{n\ge 1}\|C_n\|_2<\infty. $$
Note that the second and third bounds are immediate from \eqref{mgf-cond}. For the first bound, by the sub-Gaussianity assumption there is some universal constant $K_1>0$ such that 
\begin{align*}
	| \Ex[{\bm Y}^\top C_n{\bm Y}]|\le K_1\gamma^2|{\rm tr}(C_n)|\le K_1 \gamma^2 \alpha , 
\end{align*}
where the last step uses \eqref{mgf-cond2}. This completes the proof of part (2). 
\end{proof}

The next lemma shows that the normal variance mixture $Q_1$ in \eqref{eq:main} when $F= N(0, 1)$, is a weighted sum of centered chi-squared random variables. Theorem \ref{thm:main} then implies that the limiting distribution of $S_{G_n}$ when $F= N(0, 1)$ is a sum of a Gaussian random variable and an independent weighted sum of centered chi-squared random variables (recall Remark \ref{remark:normalqf}).

\begin{lemma}\label{lem:normalqf} Suppose $\{X_s\}_{ s \geq 1 }$ be i.i.d.~ $N(0, 1)$and let $\Sigma=((\sigma_{st}))_{s,t\ge 1}$ be an infinite dimensional matrix which satisfies the following  conditions:
\begin{itemize}
	\item For every $n\ge 1$ the $n\times n$ leading principle sub-matrix $\Sigma_n:=((\sigma_{st}))_{1\le s,t\le n}$ is positive semi-definite.
	
	\item $\sum_{s=1}^\infty\sigma_{ss}<\infty$ and $\sum_{s,t=1}^\infty \sigma_{st}^2<\infty$. 
\end{itemize}
Then the following conclusions hold:
\begin{itemize} 
	
	\item[$(1)$] With $\bm{X}_n:=(X_1,\ldots,X_n)$, the sequence of random variables $\bm{X}_{n}^\top \Sigma_n\bm{X}_n$ converges in $L^2$ to  some random variable, which is denoted by  $\bm{X}_{\infty}^\top \Sigma \bm{X}_\infty$.

	\item[$(2)$] 
	The random variable $Q_1\sim N(0,\bm{X}_{\infty}^{\top}\Sigma \bm{X}_{\infty})$ has the same distribution as $\sum_{s=1}^{\infty} \eta_s Y_s$, where $\{Y_s\}_{s\ge 1}$ are i.i.d.~$\chi_1^2-1$ random variables and for some sequence $\{\eta_s\}_{s\ge 1}$ which is square summable.

\end{itemize}
\end{lemma}

\begin{proof} 

	For any $m,n\ge 1$, 
	\begin{align*}
		\Ex[{\bm X}_{m+n}^\top \Sigma_{m+n}{\bm X}_{m+n}-{\bm X}_n^\top \Sigma_n {\bm X}_n]^2\lesssim & \Big(\sum_{s=n+1}^{n+m} \sigma_{ss}\Big)^2+\sum_{s=n+1}^{n+m} \sum_{t=1}^{m+n}\sigma_{st}^2 \\
		\le &\Big(\sum_{s=n+1}^\infty \sigma_{ss}\Big)^2+\sum_{s=n+1}^\infty\sum_{t=1}^\infty \sigma_{st}^2,
	\end{align*}
	which converges to $0$ as $n\to\infty$, under the assumptions on $((\sigma_{st}))_{s,t\ge 1}$. Thus the sequence of random variables $\bm{X}_n^\top \Sigma_n \bm{X}_n$ is Cauchy in $L^2$, and hence,  converges to a random variable $\bm{X}_\infty^\top\Sigma\bm{X}_\infty$, proving (1).

	Now, let $\lambda_{1,n}\ge \lambda_{2,n} \ge \cdots \ge \lambda_{n,n} \ge 0$ be the eigenvalues of $\Sigma_n$. 
	Due to the interlacing of the eigenvalues it follows that for each $s\ge 1$, the sequence $\{\lambda_{s,n}\}_{n \geq 1}$ is increasing in $n$. Hence, $\lambda_s:=\lim_{n \rightarrow \infty} \lambda_{s,n}$ exists, for all $s \geq 1$. Since $\sum_{s=1}^n \lambda_{s,n}= \tr(\Sigma_n) = \sum_{s=1}^n \sigma_{ss}\le \sum_{s=1}^\infty \sigma_{ss}$, 
	taking limits along with Fatou's lemma gives
	\begin{align*}
		\sum_{s=1}^\infty \lambda_s\le \sum_{s=1}^\infty \sigma_{ss}<\infty. 
	\end{align*}
	Moreover, by the Monotone Convergence Theorem, 
	$$\sum_{s=1}^n\lambda_{s,n}\to \sum_{s=1}^\infty \lambda_s.$$
	Combining the last two displays together with Scheffe's Lemma gives, 
	\begin{align}\label{eq:Fatou}\sum_{s=1}^n|\lambda_{s,n}-\lambda_s|\to 0.
	\end{align}
	Now, by the spectral decomposition $\Sigma_n = P_n^{\top }\Lambda_n P_n$ and noting that $P_n{\bm X}_n\stackrel{D}{=}{\bm X}_n$ (since $P_n$ is an orthogonal matrix) it follows that 
	$${\bm X}_n^\top \Sigma_n{\bm X}_n={\bm X}_n^\top P_n^\top \Lambda_n P_n{\bm X}_n\stackrel{D}{=}{\bm X}_n^\top  \Lambda_n{\bm X}_n=\sum_{s=1}^n\lambda_{s,n}X_s^2 \stackrel{L^1}{\to} \sum_{s=1}^\infty \lambda_s X_s^2, $$
	where the last line uses \eqref{eq:Fatou}.
	The last display, along with part (1) gives
	$${\bm X}_\infty^\top\Sigma{\bm X}_\infty\stackrel{D}{=}\sum_{s=1}^\infty \lambda_s X_s^2. $$
	{With this representation we can compute the characteristic function of $Q_1$ as follows:
		\begin{align}\label{eq:q1normalcf}
			\Ex[e^{\i t Q_1}]  = \Ex\left[e^{-\frac{t^2}{2} \sum_{s=1}^\infty
				\lambda_s X_s^2 } \right ] = \prod_{s=1}^\infty \Ex\left[e^{- \frac{t^2}{2}
				\lambda_s X_s^2 } \right]= \prod_{s=1}^\infty (1+  t^2\lambda_s)^{-\frac{1}{2}}.   \end{align}
		Now, observe that if $Y_1,Y_2$ are independent $\chi_1^2-1$ distributed random variables, the characteristic function of $\frac{\sqrt{\lambda_s}}{2} (Y_1-Y_2)$ is given by 
		$$ \Ex\left[e^{\i t \frac{\sqrt{\lambda_s}}{2} (Y_1-Y_2)}\right] = (1-\sqrt{\lambda_s} \i t)^{-\frac{1}{2}}(1+\sqrt{ \lambda_s} \i t)^{-\frac{1}{2}}=(1+t^2\lambda_s)^{-\frac{1}{2}}.$$ 
		This together with \eqref{eq:q1normalcf} shows that
		$$Q_1\stackrel{D}{=}\sum_{s=1}^\infty \frac{\sqrt{\lambda_s}}{2} (Y_{1, s}-Y_{2, s}),$$
		where $\{Y_{a, s}\}_{1\le a\le 2, s\ge 1}$ is a collection of i.i.d. $\chi_1^2-1$ random variables. This completes the proof of (2). } 
\end{proof}


\end{document}